\documentclass{article}
\usepackage[utf8]{inputenc}
\usepackage{amsmath}
\usepackage{amssymb}
\usepackage{amsfonts}
\usepackage{ bbold }
\usepackage{amsthm}
\usepackage{comment}
\usepackage{stmaryrd}
\usepackage{float}
\usepackage{url}
\usepackage[table]{xcolor}
\newcommand{\N}{\mathbb{N}}
\newcommand{\Z}{\mathbb{Z}}
\newcommand{\Q}{\mathbb{Q}}
\newcommand{\R}{\mathbb{R}}
\newcommand{\C}{\mathbb{C}}
\newcommand{\Sh}{S}
\newcommand{\Shs}{S^*}
\DeclareMathOperator{\Flm}{\mathcal{F}_{1-k,N,D,D_0}}
\DeclareMathOperator{\Plz}{\mathcal{P}_{k,N,D,D_0}}

\newcommand{\im}{\text{Im}}
\newcommand{\re}{\text{Re}}
\usepackage{bbm}
\newtheorem{theorem}{Theorem}
\newtheorem{lemma}[theorem]{Lemma}
\newtheorem{corollary}[theorem]{Corollary}
\newtheorem{proposition}[theorem]{Proposition}

\newtheorem{remark}[theorem]{Remark}
\newtheorem{definition}[theorem]{Definition}

\newtheorem{example}[theorem]{Example}

\newenvironment{proof-sketch}{%
  \proof}{\endproof}

\usepackage{graphicx} 
\title{Central Values of $L$-Functions of Twisted Modular Forms and Local Polynomials}
\author{Charlotte Dombrowsky}
\date{}

\begin{document}

\maketitle
\begin{abstract}
In this paper we study the product of two  central values of $L$-functions of a twisted modular. We show that it suffices to compute a local polynomial at a finite number of points to decide whether the product is zero. For the proof, we relate the local polynomial to the product of the $L$-functions using a locally harmonic Maass form and building on the Shimura-Shintani correspondence. This extends results from Ehlen, Guerzhoy, Kane and Rolen as well as Males, Mono, Rolen and Wagner. 
\end{abstract}
\section*{Introduction and statement of results}
A deep result in number theory is the relation between central values of L-functions of twisted modular forms of integer weight and coefficients of specific modular forms of half-integral weight. This connection was observed first by Waldspurger \cite{Waldspurger} and further developed by Zagier and Kohnen \cite{Kohnen-Zagier},\cite{Kohnen1985}.  Kohnen's Theorem was generalized by Sakata \cite{Sakata-4p^m}, \cite{Sakata-N} who removed the condition that the level had to be square-free.

From the Birch and Swinnerton-Dyer Conjecture, it follows that studying central values of twisted $L$-functions gives insight into the rank of elliptic curves over quadratic extensions. This observation was used by Tunnell \cite{Tunnell} to find a criterion, when an integer is a congruent number, i.e. the area of a right triangle with sides given by rational numbers. This problem is equivalent to finding non-torsion points on a family of quadratic twists of elliptic curves. Tunnell showed that an integer is a congruent number if and only if certain binary quadratic forms have the same number of solutions over integers. In this approach, one only considers modular forms of weight $2$, however the Bloch-Kato Conjecture generalizes the Birch and Swinnerton-Dyer Conjecture for higher weight modular forms.

While Kohnen and Sakata provide beautiful formulas to compute the central values given an integer weight form, it can be laborious to find the corresponding half-integral weight form. The MAGMA code from Purkait \cite{CodePurkait} computes the Shimura decomposition of a space of half-integral weight forms and thus helps to determine the corresponding function of half-integral weight. However, even this algorithm takes (in particular for high levels) a long time to finish. Our main theorem, which is build upon the work of Zagier \cite{Kohnen-Zagier}, Kohnen \cite{Kohnen1985} and Sakata \cite{Sakata-N}, leads to an effective algorithm to determine when the central values of $L$-functions of twisted modular forms vanish. 

The statement is a generalization of a theorem from Ehlen, Guerzhoy, Kane and Rolen. In \cite{Ehlen_2020} they constructed a locally harmonic Maass form to study $L$-functions of cusp forms of weight $2$. Using this function, they relate the vanishing of the product of two twisted $L$-functions evaluated at the central value to a local polynomial depending on some binary quadratic forms. Under the assumption that the space of cusp forms of level $N$ and weight $2$ is one dimensional, they show that the product vanishes if and only if the local polynomial evaluated at two given rational numbers is the same. This result was generalized by Kong in his PhD thesis \cite{Kong-2017}, in which he removes the condition on the dimension by using Hecke operators.

Instead of explicitly computing half-integral weight forms, we will show that due to Kohnen's  and Sakata's work one can relate the product of two central values of twisted $L$-functions to the behavior of a locally harmonic Maass form. Locally harmonic Maass forms are generalizations of harmonic Maass forms. They are modular of a given weight $k$, are annihilated by the weight $k$ hyperbolic Laplacian and display at most polynomial growth towards infinity. However, they differ from harmonic Maass forms by allowing for singularities on a certain exceptional set within the upper half-plane. Locally harmonic Maass forms were first introduced by Bringmann, Kane, and Kohnen in \cite{B-K-K}, and further explored e.g. by L{\"o}hbrich and Schwagenscheidt \cite{löbrich-schwagenscheidt} and Mono \cite{Mono-lhmF-even-weight}.

We will define a locally Harmonic Maass form which is essentially a twist of the form defined in \cite{B-K-K}. This form splits into three parts: a holomorphic part, a non-holomorphic part and a local polynomial which captures the jumps of the locally harmonic Maass form along the exceptional set. Our theorem essentially says that the product of the central values of two twisted $L$-functions is zero if and only if (a projection of) the locally harmonic Maass form is equal to (a projection of) the local polynomial, in other words the holomorphic and non-holomorphic parts vanish.

\smallskip

This project was inspired by the preprint \cite{males-mono-rollen-wagner} from Males, Mono, Rolen and Wagner. The authors prove a similar statement to the one below but assume that the level is square-free. We aimed to extend their result for general composite level. However, their proof has a mistake, relying on a misconception of the behavior of $\Plz$. We provide a new proof which gives a criterion in more generality.
\\
In our argumentation we use objects from \cite{males-mono-rollen-wagner},  often renormalizing them to make some technical steps work smoothly. Our statement adjusts their claim and extends it to general composite level. Additionally we remove the condition that both discriminant $D$ and $D_0$ need to be a square modulo $4N$.

Our theorem relates the product of $L$-functions to a local polynomial defined as follows:
\begin{align}\label{def: Plx}
    P_{k,N,D,D_0}(x):=
    c_{k,N,D,D_0,\infty}
    +c_{k,DD_0}\sum_{\substack{Q=[a,b,c] \in \mathcal{Q}_{N,DD_0}\\ a<0<Q(x,1)}} \chi_{D_0}(Q)Q(x,1)^{k-1}.
\end{align}
Here $\mathcal{Q}_{N,DD_0}$ is the set of binary quadratic forms of discriminant $DD_0$ such that $N$ divides the leading coefficient:
\begin{align*}
    \mathcal{Q}_{N,DD_0}:=\{ Q=[a,b,c] : b^2-4ac=DD_0, \quad N|a\}.
\end{align*}
Moreover  $c_{k,N,D,D_0,\infty}$ and $c_{k,DD_0}$ are explicit constants depending on the weight $k$, the level $N$ and two discriminants $D$ and $D_0$ (see equation \eqref{eq: def c_infty c_2}). 
Given a newform $F \in S_{2k}(N)$,  let $\prod_i(T_{p_i}-a_{p_i})$ be a product of Hecke operators such that
    \begin{align*}
        \mathbb{T}:=\prod_i(T_{p_i}-a_{p_i}): S_{2k}(N) \rightarrow \text{span}_{\C}\{F\}
    \end{align*}
    is not the zero map. Denote by $\tilde{\mathbb{T}}$ the lift of $\mathbb{T}$ given by
    \begin{align*}  \Tilde{\mathbb{T}}:=\prod_i(T_{p_i}-p_i^{1-2k}a_{p_i}).
    \end{align*}
The main theorem of this paper is the following:

\begin{theorem}\label{theo: main prod L-functions}
Let $k>1$ be an integer. Let $N$ be an odd integer.
Let $D,D_0$ be two fundamental discriminants such that $(D,N)=(D_0,N)=1$, $DD_0$ is not a square, $D(-1)^k, D_0(-1)^k>0$ and for all $p|N$, we have $ \left( \frac{D}{p}\right)=\left( \frac{D_0}{p}\right)$. Let $F \in S_{2k}(N)$ be a newform. 
Denote by $\tau: \{ p\text{ is a prime with } p |N\} \rightarrow \{ \pm 1\}$ the system of Atkin-Lehner eigenvalues of $F$. We assume that 
\begin{align*}
    \tau(p)=\begin{cases}
        1, & \text{if } \text{ord}_p(N) \text{ is even, }\\
        \left( \frac{D}{p}\right)=\left( \frac{D_0}{p}\right), & \text{if } \text{ord}_p(N) \text{ is odd.}
    \end{cases}
\end{align*}
The following are equivalent
    \begin{align*}
        &\text{(i)} \quad L(F\otimes \chi_D,k) L(F\otimes \chi_{D_0},k)=0 \\
        &\text{(ii)} \quad  (P_{k,N,D,D_0}|_{2-2k}\tilde{\mathbb{T}}|_{2-2k} \gamma  )(x)=(P_{k,N,D,D_0}|_{2-2k}\tilde{\mathbb{T}}) (x), \quad \forall x \in \Q, \gamma \in \Gamma_0(N)\\
        &\text{(iii)} \quad P_{k,N,D,D_0}|_{2-2k}\tilde{\mathbb{T}}(x)=0, \quad x \sim_{\Gamma_0(N)}0.
        \end{align*}
\end{theorem}
\begin{remark}
    The conditions on the eigenvalues of the Atkin-Lehner involutions reduces to:
\begin{itemize}
    \item For $N$ square-free:
    \begin{align*}
        \forall p|N: \quad \tau(p)=\left(\frac{D}{p}\right). 
    \end{align*}
    \item For $N=p^m$ a power of a prime with $m>2$:
    \begin{align*}
        \tau(p)=\begin{cases}
            1, & \text{if } m \text{ even},\\
            \left( \frac{D}{p}\right), &\text{if } m \text{ odd}.
        \end{cases}
    \end{align*}
\end{itemize}
\end{remark}

Note that since $F$ is a newform and thus in particular an eigenform for the Hecke operators, we can always choose a product of Hecke operators as in the theorem. This product is not unique. We elaborate on how to choose the Hecke operators in Section \ref{sec: 3-Hecke Operators}. From the proof it will follow that we only need to compute $P_{k,N,D,D_0}|_{2-2k}\tilde{\mathbb{T}}(x)$ for a small number of rationals $x \sim_{\Gamma_0(N)}0$ to check if it is zero everywhere.

Our theorem simplifies the computation of central values of $L$-functions: finding the modular form of half-integral weight whose coefficients encode the central values of the $L$-function as described by Kohnen and Sakata can demand a lot of effort. However, finding suitable Hecke operators takes little effort and computing $P_{k,N,D,D_0}|_{2-2k}\tilde{\mathbb{T}}(x)$ takes only a short amount of time. This makes this criterion easy to apply. This can be seen in the following example (for more details see \ref{subsec: Ex. S_4(25)}):
\begin{example}
Let $F:=q + q^2 + 7q^3 - 7  q^4 + 7  q^6 + 6  q^7 - 15  q^8 + 22  q^9 + O(q^{11}) \in S_4(25)$. Then $F|W_{25}=F$. Moreover, one can compute that 
\begin{align*}
    L(F\otimes \chi_{21},2)\approx1.62649, \quad L(F\otimes \chi_{56 },2)\approx0.37350, \text{ and }L(F\otimes \chi_{69 },2)=0.
\end{align*}
We set $P(x):=P_{2,25,D,21}|_{-2}(T_7+7^{-3}\cdot 6)(T_2+ 2^{-3}\cdot 4)(x)$.
Then we have the following results about the local polynomial:
\medskip

{\scriptsize
\begin{tabular}{c|ccccccc}
  $D$& $P(1)$& $P(1/4)$ &$P(2/7)$ &$P(4/9)$& $P(9/14)$& $P(13/18)$ &$P(16/19)$\\
  \hline
56 & 0 & -25/343 & -1160/16807 & -400/3969 & -2475/33614 & -565/7938 & -800/17689 \\
69 & 0 & 0 & 0 & 0 & 0 & 0 & 0 \\
    \end{tabular}}
    
\end{example}
Our proofs relies on the observation that $L(F \otimes \chi_D,k)L(F \otimes \chi_{D_0},k)$ vanishes, if and only if the inner product $\langle F, f_{k,N,D,D_0}\rangle=0$. Here $f_{k,N,D,D_0}$ is an explicit modular from introduced by Kohnen.  This follows from directly from Kohnen \cite{Kohnen1985} in the case when the level is square-free and can be proven for composite level with the results from Sakata \cite{Sakata-N}. We then define a locally harmonic Maass form, which detects the vanishing of this inner product. Essentially we have that:
\begin{align*}
    \mathcal{F}_{1-k,N,D,D_0}= \mathcal{E}_{f_{k,N,D,D_0}}+f^*_{k,N,D,D_0}+\mathcal{P}_{k,N,D,D_0}
\end{align*}
Here $\mathcal{E}_{f_{k,N,D,D_0}}$ and $f^*_{k,N,D,D_0}$ holomorphic and non-holomorphic Eichler integral of the modular form $f_{k,N,D,D_0}$. When the dimension of the $S_{2k}(N)$ is equal to one, then $\langle F, f_{k,N,D,D_0}\rangle=0$ if and only if $f_{k,N,D,D_0} = 0$. This implies that $\mathcal{F}_{1-k,N,D,D_0}=\mathcal{P}_{k,N,D,D_0}$. Since $\mathcal{F}_{1-k,N,D,D_0}$ is modular, so is $\mathcal{P}_{k,N,D,D_0}$. One can show that this implies that it is zero for all rational numbers equivalent to $0$ modulo $\Gamma_0(N)$. When the dimension is bigger than $1$ we get an analogous
statement using Hecke operators.

\medskip
\paragraph{Outline.} This paper is structured as follows:
In a preliminary section, we recall some basics on modular forms, binary quadratic forms and the Shimura-Shintani correspondence. We also explain how one can relate the central values of the $L$-function of a modular form of integer weight to the coefficients of a specific half-integral weight form. We end the preliminary section with some comments on the modularity errors of the (non-)holomorphic Eichler integral. In section \ref{sec: A locally harmonic Maass form.} we study a specific locally harmonic Maass form. We then study the local polynomial in section \ref{sec: local polynomial}. Finally, in section \ref{sec: proof main theorem}, we prove the main theorem. We then compute some examples in section \ref{sec: comp examples}.
The appendix is dedicated to proving some of the properties of the function $\Flm$.

\subsection*{Acknowledgements}
The author would like to thank Andreas Mono for many helpful discussions. 

\numberwithin{theorem}{section}
\section{Preliminaries}
The aim of this section is to relate products of central values of $L$-functions of a newform $F$ to the innerproduct $\langle F, f_{k,N,D,D_0}\rangle$. Here $f_{k,N,D,D_0}$ is an explicit modular form. Moreover, we study modularity errors of the holomorphic and non-holomorphic Eichler integrals. 
\subsection{Modular Forms, Twists and L-functions}
In this subsection we recall some basic definitions about modular forms and fix notation. 
\subsubsection{Modular Forms of Integer Weight}
Throughout this paper, we denote with $S_{2k}(N)$ the space of cusp forms of weight $2k$ for the group $\Gamma_0(N)$ with trivial character. The space of newforms and oldforms is denoted by $S_{2k}^{new}(N)$ and $S^{old}_{2k}(N)$ respectively.

In the following, we denote by $F=\sum a_n q^n \in S_{2k}(N)$ a cusp form. We recall the action of Hecke and Atkin-Lehner operators, as well as the definition of the twists of $F$ and its $L$-function.

\begin{itemize}
    \item[(i)] Let $p$ be a prime. Then the Hecke operator $T_p$ acts on $F$ in the following way:
\begin{align*}
        F(z)|_{2k}T_p=p^{2k-1}F(pz)+ p^{-1}\sum_{j \mod p} F\left(\frac{z+j}{p}\right).
    \end{align*}
    \item[(ii)] Let $Q|N$ with $(Q,N/Q)=1$. Set
    \begin{align*}
        W_{Q}=\begin{pmatrix}
            Q \alpha & \beta \\ N \gamma &Q \delta
        \end{pmatrix} \in M_2(\mathbb{Z})
    \end{align*}
    with determinant $Q$. The Atkin-Lehner operator acting on $S_{2k}(N)$ is given by $|_{2k} W_Q$. This action is independent of the choice of matrix. In particular if $Q=N$, then the corresponding Atkin-Lehner involution is also called Fricke involution and we can choose:
    \begin{align*}
        W_N=\begin{pmatrix}
            0 & -1 \\ N & 0
        \end{pmatrix}.
    \end{align*}
   Explicitly we get that:
    \begin{align*}
        F|_{2k} W_N(z)=\text{det}(W_N)^{k}(Nz)^{-2k}F(-1/(Nz))=N^{-k}z^{-2k}F(-1/(Nz)).
    \end{align*}
\item[(iii)]    Let $D$ be a fundamental discriminant and denote with $\chi_D$ the Dirichlet character associated to it, i.e. $\chi_D(n):=\left( \frac{D}{n}\right)$, where $\left( \frac{\cdot}{\cdot}\right)$ denotes the Kronecker symbol.
Let $F=\sum_{n=1}^{\infty}a_nq^n$ be a new form. We denote the twist of $F$ with $\chi_D$ by:
\begin{align*}
    F\otimes \chi_D=\sum_{n=1}^{\infty}a_n\chi_D(n)q^n.
\end{align*}
\item[(iv)] Finally, we recall that the $L$-function of a Hecke eigenform $F= \sum_{n=1}^{\infty} a_n q^n \in S_{2k}^{new}(N)$ is given by
        \begin{align*}
            L(F,s):=\sum_{n=1}^{\infty}\frac{a_n}{n^s}
        \end{align*}
for $s$ with $\re(s)\gg0$ and that this $L$-function has an analytic continuation to the complex plane. 
    \end{itemize}

Finally, we recall the definition of the Petersson inner product:
\begin{definition}
For two function $F,G \in S_k(N)$, the Petersson inner product is given by:
\begin{align*}
    \langle F,G \rangle:=\frac{1}{[SL_2({\Z}): \Gamma_0(N)]}\int_{\Gamma_0(N) \setminus \mathbb{H}}F(z) \overline{G(z)}y^{k-2}dxdy.
\end{align*}
Here we integrate over a fundamental domain of $\Gamma_0(N) \setminus \mathbb{H}$.
\end{definition}

\subsubsection{Modular Forms of Half-Integral Weight}
We now define half-integral weight modular forms. To define and study spaces of new- and oldforms in more depth, we give the full definition here:
\begin{definition}
 Let $k$ and $N$ be positive integers. Let $\chi$ be a Dirichlet character with modulus $4N$.
 \begin{itemize}
     \item[(i)] A holomorphic function $f: \mathbb{H} \rightarrow \mathbb{C}$ is a modular form of weight $k+\frac{1}{2}$ for $\Gamma_0(4N)$ if:
     \begin{itemize}
         \item For any $z \in \mathbb{H}$ and for any $\gamma=\begin{pmatrix}
             a & b \\ c & d
         \end{pmatrix} \in \Gamma_0(N)$, we have:
         \begin{align*}
    f(z)=f|_{k+\frac{1}{2}}\gamma (z)=\chi(d)^{-1} \left( \frac{c}{d}\right)^{2k+1} \epsilon_d^{1-2k}(cz+d)^{-k-\frac{1}{2}} f(\gamma \cdot  z). 
\end{align*}
    Here we set $\sqrt{z}$ to be the branch of the square root having argument in $[\pi/2, -\pi/2)$. This yields a holomorphic function on the complex plane with the negative real axis removed. We define $z^{k/2}:=(\sqrt{z})^k$. Moreover, $\left( \frac{c}{d}\right)$ denotes the Legendre symbol and
    \begin{align*}
        \epsilon_d:= \begin{cases}
    1 & \text{if } d \equiv 1 \mod 4,\\
    i &  \text{if } d \equiv 3 \mod 4. \end{cases}
    \end{align*}
    \item The function $f$ is holomorphic at all cusps of $\Gamma_0(4N)$.    
     \end{itemize}
     We denote the space modular forms of weight $k+\frac{1}{2}$, level $4N$ and character $\chi$ by $M_{k+1/2}(\Gamma_0(4N),\chi)$.
    \item[(ii)] If $f$ vanishes at all cusps, then $f$ is a cusp form and is in the space $S_{k+1/2}(
    \Gamma_0(4N),\chi)$.
    \item[(iii)] For the trivial character, we set $M_{k+1/2}(4N):=M_{k+1/2}(\Gamma_0(4N),\chi_{triv})$ and $S_{k+1/2}(4N):=S_{k+1/2}(\Gamma_0(4N),\chi_{triv})$.
 \end{itemize}
\end{definition} 
We recall that for a prime $p$ the Hecke operator $T_{p^2}$ acts on a modular form $f=\sum_{n=1}^{\infty} c_nq^n$ of half-integral weight in the following way: 
 \begin{align*}
        f(z)|_{k+1/2}T_{p^2}=\sum_{n=0}^{\infty} (a_{p^2n}+\chi^*(p)\left( \frac{n}{p}\right)p^{k-1}a_{n}+\chi^*(p^2)p^{2k-1}a_{n/p^2})q^n
    \end{align*}
    Here $\chi^*(n):=\left( \frac{(-1)^k}{n}\right)\chi(n)$ and as above, if $ p^2 \nmid n$, then $a_{n/p^2}=0$.
Later we will see how we can establish a correspondence between the space of half-integral and integer weight modular forms using eigenfunctions under the Hecke operators.

\begin{definition}
For two functions $f,g \in S_{k+1/2}(\Gamma_0(4N), \chi)$, we set:
    \begin{align*}
        \langle f, g \rangle= \frac{1}{[\Gamma_0(4): \Gamma_0(N)]} \int_{\Gamma_0(N) \setminus \mathbb{H}}f(z) \overline{g(z)}y^{k-3/2}dxdy.
    \end{align*}
\end{definition}

\subsection{Binary Quadratic Forms}
Here we define  binary quadratic forms and study some of their properties. This helps us to define and understand the local polynomials in the statement of Theorem \ref{theo: main prod L-functions} better. Moreover it allows us to define maps between integer and half-integral weight modular forms in the next subsection.

Let $Q=[a,b,c]$ be a binary quadratic form, i.e. a triplet of integers $a,b,c \in \Z$. We set:
\begin{align*}
    [a,b,c](X,Y)=aX^2+bXY+cY^2.
\end{align*}
For any $r \in \Z$, we say that $r$ is represented by $Q$ if there exists $x,y \in \Z$ such that $Q(x,y)=r$.
We define the set of binary quadratic forms with fixed discriminant $D$ as follows:
\begin{equation*}
    \mathcal{Q}_D:=\{ Q=[a,b,c] : b^2-4ac=D\}.
\end{equation*}
The group $SL_2(\Z)$ acts on $\mathcal{Q}_D$. Let $\gamma=\begin{pmatrix}
    \gamma_{11} & \gamma_{12} \\ \gamma_{21} & \gamma_{22}
\end{pmatrix} \in SL_2(\Z)$. Then
\begin{align*}
    \left([a,b,c] \circ \gamma \right)(x,y)=[a,b,c](\gamma_{11} x+\gamma_{12} y, \gamma_{21} x + \gamma_{22} y)
\end{align*}
This action is compatible with the M\"{o}bius transformation in the following sens:
\begin{align*}
    \left([a,b,c] \circ \gamma \right)(z,1)=(\gamma_{21} z+\gamma_{22})^2[a,b,c](\gamma \cdot z, 1),
\end{align*}
where $z \in \mathbb{H}$. For any positive integer $N$, we set:
\begin{equation*}
    \mathcal{Q}_{N,D}:=\{ Q=[a,b,c]: N | a, \quad b^2-4ac=D\}.
\end{equation*}
The action of $SL_2(\Z)$ does not preserve this set. However, the action restricted to the congruence subgroup $\Gamma_0(N)$ does.

Moreover, we can define a Fricke involution on $\mathcal{Q}_{N,D}$. 
\begin{align*}
    [a,b,c]| W_N (x,y):=\frac{1}{N}[a,b,c] \circ\begin{pmatrix}
    0 & -1 \\ N &0
\end{pmatrix} (x,y)=[cN,-b,a/N](x,y)
\end{align*}
Then the Fricke involution preserves the discriminant and one can check that it defines a bijection on $\mathcal{Q}_{N,DD_0}$.

We introduce the extended genus character, as defined by Kohnen in \cite{Kohnen1985}:
\begin{definition}\label{def: generalized genus character}
    Let $D \equiv 0,1 \mod 4$ and let $Q=[a,b,c]$ be a binary form such that $D|b^2-4ac$.
    \begin{align*}
        \chi_{D}(Q)=\begin{cases}
            0 & \text{if } (a,b,c,D)>1\\
            \left( \frac{D}{r}\right) &\text{if } (a,b,c,D)=1, \quad Q \text{ represents }r, \quad (r,D)=1.
        \end{cases}
    \end{align*}
\end{definition}
This character is independent of the choice of number represented by $Q$.
\begin{remark}
The name of this function has been introduced by Gross, Kohnen and Zagier in \cite{Zagier-Gross-Kohnen-L-series-II-1987} who study its behavior on the set of binary quadratic forms $\mathcal{Q}_{N,DD_0}$ with the extra assumption that both $D$ and $D_0$ are squares modulo $4N$. They show under these conditions that the generalized genus character is as an extension of the classical genus character defined on primitive forms in $Q_{DD_0}$. 
\end{remark}
Kohnen shows the following explicit formula:
\begin{proposition}[Proposition 6 in \cite{Kohnen1985}]\label{prop: explicit genus char}
Let $[a,b,c]$ be a binary quadratic form with $b^2-4ac=|D|m$ and $a>0$. For every prime $p$, we set $p^*=\pm p^l \in \Z$ for some $l \in \Z_{\geq 0}$ such that $|p^*|$ is the exact power of $p$ dividing $D$ and $D/p^*$ is a fundamental discriminant.
Then
\begin{align*}
    \chi_{D}([a,b,c])=\prod_{p^{\nu}||a}\left( \frac{D/p^*}{p^{\nu}}\right)\left( \frac{p^*}{ac/p^{\nu}}\right).
\end{align*}
\end{proposition}
\begin{remark}
    Kohnen considers $1$ to be a fundamental discriminant.
\end{remark}

In \cite{Zagier-Gross-Kohnen-L-series-II-1987} Gross, Kohnen and Zagier, assuming extra conditions, prove the following even more effective formula:
\begin{proposition}[Proposition 1 in \cite{Zagier-Gross-Kohnen-L-series-II-1987}]
    Let $ [aN,b,c] \in \mathcal{Q}_{N,DD_0}$ such that both $D$ and $D_0$ are squares modulo $4N$. 
    Let $N_1,N_2$ be two psoitive integers and $D_1,D_2$ be two discriminants such that   $N=N_1N_2$, $D=D_1D_2$ and such that $(D_1,N_1a)=(D_2,N_2c)=1$. Then
    \begin{align*}
        \chi_{D}([aN,b,c])=\left( \frac{D_2}{N_1a}\right)\left( \frac{D_2}{N_2c}\right)
    \end{align*}
    If no such splitting exists, then $\chi_{D}([aN,b,c])=0$.
\end{proposition}
The authors also prove that under the same assumptions, $[a,b,c]$ is invariant under the Fricke involution. We have the following more general result:
\begin{lemma}\label{lem: chi(W_N Q)}
Let $D,D_0$ be two fundamental discriminants with $(D_0,N)=(D,N)=1$. Let $Q=[a,b,c]$ be a quadratic form such that $N|a$ and $b^2-4ac=DD_0$. 
Then
\begin{align*}
    \chi_{D_0}(Q | W_N)=\left( \frac{D_0}{N}\right)\chi_{D_0}(Q).
\end{align*}
\end{lemma}
\begin{remark}
    If $D_0$ is a square mod $N$, then we recover the result from \cite{Zagier-Gross-Kohnen-L-series-II-1987}.
\end{remark}
\begin{proof}
We have that $[a,b,c]|W_N=[cN,-b,a/N]$. We consider two cases: $(D_0,a,b,c)>1$ or  $(D_0,a,b,c)=1$.\\
As $(N,D_0)=1$, it follows that $(D_0,a,b,c)=1 \Leftrightarrow (D_0,a/N,-b,cN)=1$.
From this it follows that:
\begin{align*}
    \chi_{D_0}(Q)=0 \Leftrightarrow (D_0,a,b,c)>1 \Leftrightarrow (D_0,a/N,-b,cN)>1 \Leftrightarrow \chi_{D_0}(Q|W_N)=0
\end{align*}
Now assume that $(D_0,a,b,c)=(D_0,a/N,b,cN)=1$. Let $(D_0,r)=1$ be such that for some integers $x,y$:
    \begin{align*}
        [a,b,c](x,y)=ax^2+bxy+cy^2=r.
    \end{align*}
Then
\begin{align*}
    [Nc,-b,a/N](-y,Nx)=cNy^2-b(-y)(Nx)+a/N(Nx)^2=Nr
\end{align*}
Since $(D_0,N)=1$ the result follows.
\end{proof}
We end this section with some comments about the geodesic associated to a binary quadratic form $Q=[a,b,c]$. For $z \in \mathbb{H}$, we set:
\begin{align*}
    Q_z=[a,b,c]_z:=\frac{1}{\im(z)}(a|z|^2+b\re(z)+c).
\end{align*}
Let $S_Q:=\{ z \in \mathbb{H} : Q_z=0\}$ be the geodesic associated to $Q$. We denote by $C_Q$ its image in $\Gamma_0(N) \setminus \mathbb{H}$. Let $\Gamma_Q$ be the stabilizer of $Q$ in $\Gamma_0(N)$, then
\begin{align*}
     C_Q = \Gamma_Q \setminus S_Q.
\end{align*}
Let $D$ be the discriminant of $Q$ and let $(t,r)$ be the smallest positive solution to the Pell equation $t^2-|D|mu^2=4$, then $\Gamma_Q / \{ \pm 1 \} $ is generated by:
\begin{align}\label{eq: def stab Q}
    \begin{pmatrix}
        \frac{t-bu}{2}& -cr \\ au & \frac{t+bu}{2}
    \end{pmatrix}.
\end{align}
\subsection{Shimura-Shintani correspondance}
The Shimura and Shintani lifts map cusp forms of half-integral weight to cusp forms of integer weight, and vice versa, preserving the Hecke eigenvalues. In this section, we define the two operators and give a characterization in terms of a kernel operator. We follow Kohnen \cite{Kohnen1985}. We point out, that the assumption that the level $N$ is square-free is not needed for the results stated in this section. 

The Shimura lift maps forms of half-integral weight to forms of integer weight.
\begin{definition}\label{def: Shimura lift}
Let $N$ be a integer. Let $f= \sum c_n q^n \in S_{k+1/2}(4N, \chi)$ and $t$ be a square-free-integer. Set:
\begin{align*}
        \chi_t(d)=\chi(d) \left( \frac{(-1)^{k}t}{d}\right)
    \end{align*}
to be a character modulo $4Nt$. Define 
    \begin{align*}
        \Sh_t(f):=\sum_{n=0}^{\infty} \left( \sum_{d | n} \chi_t(d)d^{k-1} c_{t (n/d)^2} \right) q^n.
    \end{align*}
    Then $\Sh_t(f) \in M_{2k}(2N, \chi^2)$ is the $t$\textsuperscript{th} Shimura lift of $f$.
\end{definition}

The Shimura lift has the property that is preserves the Hecke eigenvalues in the following sens:
\begin{proposition}
      Let $f \in S_{k+1/2}(4N, \chi)$ and let $p$ be a prime. Then:
    \begin{align*}
        \Sh_t(f|_{k+1/2}T_{p^2})=\Sh_t(f)|_{2k}T_p.
    \end{align*}
\end{proposition}
The function mapping integer weight forms to half-integral weight forms is the Shintani operator. Basically, it is given by a sum over the following integral:
\begin{definition}
Let $F\in S_{2k}(N)$. Let $D$ be a fundamental discriminant such that $(-1)^kD>0$. Let $m \in \N$ with $(-1)^km \equiv 0,1 \mod 4$. We define:
\begin{align*}
    r_{k,N,F,D,(-1)^km}:= \sum_{Q \in \Gamma_0(N) \setminus Q_{N,|D|m}}\chi_{D}(Q) \int_{\Gamma_Q\setminus S_Q}F(z)Q(z,1)^{k-1}dz.
\end{align*}
\end{definition}
This integral allows us to define the Shintani lift:
\begin{definition}\label{def: Shintani-lift}
    Let $ F \in S_{2k}(N)$. Let $D$ be a fundamental discriminant such that $(-1)^kD>0$. Let $m \in \N$ with $(-1)^km \equiv 0,1 \mod 4$. Set
    \begin{align*}
        F|S^*_{D}:=\sum_{\substack{m \geq 1, \\ (-1)^km \equiv 0,1 \mod 4}} \left(\sum_{t |N} \mu(t) \left( \frac{D}{t}\right)t^{k-1} r_{k,Nt,F,D,(-1)^kmt^2}\right)e^{2 \pi i m t}.
    \end{align*}
    Where $\mu(t)$ is the Mobius function. Then $F|S^*_{D}$ is in the space $S_{k+1/2}(N)$.\footnote{As we will see later, it follows directly from this definition that the Shintani lift maps $F$ to the Kohnen subspace (see definition \ref{def: Kohnen space non trivial character}), i.e. $F|S^*_{D} \in S_{k+1/2}^K(N)$.}
\end{definition}

We describe how the two functions are connected through a kernel operator. This kernel is essentially a sum of Petersson innerproducts with the following functions:

\begin{definition}\label{def: sec1-f_k,N,D,D_0}
Let $N \geq 1$. Let $D,D_0$ be such that $D,D_0 \equiv 0,1 \mod 4$ and $(-1)^kD,(-1)^kD_0>0$. Let $k \geq 2$
\begin{align*}
        f_{k,N,D,D_0}(z)=\sum_{Q \in \mathcal{Q}_{N,DD_0}}\chi_{D_0}Q(z,1)^{-k}
    \end{align*}
This is a cusp form of weight $2k$ for $\Gamma_0(N)$.
\end{definition}

\begin{remark}\label{rem: real cf of f_k,N,D,D_0}
    One can easily show that the q-expansion of $f_{k,N,D,D_0}$ has real coefficients. Indeed as $\chi_{D_0}([a,-b,c])=\chi_{D_0}([a,b,c])$, we can map $[a,b,c] \in \mathcal{Q}_{N,DD_0}$ to $[a,-b,c]\in \mathcal{Q}_{N,DD_0}$ and get that
    \begin{align*}
        \overline{f_{k,N,D,D_0}(-\overline{z})}=\sum_{Q \in \mathcal{Q}_{N,DD_0}}\chi_{D_0}(Q)(az^2-bz+c)^{-k}= f_{k,N,D,D_0}(z).
    \end{align*}
\end{remark}

The kernel operator is given by:
\begin{definition}
Let $N$ be odd, $k \geq 1$ an integer and let $D$ be a fundamental discriminant with $(-1)^kD>0$. For $z,\tau \in \mathbb{H}$ we set:
\begin{align*}
    \Omega_{k,N,D}&(z,\tau):=i_Nc_{k,D} \times \\
    &\sum_{\substack{m \geq 1,\\ (-1)^km \equiv 0,1 \mod 4}} m^{k-1/2} 
    \left( \sum_{t | N} \mu(t) \left( \frac{D}{t} \right) t^{k-1} f_{k, N/t,D,(-1)^km}(tz)\right)e^{2 \pi i m \tau }.
\end{align*}
Here
\begin{align*}
    c_{k,D}=\left((-1)^{[k/2]}|D|^{-k+1/2} \pi \binom{2k-2}{k-1}2^{-3k+2}\right)^{-1}, \quad i_N=[SL_2(\Z):\Gamma_0(N)].
\end{align*}
\end{definition}
We are now ready to state the following result:
\begin{theorem}[Theorem 2 in \cite{Kohnen1985}]\label{theo: kernel Shintani Shimura} The Shimura and Shintani-lift are adjoint operators on cusp forms. Moreover, we have:
\begin{itemize}
    \item For all $F \in S_{2k}(4N)$, we have:
    \begin{align*}
        F|\Shs_D(\tau)=\langle F,\Omega_{k,N,D,}(\cdot, - \overline{\tau})\rangle.
    \end{align*}
    \item For all $f \in S_{k+1/2}(4N)$, we have:
\begin{align*}
        f|\Sh_D(z)=\langle f, \Omega_{k,N,D,}(-\overline{z}, \cdot )\rangle.
    \end{align*}
\end{itemize}
\end{theorem}

Due to this kernel operator, we can express the $D_0$th coefficient of the $D$th Shintani lift of a newform $F$ as an inner product:
\begin{lemma}\label{lem: Coeff of Shintani lift}
    Let $F \in S_{2k}^{new}(N)$ be a newform. Then:
    \begin{align*}
        c_{|D_0|}(F|S^*_D)=i_{N}c_{k,D} \langle F, f_{k,N,D,D_0}\rangle.
    \end{align*}
\end{lemma}
\begin{proof}
By Theorem \ref{theo: kernel Shintani Shimura} we have:
\begin{align*}
    F|S_D^*(\tau)&=\langle F, \Omega_{k,N,D}(\cdot,-\overline{\tau}) \rangle\\
    & = \langle F, i_{N}c_{k,D} \sum_{\substack{b \geq 1,\\ (-1)^kb \equiv 0,1 \mod 4}} b^{k-1/2}\left( \sum_{t | N} \mu(t) \left( \frac{D}{t} \right) t^{k-1} f_{k, N/t,D,(-1)^kb}\right)e^{-2 \pi i b \overline{\tau} } \rangle\\
    & = i_{N}c_{k,D} \sum_{\substack{b \geq 1,\\ (-1)^kb \equiv 0,1 \mod 4}} b^{k-1/2} e^{2 \pi i b \tau } \sum_{t | N} \mu(t) \left( \frac{D}{t} \right) t^{k-1}\langle F, f_{k, N/t,D,(-1)^kb} \rangle.
\end{align*}
Thus the $|D_0|$th coefficient is given by
\begin{equation}\label{N-Kohnen-D_0}
\begin{split}
c_{|D_0|}(f|S_D^*)&=i_{N}c_{k,D} \sum_{t | N} \mu(t) \left( \frac{D}{t} \right) t^{k-1}\langle F, f_{k, N/t,D,(-1)^kD_0} \rangle\\
   & =i_{N}c_{k,D} \langle f, f_{k,N,D,(-1)^kD_0} \rangle
   \end{split}
\end{equation}
Here the last equation follows since all other inner product vanish, as $F$ is a newform and $f_{k, N/t,D,(-1)^km}$ is of level $N/t$.
\end{proof}

\subsection{Central values of $L$-functions and coefficients of Modular forms of half-integral weight}
In the following we define a subspace of modular forms of half-integral weight, which is in one-to-one correspondence with new forms of integer weight. We then explain, how the coefficients of the half-integral weight form encode the central value of the $L$-function of the  twists of the form of integer weight.

The case when the level is square-free has been described by Kohnen (\cite{Kohnen1982}, \cite{Kohnen1985}). Here, we follow Ueda  \cite{Ueda-1998} who generalized the correspondence between integer and half-integral weight forms and Sakata \cite{Sakata-N} who generalized Kohnen's statements about the central values of twisted $L$-functions.
\subsubsection{Newforms of Half-Integral Weight}
We start by defining newforms of half-integral weight.
\begin{definition}\label{def: Kohnen space non trivial character}
Let $\chi=\left( \frac{N_0}{\cdot}\right) $ be a character modulo $N$, i.e. $N_0|N$, and set $\epsilon:=\left( \frac{-1}{N_0}\right)$. Then we define the Kohnen space as:
\begin{align*}
    S_{k+1/2}^K(\Gamma_0(4N), \chi):= \left\{ f=\sum_{n=1}^{\infty} c_nq^n \in \right. &S_{k+1/2}(\Gamma_0(4N), \chi) \text{ s.t. } \\
    & \Biggl. c_n=0 \text{ for } \epsilon(-1)^kn\equiv2,3 \mod 4 \Biggr\}
\end{align*}
\end{definition}
We now define the lifting maps:
\begin{definition}
    Let $p|N$ be an odd prime divisor. Then
    \begin{align*}
        U_p: S_{K+1/2}(\Gamma_0(4N), \chi) &\rightarrow S_{k+1/2}\left(\Gamma_0(4N), \chi \left( \frac{p}{\cdot}\right) \right),\\
        f=\sum_{n=1}^{\infty}c_nq^n &\mapsto f|U_p:=\sum_{n=1}^{\infty} c_{pn}q^n.
    \end{align*}
\end{definition}
Moreover, we define:
\begin{definition}
Let $d$ be a positive integer. Then we define the map:
\begin{align*}
    V_d : M_{k+1/2}(\Gamma_0(4N),\chi) &\rightarrow M_{k+1/2}\left(\Gamma_0(4Nd), \chi \left(\frac{d}{\cdot}\right)\right) \\
    f= \sum_{n=0}^{\infty} c_nq^n&\mapsto f|V_d (z)=f(dz)=\sum_{n=1}^{\infty} a_n q^{dn}.
\end{align*}
\end{definition} 
Finally, we define the twisting operators $R_p$:
\begin{definition}\label{def: R_p half integral}
Let $p$ be an odd prime. 
    \begin{align*}
        S_{k+1/2}(\Gamma_0(4N), \chi) &\rightarrow S_{k+1/2}\left(\Gamma_0(4Np), \chi \left( \frac{p^2}{}\right)\right), \\
        f=\sum_{n=1}^{\infty}c_nq^n &\mapsto f|R_p:=\sum_{n=1}^{\infty}c_n \left( \frac{n}{p}\right)q^n.
    \end{align*}
    If $p^2|N$, then $R_p$ preserves the level. 
\end{definition}

The space of oldforms is defined as follows:
\begin{definition}\label{def: S^K,old S^k,new}
    Let $N$ be an odd integer. Let $k \geq 2$. We define the space of old forms of level $4N$ and weight $k+1/2$ as:
    \begin{align*}
        S&^{K,old}_{k+1/2}(4N):=\\
        &\sum_{\substack{0<d_1|N\\d_1 \neq N}}
        \sum_{0 <d_2 | (N/d_1)} \sum_{\substack{\xi: (\Z/4d_1\Z )^*\rightarrow \{\pm 1 \} \\\xi\left( \frac{d_2}{\cdot}\right)=\mathbb{1}}}
        S^K_{k+1/2}(\Gamma_0(4d_1), \xi)|V_{d_2}\\
         &+ \sum_{\substack{0<d_1|N\\d_1 \neq N}}
        \sum_{0 <d_2 | (N/d_1)^2} 
        \sum_{\substack{\xi: (\Z/4d_1\Z )^*\rightarrow \{\pm 1 \} \\\xi\left( \frac{d_2}{\cdot}\right)=\mathbb{1}}} \sum_{\substack{0 \leq e_l \leq 2\\ l \in II(N)} } \left(S^K_{k+1/2}(\Gamma_0(4d_1), \xi)|U_{d_2} \right) \vert \prod_{l \in II(N)} R_l^{e_l}
    \end{align*}
Here $II(N):=\{p|N: \text{ord}_p(N) \geq 2\}$.
    
    The space of newforms of level $4N$ and weight $k+1/2$ is defined as the orthogonal complement of $S^{K,old}_{k+1/2}(4N)$ and denoted by $S^{K,new}_{k+1/2}(4N)$.
\end{definition}

\begin{remark}
    \begin{itemize}
        \item[(i)] If $N$ is square-free, then we recover the definition from Kohnen in \cite{Kohnen1982}.
        \item[(ii)] Let $p$ be an odd prime. If $m \geq 3$, then the space of oldforms is given by:
    \begin{align*}
       S^{K,old}_{k+1/2}(4 \cdot p^m) =S_{k+1/2}^K\left(4\cdot p^{m-1}\right)+S_{k+1/2}^K\left(\Gamma_0(4\cdot p^{m-1}), \left( \frac{p}{}\right)\right)|V_{p}
    \end{align*}
If $m=2$, then 
    \begin{align*}
       S^{K,old}_{k+1/2}(4 \cdot p^2) &=S_{k+1/2}^K\left(4\cdot p\right)
       +S_{k+1/2}^K\left(\Gamma_0(4\cdot p), \left( \frac{p}{}\right)\right)|V_{p}\\
       &+S_{k+1/2}^K\left(4\cdot p\right)|R_{p}
       +S_{k+1/2}^K\left(4\cdot p\right)|R_{p^2}
    \end{align*} 
    \end{itemize}
\end{remark}
\subsubsection{Eigenspaces of Newforms}
To state the correspondence, we need to define several eigenspaces, both of newforms of integer weight as well as of half-integral weight. 

To do so we will distinguish primes which divide the level to an odd or even power. We define the following notation:
\begin{itemize}
    \item  $I(N):=\{ p |N: \text{ord}_p(N) = 1\}$
    \item  $II(N):=\{ p |N: \text{ord}_p(N) \geq 2\}$
    \item $II(N)_{even}:=\{ p |N: \text{ord}_p(N) \geq 2 \text{ and } 2 \mid \text{ord}_p(N)\}$
    \item $II(N)_{odd}:=\{ p |N: \text{ord}_p(N) \geq 2 \text{ and } 2 \nmid \text{ord}_p(N) \}$
    \item $II(N)_2:=\{ p | N: \text{ord}_p(N)=2\}$
    \item $II(N)_2^*:=\left\{ p | N: \text{ord}_p(N)=2 \text{ and } \left( \frac{-1}{p}\right)=1\right\}$
\end{itemize}

\paragraph{Very Newforms}
Before defining the eigenspace, we introduce a subcategory of newforms of integer weight. These are forms that cannot be constructed by twisting lower level modular forms with the quadratic character $\left( \frac{\cdot}{p}\right)$:
\begin{definition}\label{def: R_p 2k}
    Let $N'$ be the least common multiple of $N$ and $p^2$. Then
    \begin{align*}
        R_p: S_{2k}(N) &\rightarrow S_{2k}(N'),\\
        F=\sum_{n=0}^{\infty} a_nq^n &\mapsto \sum_{n=0}^{\infty}a_n \left( \frac{n}{p}\right)q^n.
    \end{align*}
    In particular, if $p^2|N$, then this map is level preserving. 
\end{definition}

This map twists the eigenvalues of an eigenfunction in the following sense:
\begin{lemma}[Proposition 3.2 and 3.4 in \cite{Atkin-Li-1978},  Proposition A.3 in \cite{Ueda-1993}]\label{lem: F|R_p|W_Pp=...}
Let $F \in S_{2k}(N)$. Let $p,l$ be two distinct primes.
\begin{itemize}
    \item[(i)] If $(l,N)=1$, then
    \begin{align*}
        F|R_p|_{2k}T_l=\left( \frac{l}{p}\right)F|_{2k}T_l|R_p.
    \end{align*}
    \item[(ii)] If $l^m||N$, then  
    \begin{align*}
        F|R_p|W_{l^m}=\left( \frac{l^m}{p}\right)F|W_{l^m}|R_p.
    \end{align*}
    \item[(iii)] If $p||N$, then:
    \begin{align*}
        F|R_p|W_{p^2}=\left(\frac{-1}{p}\right)F|R_p.
    \end{align*}
\end{itemize}
\end{lemma}
Very newforms are defined as follows:
\begin{definition}\label{def: very new form}
    Let $F=\sum a_n q^n \in S_{2k}(N)$ be a newform. $F$ is a very newform if it is in the orthogonal complement of newforms of level strictly smaller than $N$, lifted by twisting operators $R_p$ with $p^2||N$, i.e. in the orthogonal complement given by:
    \begin{align*}
        \sum_{0<d<N} \sum_{p^2 || N} S_{2k}^{new}(d)|R_p.
    \end{align*}
    We denote the space of very newforms by $S^*_{2k}(N)$.
\end{definition}
\begin{remark}
\begin{itemize}
    \item[(i)]
If $p^3|N$ then for any $F \in S_{2k}^{new}(d)$ with $d<N$, the form $F|R_p$ is in $S_{2k}^{new}(N')$, where $N'$ is the least common multiple of $d$ and $p^2$. Thus in particular $N'<N$ and therefore $F|R_p$ is not a newform of level $N$. Hence the orthogonal complement of $\sum_{0<d<N} \sum_{p^2 || N} S_{2k}^{new}(d)|R_p$ in $S_{2k}^{new}(N)$ is equal to the orthogonal complement of $\sum_{0<d<N} \sum_{p^2 | N} S_{2k}^{new}(d)|R_p$.
    \item[(ii)] Independently from Ueda, other mathematicians have studied the space of very new forms. Mao shows in \cite{Mao} (Lemma 2.3) that $F$ is a very newform if and only if $F|R_p$ is a newform for all primes $p$ such that $p^2|N$. Therefore we have:
    \begin{align*}
    S^*_{2k}(N):=\left\{ F \in  \right. &S^{new}_{2k}(N) :  \\
   &  \left.  F|R_p \text{ is also a newform in }S^{new}_{2k}(N), \quad \forall p \text{ with } p^2 | N \right\}.
\end{align*}
    Note that this equivalent definition has computational advantages: it is easier to check if the twists of a given newform are also newforms, instead of checking if it is in the orthogonal complement of some subspaces of lower level  modular forms. 
\end{itemize}
\end{remark}

\paragraph{Eigenspaces of integer weight forms under Atkin-Lehner involutions.}
We divide the space of newforms of integer weight forms into subspaces as follows: Let $F$ be a newform of level $N$. Let $S=II(N)$ or $S=I(N)\cup II(N)$. Then
\begin{align}\label{eq: def of tau}
    \tau: S \rightarrow \{\pm 1\}
\end{align}
describes the eigenvalues of $F$ under the Atkin-Lehner operators corresponding to the set $S$ if
\begin{align*}
    F|W_{p^{\text{ord}_p(N)}}=\tau(p)F, \quad \forall p \in S.
\end{align*}
We denote the subspaces with fixed eigenvalue under the Atkin-Lehner involutions by:
\begin{align*}
    S^{new,\tau}_{2k}(N)&:=\{ F \in S_{2k}^{new}(N) : F|W_{p^{\text{ord}_p(N)}}=\tau(p)F, \quad \forall p \in S\},\\
    S^{*,\tau}_{2k}(N)&:=\{ F \in S^*_{2k}(N) : F|W_{p^{\text{ord}_p(N)}}=\tau(p)F, \quad \forall p \in S\}.
\end{align*}

\paragraph{Eigenspaces of half-integral weight forms under twisting operators $R_p$}
We also define subspaces of half-integral weight forms using the twisting operators $R_p$. It suffices to consider the primes whose square divides the level, i.e. the set $II(N):=\{ p: p^2 |N \}$. Let
\begin{align}\label{eq: def of kappa}
    \kappa: II(N)\rightarrow \{\pm 1\}.
\end{align}
We set:
    \begin{align*}
        S_{k+1/2}^{K,new,\kappa}(4N):=\left\{ g \in S_{k+1/2}^{K, new}(4N): g|R_p=\kappa(p)g, \quad \forall p \in II(N)\right\}.
    \end{align*}

It follows from the following lemma, that these subspaces are orthogonal to each other:
\begin{lemma}[Corollary 1.10 in \cite{Ueda-1991-trace}]\label{lem: self-adjoint-twisting-operator}
    Let $f,g \in S_{k+1/2}(4N)$ be two cusp forms of half integral weight. Let $\chi$ be a non-trivial primitive character of odd conductor $M$ with $M^2|N$ and such that $\chi^2=1$. Then
    \begin{align*}
        \langle f\otimes \chi, g \rangle = \langle f, g \otimes \chi\rangle.
    \end{align*}
\end{lemma}

We finish this subsection by showing that given a discriminant $D$, all half-integral weight forms with a non-zero $|D|$th coefficient have the same system of eigenvalues under the twisting operator.
\begin{lemma}\label{lem: cf-are-zero-except-kappa-0}
Let $N$ be an odd integer and $k \geq 1$. Let $\kappa: II(N) \rightarrow \{ \pm 1 \}$ and let $f_{\kappa}= \sum_{n=1}^{\infty} c_nq^n \in S_{k+1/2}^{K,new, \kappa}(4N)$. Let $D$ be a fundamental discriminant. If there exists a prime $p_0 \in II(N)$ such that
\begin{align*}
    \kappa(p_0)\neq\left( \frac{|D|}{p_0}\right),
\end{align*}
then 
\begin{align*}
    c_{|D|}=0.
\end{align*}
\end{lemma}

\begin{proof}
We know that $f$ satisfies $f_{\kappa}|R_p=\kappa(p)f_{\kappa}$ for all $p \in II(N).$ For $p_0$, this implies that:
\begin{align*}
    \sum_{n=0}^{\infty} \left( \frac{n}{p_0}\right) c_n q^n
    =\sum_{n=0}^{\infty} \kappa(p_0) c_n q^n.
\end{align*}
As $\kappa(p_0) \neq \left( \frac{|D|}{p_0}\right)$, this implies that $c_{|D|}=0$.
\end{proof}
\subsubsection{Correspondance}
We are now ready to state the isomorphism between the space of half-integral and integer weight cusp forms.

\begin{theorem}[Theorem 1 in \cite{Sakata-N}]\label{Sakata-general-decomposition} Let $k>1$.
Set $N_1= \prod_{p \in I(N)} p$ be the square-free part of $N$.
Let $\kappa: II(N) \rightarrow \{ \pm 1\}$ be a map.  We define a map $\tau_{\kappa}: II(N) \rightarrow \{ \pm 1\}$ by:
\begin{align*}
    \tau_{\kappa}=\begin{cases}
        1 & \text{ if } p \in II(N)_{even},\\
        \kappa(p) \left( \frac{-1}{p}\right)^k & \text{ if } p \in II(N)_{odd}.
    \end{cases}
\end{align*}
Let $I,J,K \subset II(N)_2$ be a partition with $I\cup J \neq \emptyset$ and $I,J \subset II(N)^*_2=\{ p \in I(N)_2 | \left( \frac{-1}{p}\right)=1\}$. For a fixed partition, we let $\tilde{\tau}_{\kappa}$ run over the set of maps $II(N)\setminus I \rightarrow \{\pm 1\}$ satisfying:
\begin{align*}
    \tilde{\tau}_{\kappa}(p)=
    \begin{cases}
        1 & \text{ if } p \in II(N)_{even} \setminus(I+J), \\
        \kappa(p) \left( \frac{-1}{p}\right)^k \prod_{q \in I+J }\left( \frac{p}{q}\right)& \text{ if } p \in II(N)_{odd}.
    \end{cases}
\end{align*}
Then we have the following isomorphism as modules over the Hecke algebra:
\begin{align*}
    S_{k+1/2}^{K,new,\kappa}(4N) \cong &S^{*, \tau_{\kappa}}_{2k}(N)\bigoplus\\
    &\bigoplus_{\substack{II(N)_2=I \cup J \cup K\\I\cup J \neq \emptyset, \quad I,J \subset II(N)^*_2}} \bigoplus_{\tilde{\tau}_{\kappa}} \quad S_{2k}^{*, \tilde{\tau}_{\kappa}} \left( N_1 \prod_{l \in J}l \prod_{II(N)-(I\cup J)}p^{ord_p(N)} \right)| \prod_{p \in I\cup J}R_p.
\end{align*}
\end{theorem}

\begin{remark}
    Note that Sakata makes two typos when stating the theorem. First he claims that $\tilde{\tau}$ is defined on $II(N)$, while this map is not defined on $I$. He also fixes the conditions for this map on the entire set $II(N)$ and does neither exclude the sets $I$ nor $J$.
\end{remark}

We show how we can simplify this statement by studying the space on the right hand side.
If an integer weight form $F$ is obtained by twisting a lower weight form, then we have the following relation between their respective eigenvalues under the Atkin-Lehner involutions:
\begin{lemma}\label{lem: system eigenvalues of twist}
    Let $N$ be a positive integer. Let $N_1$ be the square-free part of $N$ and $I,J \subset II(N)^*$ be two disjoint subsets. Let $\tilde{\tau}: I(N) \cup II(N)\setminus I \rightarrow \{ \pm 1\}$ be a map and assume that
    \begin{align*}
    F \in S^{*, \tilde{\tau}}_{2k}\left(N_1 \prod_{p \in J}p \prod_{p \in II(N)-(I\cup J)} p^{ord_p(N)} \right)|\prod_{p \in I \cup J}R_p \subset S_{2k}(N).
\end{align*}
Then $\tau(p):I(N)\cup II(N)\rightarrow \{\pm 1\}$, the system of eigenvalues under the Atkin-Lehner involutions of $F$, satisfies:
\begin{align*}
    \tau(p)=\begin{cases}
    \tilde{\tau}(p)& \text{if } p \in II(N)_{even}\setminus (I\cup J), \\
    \tilde{\tau}(p)\prod_{q \in I\cup J} \left( \frac{p}{q}\right)& \text{if } p \in I(N) + II(N)_{odd},\\
    1 & \text{if }p \in I \cup J.
    \end{cases}
\end{align*}
\end{lemma}

\begin{proof}
For all $p \in I(N)\cup II(N)\setminus (I\cup J)$, with $p^m||N$. Using Lemma \ref{lem: F|R_p|W_Pp=...} (ii) we find that:
\begin{align*}
    F|\prod_{q \in I\cup J}R_q|W_{p^m}
    =\prod_{q \in I\cup J} \left( \frac{p^m}{q}\right) F|W_{p^m}|\prod_{q \in I\cup J}R_q
    =\tilde{\tau}(p)\prod_{q \in I\cup J} \left( \frac{p^m}{q}\right) F|\prod_{q \in I\cup J}R_q.
\end{align*}
Therefore
\begin{align*}
    \tau(p)=\begin{cases}
        \tilde{\tau}(p)& \text{if } m \text{ is even,}\\
        \tilde{\tau}(p)\prod_{q \in I\cup J} \left( \frac{p}{q}\right)& \text{if } m \text{ is odd.}
    \end{cases}
\end{align*}
Let $p \in I\cup J \subset \left\{ p |N : p^2||N, \left( \frac{-1}{p}\right)=1\right\} $. Then by Lemma \ref{lem: F|R_p|W_Pp=...} (ii) and (iii) we have that:
\begin{align*}
    F|\prod_{q \in I\cup J}R_q|W_{p^2}
    &= F|R_p|\prod_{q \in I\cup J\setminus \{p\}}R_q|W_{p^2}\\
    &= \prod_{q \in I\cup J\setminus \{p\}}\left( \frac{p^2}{q}\right) F|R_p|W_{p^2}|\prod_{q \in I\cup J\setminus \{p\}}R_q\\
    &= \prod_{q \in I\cup J\setminus \{p\}} \left( \frac{p^2}{q}\right) \left( \frac{-1}{p}\right)F|R_p|\prod_{q \in I\cup J\setminus \{p\}}R_q
\end{align*}
Since $\left( \frac{-1}{p}\right)=1$ we have that:
\begin{align*}
    \tau(p)=1
\end{align*}
\end{proof}

\begin{lemma}\label{lem: all newforms with positive even values}
With the same notation as in Theorem \ref{Sakata-general-decomposition} set:
    \begin{align*}
       V_1:=&S^{*, \tau_{\kappa}}_{2k}(N)\bigoplus\\
    &\bigoplus_{\substack{II(N)_2=I \cup J \cup K\\I\cup J \neq \emptyset, \quad I,J \subset II(N)^*_2}} \bigoplus_{\tilde{\tau}_{\kappa}} \quad S_{2k}^{*, \tilde{\tau}_{\kappa}} \left( N_1 \prod_{l \in J}l \prod_{II(N)-(I\cup J)}p^{ord_p(N)} \right)| \prod_{p \in I\cup J}R_p.
    \end{align*}
  Set $ V_2:= S_{2k}^{new, \tau}(N)$ where
  \begin{align*}
      \tau(p)=\begin{cases}
            1, \quad & p \in II(N)_{even},\\
            \kappa(p) \left( \frac{-1}{p}\right)^k & p \in II(N)_{odd}.
        \end{cases}
  \end{align*}
    Then $V_1=V_2$.
\end{lemma}

\begin{proof}
    From Lemma \ref{lem: system eigenvalues of twist}, it follows that $V_1 \subset V_2$. Now assume that there is a form $F \in V_2 \setminus V_1$. Since in particular very newforms are contained in $V_1$ we know that $F$ is of the form $F_0|R_p$ for some $F_0 \in S_{2k}(N_0)$ with $N_0<N$ and $p^2||N$ (see also definition \ref{def: very new form}). 
    Assume that $p^2||N$ with $\left( \frac{-1}{p}\right)=-1$. By Lemma \ref{lem: F|R_p|W_Pp=...}, we have that:
    \begin{align*}
        F_0|R_p|W_{p^2}=\left( \frac{-1}{p}\right)F_0|R_p=-F.
    \end{align*}
    Thus $F \notin V_2$. This proves the claim.
\end{proof}
Hence we have the following corollary from Theorem \ref{Sakata-general-decomposition}:
\begin{corollary}\label{cor: decomp Sakata}
Let $k>1$. Let $N$ be an integer. Let $\kappa: II(N) \rightarrow \{ \pm 1\}$ be a map. Let
\begin{align*}
      \tau(p)=\begin{cases}
            1, \quad & p \in II(N)_{even},\\
            \kappa(p) \left( \frac{-1}{p}\right)^k & p \in  II(N)_{odd}.
        \end{cases}
  \end{align*}
Then we have the following isomorphism as modules over the Hecke algebra:
\begin{align*}
    S_{k+1/2}^{K,new,\kappa}(4N) \cong S_{2k}^{new,\tau}(N).
\end{align*}
\end{corollary}

\begin{remark}
    \begin{itemize}
        \item[(i)] If $N$ is square-free, then this is Theorem 2 from \cite{Kohnen1982}. 
        \item[(ii)] Let $k>1$ and $p$ be an odd prime. For $m>1$ we have the following isomorphisms as Hecke modules.
\begin{align*}
        S^{K,new, \pm 1}_{k+1/2}(4p^m)\simeq \begin{cases}
            S_{2k}^{new,+1}(p^m), & \text{if m is even},\\
            S_{2k}^{new,\pm \left( \frac{-1}{p}\right)^k}(p^m), & \text{if m is odd}.
        \end{cases}
    \end{align*}
    \end{itemize}
\end{remark}
\subsubsection{Central Values of $L$-functions.}
We are finally ready to state the result from Sakata \cite{Sakata-N}, generalizing Kohnen's result about twisted L-functions and coefficients of half-integral weight forms to general level. The following property will play an important role:
\begin{definition}\label{def: tau_odd}
    Let $I,J \subset II(N)^*$ be two disjoint subsets. Let $\kappa: II(N) \rightarrow \{ \pm 1\}$ and $ \tilde{\tau}: II(N)\setminus I \rightarrow \{ \pm 1\}$ be two maps.  We say that $\tau$ satisfy the condition $(\tilde{\tau}_{odd})$ if
    \begin{align*}
    \kappa(p)=\tilde{\tau}(p)\left(\frac{-1}{p}\right)^k \prod_{q \in I\cup J}\left( \frac{p}{q}\right), \quad \forall p \in II(N)_{odd}, \quad (\tilde{\tau}_{odd}). 
    \end{align*}
\end{definition} 

We now state Sakata's theorem. It is somewhat ambiguous what he claims exactly since he does not distinguish in his paper between the system of Atkin-Lehner eigenvalues of a modular form and its twist. With Corollary \ref{cor: Sakata L-functions} we formulate a more elegant and clean version of his result.
\begin{theorem}[Theorem 3 in \cite{Sakata-N}]\label{theo: Sakata L-function}
Let $k>1$.
Let $N$ be an integer with square-free part $N_1$ and $I,J$ be two disjoint subsets in $II(N)_2^*$. Let
\begin{align*}
    F \in S^{*, \tilde{\tau}}_{2k}\left(N_1 \prod_{p \in J}p \prod_{p \in II(N)-(I\cup J)} p^{ord_p(N)} \right)|\prod_{p \in I \cup J}R_p \subset S_{2k}(N).
\end{align*}
Let $D$ be a fundamental discriminant with $(-1)^kD>0$ and $(D,N)=1$. Denote by $\tau$ the system of eigenvalues of $F$ under the Atkin-Lehner operators. Assume that $\tau$ satisfies:
\begin{align*}
    \tau(p)=\begin{cases}
        \left( \frac{D}{p}\right) & \text{if } p \in I(N) +II(N)_{odd},\\
        1 & \text{if } p \in II(N)_{even}-(I+J).
    \end{cases}
\end{align*}
Then we have:
\begin{align*}
    \sum_{\substack{\kappa \in Map(II(N), \{\pm 1\})\\ \kappa \text{ with } (\tilde{\tau}_{odd})} } \frac{|c_{f_{\kappa}(|D|)}|^2}{\langle f_{\kappa}, f_{\kappa} \rangle} =
    2^{\nu(N)}\frac{(k-1)!}{\pi^k} |D|^{k-1/2}\frac{L(F\otimes \chi_D,k)}{\langle F,F\rangle}.
\end{align*}
Here the map $f_{\kappa}$ is the (up to scalar multiplication) unique half-integral weight modular form in $S_{k+1/2}^{K,new,\kappa}(4N)$ which corresponds to $F$ as in Theorem \ref{Sakata-general-decomposition}. Moreover $\nu(N)$ is the number of different prime divisors of $N$.
\end{theorem}
Note that in the previous subsection (in particular in Theorem \ref{Sakata-general-decomposition} and Corollar \ref{cor: decomp Sakata}) we only considered the eigenvalues of the integer weight forms under Atkin-Lehner involutions corresponding to primes in $II(N)$. Here we also fix the eigenvalues of Atkin-Lehner involutions corresponding to primes in $I(N)$.
\begin{remark}
\begin{itemize}
\item[(i)] In particular, if $J = I = \emptyset$, then $F$ is a very newform.
\item[(ii)] As mentioned before Sakata does not distinguish between the systems of Atkin-Lehner eigenvalues $\tau$ (corresponding to $F$) and $\tilde{\tau}$ (corresponding to the form of lower level whose twist is equal to $F$). These two maps are not the same in general: from Lemma \ref{lem: system eigenvalues of twist} it follows that for $p \in I(N) \cup II(N)_{odd}$ that $\tau(p)=\tilde{\tau}(p)\prod_{q \in I \cup J}\left( \frac{p}{q}\right)$, thus in general $\tau(p)\neq\tilde{\tau}(p) $. This implies it does make a difference whether $\kappa$ satisfies $(\tilde{\tau})_{odd}$ or $(\tau)_{odd}$ (see p. 76 in \cite{Sakata-N}). 
\item[(iii)] In light of Lemma \ref{lem: system eigenvalues of twist}, we can replace the conditions on $\tau$, the eigenvalues of $F$, with the following conditions on $\tilde{\tau}$:
    \begin{align*}
        \tilde{\tau}(p)=\begin{cases}
            \left( \frac{D}{p}\right)\prod_{q\in I+J}\left( \frac{p}{q}\right) &\text{if } p \in I(N)+II(N)_{odd}\\
            1 &\text{if } p \in II(N)_{even} \setminus (I+J)
        \end{cases}
    \end{align*}
\end{itemize}
\end{remark}
We can simplify the statement from Theorem \ref{theo: Sakata L-function} as follows:
\begin{corollary}\label{cor: Sakata L-functions}
Let $k>1$. Let $F \in S_{2k}^{new}(N)$ be a newform. Let $D$ be a fundamental discriminant with $(-1)^kD>0$ and $(D,N)=1$. Denote by $\tau$ the system of eigenvalues of $F$ under the Atkin-Lehner operators. Assume that $\tau$ satisfies:
\begin{align*}
    \tau(p)=\begin{cases}
        \left( \frac{D}{p}\right) & \text{if } p \in I(N) +II(N)_{odd},\\
        1 & \text{if } p \in II(N)_{even}.
    \end{cases}
\end{align*}
 Let $\kappa_0: II(N) \rightarrow \{ \pm 1 \}$ be defined by:
\begin{align*}
    \kappa_0(p):=\left( \frac{|D|}{p}\right) \quad \forall p \in II(N).
\end{align*}
Then we have:
\begin{align*}
    \frac{|c_{f_{\kappa_0}}(|D|)|^2}{\langle f_{\kappa_0}, f_{\kappa_0} \rangle} =
    2^{\nu(N)}\frac{(k-1)!}{\pi^k} |D|^{k-1/2}\frac{L(F\otimes \chi_D,k)}{\langle F,F\rangle}.
\end{align*}
Here the map $f_{\kappa_0}$ is the unique (up to scalar multplication) half-integral weight modular form in $S_{k+1/2}^{K,new,\kappa_0}(4N)$ which corresponds to $F$ as in Theorem \ref{Sakata-general-decomposition}. Moreover $\nu(N)$ is the number of different prime divisors of $N$.
\end{corollary}
\begin{proof}
 From Lemma \ref{lem: all newforms with positive even values} it follows that for any $F \in S_{2k}^{new}(N)$ with $\tau(p)=1$ for all $p \in II(N)_{even}$, there exists $I,J \subset II(N)^*$ such that
 \begin{align*}
    F \in S^{*, \tilde{\tau}}_{2k}\left(N_1 \prod_{p \in J}p \prod_{p \in II(N)-(I\cup J)} p^{ord_p(N)} \right)|\prod_{p \in I \cup J}R_p \subset S_{2k}(N).
\end{align*}
We now show why the sum on the left hand side of Sakata's original statement reduces to a single coefficient. We first study what it means for a map $\kappa: II(N)\rightarrow \{\pm 1\}$ to satisfy $(\tilde{\tau})_{odd}$ (see definition \ref{def: tau_odd}). For all $p \in II(N)_{odd}$, we have:
    \begin{align*}
        \kappa(p)
        &=\tilde{\tau}(p)\left(\frac{-1}{p}\right)^k \prod_{q \in I+J}\left( \frac{p}{q}\right)
        =\tau(p)\left(\frac{-1}{p}\right)^k \left(\prod_{q \in I+J}\left( \frac{p}{q}\right)\right)^2\\
        &= \tau(p)\left(\frac{-1}{p}\right)^k
        = \left( \frac{|D|}{p}\right)
    \end{align*}
    where we used Lemma \ref{lem: system eigenvalues of twist} to replace $\tilde{\tau}$ with $\tau(p)\prod_{q \in I+J}\left( \frac{p}{q}\right)$ and that by assumption $(-1)^kD>0$.
    This shows that in particular $\kappa_0$ satisfies $(\tilde\tau)_{odd}$. It follows from Lemma \ref{lem: cf-are-zero-except-kappa-0} that for $\kappa\neq \kappa_0$ the coefficient $c_{f_{\kappa}}(|D|)$ is zero. This proves the claim. 
\end{proof}
\begin{remark}
    Note that Kohnen in \cite{Kohnen_1988-remark} and Mao in \cite{Mao} prove similar statements as Corollary \ref{cor: Sakata L-functions}. We chose to exhibit the approach from Ueda and Sakata as they provide an insightful description of newforms for half-integral weight when the level is not square-free, offering the most detailed understanding of the underlying structure.
\end{remark}
For the proof of the main theorem we need that  the product of two $L$-functions of twists of a newform $F$ evaluated at the central value can be expressed as the inner product of $F$ and the cusp from $f_{k,N,D,D_0}$ from definition \ref{def: sec1-f_k,N,D,D_0}. More precisely, we have the following:
\begin{theorem}\label{theo: composite N product l-function inner product}
    Let $k >1$ and $N$ an odd integer. Let $D,D_0$ be two fundamental discriminants such that $(-1)^kD,(-1)^kD_0 >0$ and $(D,N)=(D_0,N)=1$.  We assume that for all primes $p \in II(N)$ we have that $\left( \frac{D}{p}\right)=\left( \frac{D_0}{p}\right)$. Let $F \in S_{2k}^{new}(N)$ be a newform and 
    assume that its system of eigenvalues under the Atkin-Lehner involutions satisfies:
    \begin{align*}
        \tau(p)=\begin{cases}
            \left( \frac{D}{p}\right), &p \in I(N)+II(N)_{odd},\\
            1, &p \in II(N)_{even}. \\
        \end{cases}
    \end{align*}
Then we have that:
\begin{align*}
    |i_{N}c_{k,D} &\langle F, f_{k,N,D,(-1)^kD_0} \rangle|^2\\
    &=\left(\frac{2^{\nu(N)}(k-1)!}{\pi^k} \right)^2|DD_0|^{k-1/2}L(F \otimes \chi_D,k)L(F \otimes \chi_{D_0},k)
\end{align*}
Here $\nu(N)$ denotes the number of different prime divisors of N.
In particular we have:
\begin{align*}
    L(F \otimes \chi_D,k)
   L(F \otimes \chi_{D_0},k)=0 \Longleftrightarrow \langle F, f_{k,N,D,(-1)^kD_0}\rangle=0.
\end{align*}
\end{theorem}

\begin{proof}
We will express the $D_0$th coefficient of the $D$th Shintani lift of $F$ in two different ways.
From Lemma \ref{lem: Coeff of Shintani lift}, we know that $|D_0|$th coefficient of the $D$th Shintani lift of $F$ is given by
\begin{align}\label{eq: 1}
c_{|D_0|}(F|S_D^*)=i_{N}c_{k,D} \langle F, f_{k,N,D,(-1)^kD_0} \rangle
\end{align}
We now express $c_{|D_0|}(F|S_D^*)$ in terms of coefficients of modular form of half-integral weight. From Lemma \ref{lem: all newforms with positive even values}, we know that
there exists $I,J \subset II(N)^*$ such that
 \begin{align*}
    F \in S^{*, \tilde{\tau}}_{2k}\left(N_1 \prod_{p \in J}p \prod_{p \in II(N)-(I\cup J)} p^{ord_p(N)} \right)|\prod_{p \in I \cup J}R_p \subset S_{2k}(N).
\end{align*}
Denote with $\{\kappa: II(N) \rightarrow \{ \pm 1 \}\}$ the set of all functions satisfying $(\tilde{\tau})_{odd}$. In particular, we set $\kappa_0: II(N) \rightarrow \{ \pm 1 \}$ to be:
\begin{align*}
    \kappa_0(p):=\left( \frac{|D|}{p}\right), \quad \forall p \in II(N)_{odd}.
\end{align*}
As shown in the proof of Theorem \ref{cor: Sakata L-functions} $\kappa_0$ satisfies $(\tilde{\tau})_{odd}$.
 We can decompose the space of half-integral weight forms with the same system of Hecke eigenvalue as $F$ into subspaces corresponding to their behavior under the twisting operators $R_p$:
    \begin{align*}
        S^{K,new}_{k+1/2}(N;F)=
        \bigoplus_{\substack{\kappa: II(N) \rightarrow \{ \pm 1 \}\\ \kappa \text{ satisfies }(\tilde{\tau})_{odd}}} S^{K,\kappa}_{k+1/2}(N) \cap S^{K}_{k+1/2}(N;F)
    \end{align*}
    Let $f_{\kappa}=\sum_{n=1}^{\infty}c_n(f_{\kappa})q^n$ be the (up to scalar multplication) unique form in $S_{k+1/2}^{K,\kappa}(N)$. By Lemma \ref{lem: self-adjoint-twisting-operator}, for $\kappa_1 \neq \kappa_2$, we have that  $\langle f_{\kappa_1},f_{\kappa_2} \rangle =0 $. Therefore they form an orthogonal basis. The Shimura lift satisfies:
\begin{align*}
    f_{\kappa}|\Sh_D = c_{|D|}(f_{\kappa})F.
\end{align*}
Thus we can write:
\begin{align*}
    F |S_D^* &= \sum_{\substack{\kappa: II(N) \rightarrow \{ \pm 1 \}\\ \kappa \text{ satisfies }(\tilde{\tau})_{odd}}} \frac{\langle F |S_D^*, f_{\kappa} \rangle}{\langle f_{\kappa}, f_{\kappa}\rangle} f_{\kappa}
    = \sum_{\substack{\kappa: II(N) \rightarrow \{ \pm 1 \}\\ \kappa \text{ satisfies }(\tilde{\tau})_{odd}}}
    \frac{\langle F , f_{\kappa} | \Sh_D \rangle}{\langle f_{\kappa}, f_{\kappa}\rangle} f_{\kappa}\\
    &= \sum_{\substack{\kappa: II(N) \rightarrow \{ \pm 1 \}\\ \kappa \text{ satisfies }(\tilde{\tau})_{odd}}} \frac{\langle F , F \rangle}{\langle f_{\kappa}, f_{\kappa}\rangle} \overline{c_{|D|}(f_{\kappa})} f_{\kappa}
    = \frac{\langle F , F \rangle}{\langle f_{\kappa_0}, f_{\kappa_0}\rangle} \overline{c_{|D|}(f_{\kappa_0})}f_{\kappa_0}
\end{align*}
Here we used Lemma \ref{lem: cf-are-zero-except-kappa-0} in the last line. Thus
\begin{align}\label{eq: 2}
   c_{|D_0|}(F|S_D^*)=\overline{c_{|D|}(f_{\kappa_0})}\frac{\langle F,  F\rangle}{\langle f_{\kappa_0}, f_{\kappa_0}\rangle} c_{|D_0|}(f_{\kappa_0})
\end{align}
Combining equations \eqref{eq: 1} and \eqref{eq: 2} we get that
\begin{align*}
    \overline{c_{|D|}(f_{\kappa_0})}c_{|D_0|}(f_{\kappa_0})\frac{\langle F,  F \rangle}{\langle f_{\kappa_0}, f_{\kappa_0}\rangle} =i_{N}c_{k,D} \langle F, f_{k,N,D,(-1)^kD_0} \rangle
\end{align*}
Theorem \ref{cor: Sakata L-functions} implies that
\begin{align*}
    &\frac{|c_{|D|}(f_{\kappa_0})c_{|D_0|}(f_{\kappa_0})|^2}{\langle f_{\kappa_0}, f_{\kappa_0} \rangle^2}\\
    &=\left(\frac{2^{\nu(N)}(k-1)!}{\pi^k} \right)^2|DD_0|^{k-1/2}\frac{L(F\otimes \chi_D,k)L(F \otimes \chi_{D_0},k)}{\langle F,F\rangle^2}
\end{align*}
Thus finally:
\begin{align*}
    |i_{N}c_{k,D} &\langle F, f_{k,N,D,(-1)^kD_0} \rangle|^2\\
    =
    &\left(\frac{2^{\nu(N)}(k-1)!}{\pi^k} \right)^2|DD_0|^{k-1/2}L(F\otimes \chi_D,k)L(F\otimes \chi_{D_0},k)
\end{align*}
This proves the statement.
\end{proof}
\subsection{Modularity errors of the holomorphic and non-holomorphic Eichler integrals}
In section \ref{sec: A locally harmonic Maass form.} we will construct a locally harmonic Maass form which essentially splits into three parts: a local polynomial and the holomorphic and non-holomorphic Eichler integral of the function $f_{k,N,D,D_0}$. In this section we define the Eichler integrals and compute their modularity errors under a matrix $\gamma \in \Gamma_0(N)$. We then explain how Eichler-Shimura theory implies that for any nonzero cusp form there always exists a matrix $\gamma \in \Gamma_0(N)$ such that the error term is nonzero. 

\medskip

For a cusp form $F=\sum\limits_{n=1}^{\infty} a_n q^n \in S_{2k}(N)$ we define the holomorphic Eichler integral to be:
    \begin{align}\label{eq: holomorphic Eichler}
    \mathcal{E}_{F}(z):= -\frac{(2 \pi i)^{2k-1}}{(2k-2)!}\int_z^{i \infty}F(w)(z-w)^{2k-2}dw= \sum_{n=1}^{\infty}a_nn^{1-2k}q^n.
\end{align}
Moreover, the non-holomorphic Eichler integral is given by:
    \begin{align}\label{eq: non-holomorphic eichler}
    F^*=(2i)^{1-2k}\int_{-\overline{z}}^{i \infty}\overline{F(-\overline{w})}(w+z)^{2k-2}dw.
\end{align}
Under the modularity operator, the holomorphic and non-holomorphic Eichler integral have the following error term:
\begin{lemma}\label{lem: error mod eichler integrals}
    Let $F \in S_{2k}(N)$ be a cusp form. Let $ \gamma \in \Gamma_0(N)$.
    \begin{itemize}
        \item[(i)] The modularity error of $\mathcal{E}_F$ is given by: 
        \begin{align*}
\mathcal{E}_F|_{2-2k} (\gamma-1) (z)= -\frac{(2 \pi i )^{2k-1}}{(2k-2)!} \int_{i \infty}^{\gamma^{-1}(i \infty)} F(w)(z-w)^{2k-2}dw   \end{align*}
        \item[(ii)] The modularity error of $F^*$ is given by:
        \begin{align*}
            F^*|_{2-2k}(\gamma-1)(z)= (2i)^{1-2k}\int_{i \infty}^{-\gamma^{-1}(-i \infty)}\overline{F(-\overline{w})}(w+z)^{2k-2}dw
        \end{align*}
    \end{itemize}
\end{lemma}
\begin{proof}
    The proof follows from a straight forward computation, in which on substitutes $w$ with $\gamma^{-1}(w)$ and $-(\gamma^{-1}(-w))$ respectively  and uses that $F$ is modular of weight $2k$. 
\end{proof}

For the proof of the Theorem \ref{theo: main prod L-functions} we need that if $F$ is a nonzero function, then there exists a matrix $\gamma\in \Gamma_0(N)$ such that 
\begin{align}\label{eq: mod eichler error raw}
\begin{split}
        z^{2k-2}N^{k-1}\left(\int_{i \infty}^{\gamma (i \infty)} \right.& F( w) \left(\frac{-1}{Nz}-w\right)^{2k-2} d w \\
& \left. - \int_{i \infty}^{-(\gamma (-i\infty))}F(w)\left(w+\frac{-1}{Nz}\right)^{2k-2}dw\right)
\end{split}
\end{align}
is not zero for all $z \in \mathbb{H}$. To prove this statement, we introduce Eicher-Shimura Theory. We closely follow the exposition in \cite{Antoniadis1992}. 

We denote by $V_{\C}$ the space of polynomials with complex coefficients of degree at most $2k-2$, i.e:
\begin{align*}
    V_{\C}:=\{ P(z)  \in \C[z] : \deg(P(z)) \leq 2k-2\}.
\end{align*}
We let $SL_2(\Z)$ act on the vector space $V_{\C}$ by the usual action: for $\gamma =\begin{pmatrix}
    \gamma_{11} & \gamma_{12} \\ \gamma_{21}& \gamma_{22}
\end{pmatrix} \in SL_2(\Z)$, we set:
\begin{align*}
    \gamma  \cdot P (z) =P|_{2-2k}\gamma ^{-1} (z)=P\left( \frac{\gamma_{22}z-\gamma_{12}}{-\gamma_{21}z+\gamma_{11}}\right)\left( -\gamma_{21}z+\gamma_{11}\right)^{2k-2}
\end{align*}

We look at functions mapping the space $\Gamma_0(N)$ to polynomials, i.e. the set 
\begin{align*}
   \{ \Gamma_0(N) \rightarrow V_{\C}\}.
\end{align*}
Note that for a fixed form $F$, equation \eqref{eq: mod eichler error raw} defines such a map. We decompose the parabolic homology group of the space $\{ \Gamma_0(N) \rightarrow V_{\C}\}$ into two subspaces. Recall that a matrix in $SL_2(\Z)$ is parabolic if its trace is equal to $\pm2$. We define the space of  $1$-cocycles by:
\begin{align*}
    C_{\C}:=\{ &\sigma: \Gamma_0(N) \rightarrow V_{\C}: \forall \gamma _1,\gamma _2 \in \Gamma_0(N): \sigma(\gamma _1\gamma _2)=\sigma(\gamma _1) + \gamma _1 \cdot \sigma(\gamma _2) \text{ and }  \\
    &\forall \text{ parabolic } \gamma  \in \Gamma_0(N), \exists P(z) \in V_{\C} \text{ such that } \sigma(\gamma )=P(z)-(\gamma  \cdot P)(z) \}.
\end{align*}
Moreover we define the space of $1$-coboundaries as:
\begin{align*}
    B_{\C}=\{ \sigma: \Gamma_0(N) \rightarrow V_{\C}: &\exists P(z) \in V_{\C} \\
    &\text{ such that } \forall \gamma  \in \Gamma_0(N): \sigma(\gamma )=P(z)-(\gamma \cdot P)(z) \}.
\end{align*}
The parabolic cohomology group is given by:
\begin{align*}
    H^1_P(\Gamma_0(N), V_{\C})=\frac{C_{\C}}{B_{\C}}.
\end{align*}

We want to decompose the space $H^1_P(\Gamma_0(N), V_{\C})$ into $\pm 1$ eigenspaces under the involution given by the matrix
\begin{align*}
\epsilon=\begin{pmatrix}
    -1 & 0 \\ 0 & 1
\end{pmatrix}.
\end{align*}
Let $\sigma: \Gamma_0(N) \rightarrow V_{\C}$ be any map sending matrices to polynomials, then we set
\begin{align*}
    \tilde{\sigma}(\gamma )=(\epsilon \cdot \sigma (\epsilon \gamma  \epsilon))x= \sigma (\epsilon \gamma  \epsilon)(-x)
\end{align*}
We can define an involution on $H^1_P(\Gamma_0(N), V_{\C})$ as follows:
\begin{align*}
    [\sigma] \mapsto [\Tilde{\sigma}]
\end{align*}
We decompose $H^1_P(\Gamma_0(N), V_{\C})$ into its eigenspaces under this involution:
\begin{align*}
    H^1_P(\Gamma_0(N), V_{\C}) = H^1_P(\Gamma_0(N), V_{\C})^+ \oplus H^1_P(\Gamma_0(N), V_{\C})^-. 
\end{align*}
According to the Eichler-Shimura isomorphism \cite{Shimura-Arithmetic-Automorphic-Forms}, there is an isomorphism between the $\R$-vectorspaces $S_{2k}(\Gamma_0(N))$ and $H_p^1(\Gamma_0(N), V_{\R})$. (Where  $V_{\R}$ is the space of polynomials of degree at most $2k-2$ with coefficients in $\R$.) Antoniadis \cite{Antoniadis1992} states the following as a corollary of the Eichler-Shimura isomorphism.
\begin{proposition}[Proposition 1 in \cite{Antoniadis1992}]\label{prop: Antonidas}
Let $F \in S_{2k}(\Gamma_0(N))$. Set:
\begin{align*}
\pi_F(\gamma )=\int_0^{\gamma (0)} F(w)(z-w)^{2k-2}dw, \text{ and } \tilde{\pi}_F(\gamma )=\int_0^{-\gamma (0)} F(w)(z+w)^{2k-2}dw.    
\end{align*}
    The following maps define isomorphisms:
\begin{align*}
    \Phi^+&: S_{2k}(\Gamma_0(N)) \rightarrow H^1_P(\Gamma_0(N), V_{\C})^+, \quad F \mapsto \frac{1}{2i}( \pi_F + \tilde{\pi}_F), \text{ and } \\
    \Phi^-&: S_{2k}(\Gamma_0(N)) \rightarrow H^1_P(\Gamma_0(N), V_{\C})^-, \quad F \mapsto \frac{1}{2}( \pi_F - \tilde{\pi}_F).
\end{align*}
\end{proposition}
One immediate corollary from the proposition is:
\begin{corollary}\label{cor: Antonidas}
    Let $F \in S_{2k}(\Gamma_0(N))$ be any nonzero function. Then there exists a matrix $\gamma  \in \Gamma_0(N)$ such that
    \begin{align*}
    \int_0^{\gamma (0)} F(w)(z-w)^{2k-2}dw-\int_0^{-\gamma (0)} F(w)(z+w)^{2k-2}dw 
\end{align*}
is not the zero polynomial in $z$.
\end{corollary}
As mentioned above for the proof of Theorem \ref{theo: main prod L-functions} we need a slight reformulation of the statement:
\begin{lemma}\label{lem: transform to Antoniadis function}
Let $F \in S_{2k}(\Gamma_0(N))$ be a newform. Then there exists a matrix $\gamma  \in \Gamma_0(N)$ such that
    \begin{align*}
    z^{2k-2}N^{k-1}\left(\int_{i \infty}^{\gamma (i \infty)} \right. & F( w) \left(\frac{-1}{Nz}-w\right)^{2k-2} d w \\
& \left.- \int_{i \infty}^{-(\gamma (-i\infty))}F(w)\left(w+\frac{-1}{Nz}\right)^{2k-2}dw\right)
\end{align*}
is not the zero polynomial.
\end{lemma}
    \begin{proof}
\begingroup
\allowdisplaybreaks
Let $w_N$ is the eigenvalue of $F$ under the Fricke involution, let $\gamma =\begin{pmatrix}
    \gamma_{11} & \gamma_{12} \\ \gamma_{21}& \gamma_{22}
\end{pmatrix}$ and set $\gamma _0=\begin{pmatrix}
    \gamma_{22} & -\gamma_{21}/N \\ -\gamma_{12}N&\gamma_{11}
\end{pmatrix}$.
A direct computation in which one substitutes $w=-1/(Nw')$ and uses that $F$ is an eigenfunction under the Fricke involution yields that: 
\begin{align*}
z^{2k-2}N^{k-1}&\int_{i \infty}^{\gamma (i \infty)} F( w) \left(\frac{-1}{Nz}-w\right)^{2k-2} d w\\
- &z^{2k-2}N^{k-1}\int_{i \infty}^{-(\gamma (-i\infty))}F(w)\left(w+\frac{-1}{Nz}\right)^{2k-2}dw
 \\
&=w_N  \left(\int_{0}^{\gamma _0(0)} F(w) \left(z-w\right)^{2k-2}d w
-\int_{0}^{-\gamma _0(0)}F(w)(w+z)^{2k-2}dw \right). 
\end{align*}
 The result follows then from Corollary \ref{cor: Antonidas}.
\endgroup
\end{proof}

\section{Locally harmonic Maass forms}\label{sec: A locally harmonic Maass form.}
The goal of this section is to define a locally harmonic Maass form which detects if the product of two twisted $L$-functions evaluated at the central value is zero. 

\subsection{Definition of locally harmonic Maass forms}
We start by introducing locally harmonic Maass forms. For $k>1$ and a nonsquare positive discriminant $D$, Bringmann, Kane and Kohnen  define and study in \cite{B-K-K} the following function:
\begin{align*}
    \mathcal{F}_{1-k,D}(z):=
    \frac{D^{\frac{1}{2}-k}}{\binom{2k-2}{k-1}2\pi}
    \sum_{Q=[a,b,c]\in \mathcal{Q}_D} \text{sgn}(Q_{z})Q(z,1)^{k-1}\beta\left(\frac{Dy^2}{|Q(z,1)|^2};k-\frac{1}{2},\frac{1}{2}\right)
\end{align*}
Here $z=x+iy \in \mathbb{H}$. Moreover for $s,w \in \C$ with $\re(s),\re(w)>0$, we set:
\begin{align}\label{eq: def beta}
    \beta(v;s,w):=\int_0^v u^{s-1}(1-u)^{w-1}du.
\end{align}
For $z \in \mathbb{H}$, we define:
\begin{align}\label{eq: def Q_z}
    Q_z:=\frac{1}{\im(z)}(a|z|^2+b\re(z)+c).
\end{align}
Recall that
\begin{align*}
    \mathcal{Q}_D:=\{[a,b,c] : b^2-4ac=D\} 
\end{align*}
The function $\mathcal{F}_{1-k,D}$ is a weight $2-2k$ locally harmonic Maass form for $SL_2(\Z)$ with exceptional set
\begin{align*}
    E_D:=\left\{ z \in \mathbb{H}| \quad \exists Q \in \mathcal{Q}_D \text{ such that } Q_{z}=0 \right\}.
\end{align*}
in the following sense:
\begin{definition}
A function $\mathcal{F}:\mathbb{H} \rightarrow \C$ is called a locally harmonic Maass form of weight $2-2k$ for the congruence subgroup $\Gamma$ and with exceptional set $E$ if:
\begin{itemize}
    \item[(i)] For every $\gamma \in \Gamma$, for every $z \in \mathbb{H}$: $\mathcal{F}|_{2-2k}\gamma(z)=\mathcal{F}(z)$.
    \item[(ii)] For every $z \in \mathbb{H} \setminus E$, there is a neighborhood $\mathcal{U}$ of $z$ in which $\mathcal{F}$ is real analytic and $\Delta_{2-2k}(\mathcal{F})=0$. 
    
    \item[(iii)] For every $z \in E$, we have:
    \begin{align*}
        \mathcal{F}(z)=\frac{1}{2}\lim_{w \rightarrow 0^+} (\mathcal{F}(z +iw)+\mathcal{F}(z -iw)), \quad w \in \R.
    \end{align*}
    \item[(iv)] The function $\mathcal{F}$ exhibits at most polynomial growth towards $ i \infty$.
\end{itemize}
\end{definition}
Here the weight $k$ Laplacian is given by:
    \begin{align*}
        \Delta_k:=-y^2 \left(\frac{\partial^2}{\partial x^2}+\frac{\partial^2}{\partial y^2}\right)+iky\left(\frac{\partial}{\partial x}+i\frac{\partial}{\partial y}\right).
\end{align*}

Bringmann, Kane and Kohnen show that the function $\mathcal{F}_{1-k,D}$ splits into three parts: a holomorphic part, a non-holomorphic part and a local polynomial. They also investigate its behavior under the Hecke operator. We will study a form similar to $\mathcal{F}_{1-k,D}$ with analogous properties. 
\subsection{The function $\mathcal{F}_{1-k,N,D,D_0}$}
 Let $k>1$, $N$ an integer and $D,D_0$ two fundamental discriminants such that $DD_0>0$ is not a perfect square. We set
\begin{align}\label{def: Flm}
\begin{split}
        \mathcal{F}_{1-k,N,D,D_0}(z):=\frac{(DD_0)^{\frac{1}{2}-k}}{\binom{2k-2}{k-1}2\pi}\sum_{Q \in \mathcal{Q}_{N,DD_0}}&\chi_{D_0}(Q)\text{sgn}(Q_z)\\
    &Q(z,1)^{k-1}\beta\left( \frac{DD_0y^2}{|Q(z,1)|^2};k-\frac{1}{2},\frac{1}{2}\right).
\end{split}
\end{align}
This function has been introduced in \cite{males-mono-rollen-wagner}. Note that we use a different normalization than there.
In the following, we describe some properties of the function $\mathcal{F}_{1-k,N,D,D_0}$. For the proofs, which  are in general very similar to the ones from \cite{B-K-K}, we refer to Appendix \ref{app: locally harmonic Maass Form}.
Let
    \begin{align*}
        E_{N,DD_0}:= \{ z \in \mathbb{H} |\quad \exists Q \in \mathcal{Q}_{N,DD_0}: Q_z=0 \}.
    \end{align*}
We prove in Theorem \ref{theo: Flm is locally harmonic} that the function $\mathcal{F}_{1-k,N,D,D_0}$ is a locally harmonic Maass form of weight $2-2k$ for $\Gamma_0(N)$ with exceptional set $E_{N,DD_0}$.

We now explain how the function $\mathcal{F}_{1-k,N,D,D_0}$ detects the vanishing of the product of two twisted $L$-functions. Recall from Theorem \ref{theo: composite N product l-function inner product}, that for $F$ and $D,D_0$ as in Theorem \ref{theo: main prod L-functions}, we have that
\begin{align}\label{eq: <F,f>=0 <=> LL=0}
    L(F \otimes \chi_D,k)
   L(F \otimes \chi_{D_0},k)=0 \Longleftrightarrow \langle F, f_{k,N,D,D_0}\rangle=0.
\end{align}
Here 
    \begin{align}\label{eq: def sec3-f_k,N,D,D_0}
        f_{k,N,D,D_0}(z):=\frac{(DD_0)^{k-\frac{1}{2}}}{\binom{2k-2}{k-1}\pi}\sum_{Q \in \mathcal{Q}_{N,DD_0}}\chi_{D_0}(Q) Q(z,1)^{-k} \in S_{2k}(N).
    \end{align}
\begin{remark}
    Here we multiply the function defined in definition \ref{def: sec1-f_k,N,D,D_0} by $\frac{(DD_0)^{k-\frac{1}{2}}}{\binom{2k-2}{k-1}\pi}$.
\end{remark}
From Theorem \ref{theo: splitting Flm} it follows that the function $\mathcal{F}_{1-k,N,D,D_0}$ can be written as the sum over the holomorphic Eichler integral of $f_{k,N,D,D_0}$, the non-holomorphic Eicher integral of $f_{k,N,D,D_0}$ and a local polynomial:
\begin{align}\label{eq: split F}
\begin{split}
        \mathcal{F}_{1-k,N,D,D_0}(z)=P_{k,N,D,D_0}(z)&+(DD_0)^{\frac{1}{2}-k}f^*_{k,N,D,D_0}(z)\\
        &-(DD_0)^{\frac{1}{2}-k}\frac{(2k-2)!}{(4\pi)^{2k-1}}
\mathcal{E}_{f_{k,N,D,D_0}}(z).
\end{split}
    \end{align}
Here
\begin{align*}
    \mathcal{P}_{k,N,D,D_0}(z)=c_{k,N,D,D_0,\infty} + c_{k,DD_0}\sum_{\substack{Q=[a,b,c] \in \mathcal{Q}_{N,DD_0}\\ a< 0 < Q_z}} \chi_{D_0}(Q)Q(z,1)^{k-1}.
\end{align*}
With
\begin{align}\label{eq: def c_infty c_2}
    c_{k,N,D,D_0,\infty}&:=\frac{-2^{3-2k}}{(2k-1)\binom{2k-2}{k-1}}\sum_{\substack{a \in \N, \\N |a}} a^{-k}\sum_{\substack{b \mod 2a,\\b^2\equiv DD_0 \mod 4a}}\chi_{D_0}\left(\left[a,b\frac{b^2-DD_0}{4a}\right]\right)\\
    c_{k,DD_0}&:=(-1)^k2^{3-2k}(DD_0)^{\frac{1}{2}-k}
\end{align}

Recall that we defined the local polynomial on the rational as
\begin{align}\label{eq: def-3 Plx}
    P_{k,N,D,D_0}(x)=c_{k,N,D,D_0,\infty} + c_{k,DD_0}\sum_{\substack{Q=[a,b,c] \in \mathcal{Q}_{N,DD_0}\\ a< 0 < Q(x,1)}} \chi_{D_0}(Q)Q(x,1)^{k-1}.
\end{align}

\begin{remark}
We show in Lemma \ref{lem: def lim Plz =Plx} that $\lim_{z \rightarrow x}\mathcal{P}_{k,N,D,D_0}(z)=P_{k,N,D,D_0}(x)$.
\end{remark}
When $\dim(S_{2k}(N))=1$ then $f_{k,N,D,D_0}=\alpha F$ for some $\alpha \in \C$ and therefore:
\begin{align}\label{eq: prod L-function equivalences dim =1}
\begin{split}
    L(F \otimes \chi_D,k)
   L(F \otimes \chi_{D_0},k)=0 &\Longleftrightarrow \langle F, f_{k,N,D,D_0}\rangle=0\\
   &\Longleftrightarrow  f_{k,N,D,D_0}=0\\
   &\Longleftrightarrow \mathcal{F}_{1-k,N,D,D_0}(z)=\mathcal{P}_{k,N,D,D_0}(z) 
   \end{split}
\end{align}
This is not true any longer if $\dim(S_{2k}(N))>1$. Using Hecke operators, we can prove an analogous statement in this case. We therefore study Hecke operators in the next section.
\subsection{Hecke operators}\label{sec: 3-Hecke Operators}
In this section we explain how to pick the Hecke operators as in the Theorem \ref{theo: main prod L-functions} and show how one can lift the Hecke operator from $f_{k,N,D,D_0}$ to $\mathcal{F}_{1-k,N,D,D_0}$. This is based on the ideas of \cite{Kong-2017} as well as \cite{males-mono-rollen-wagner}.

We start by explaining how we can recover a statement similar to equation \eqref{eq: prod L-function equivalences dim =1} when $\dim(S_{2k}(N))>1$ by choosing appropriate Hecke operators. The main idea is to project $f_{k,N,D,D_0}$ to the subspace spanned by a newform $F$ as follows.

Let $F =\sum_{n=1}^{\infty} c_n q^n\in S_{2k}(N)$ be a normalized newform. Then there exists a basis $\{F_n\}_n^m$ consisting of eigenforms with $F=F_1$. For each $n \neq 1$ there exists a prime $p_n$ such that $F_n$ and $F$ have different eigenvalues under the Hecke operator $T_{p_n}$. Let $a_{p_n}$ be the eigenvalue of $F_n$ under $T_{p_n}$. Then 
\begin{align*}
    F_n |_{2k} (T_{p_n}-a_{p_n})=0, \text{ and } F|_{2k} (T_{p_n}-a_{p_n})=(c_{p_n}-a_{p_n})F \neq 0.
\end{align*}
Hence
\begin{align}\label{rem: construction Hecke like operator}
    \mathbb{T}:=\prod_{n=2}^m(T_{p_n}-a_{p_n}): S_{2k}(N) \rightarrow \text{span}_{\C}\{F\}
\end{align}
is a non-zero map. 
Let $G= \sum_{n=1}^m \alpha_n F_n \in S_{2k}(N)$ for some complex numbers $\alpha_n$. Then
\begin{align*}
    G|_{2k} \mathbb{T}=\alpha_1 \prod_{n=2}^m(c_{p_n}-a_{p_n})F.
\end{align*}
Thus we have
\begin{align*}
    \langle F,G \rangle=\overline{\alpha_1} |F|^2 \text{ and }\langle F, G|_{2k}\mathbb{T} \rangle= \overline{\alpha_1} \prod_{n=2}^n\overline{(c_n-a_{p_n})}|F|^2
\end{align*}
In particular we have:
\begin{align*}
    \langle F,G \rangle=0 \Leftrightarrow \alpha_1=0 \Leftrightarrow \langle F, G|_{2k}\mathbb{T} \rangle= 0.
\end{align*}
\begin{remark}
    Note that it might be possible, that for some prime $p$ several eigenfunctions have the same eigenvalue, distinct than the one from $F$, under $T_p$. In this case, it is possible to construct an operator $\mathbb{T}$ with the same property consisting of fewer than $m-1$ factors. This is the case in the example from subsection \ref{subsec: Ex. S_4(25)}.
\end{remark}

We now explain how to lift Hecke operators from $f_{k,N,D,D_0}$ to $\mathcal{F}_{1-k,N,D,D_0}$.

\begin{lemma}\label{lem: lifting Hecke operator}
Let $k>0$ and let $D,D_0$ be two fundamental discriminants with $(-1)^kD,(-1)^kD_0>0$ and such that $DD_0$ is not a perfect square.

Let $\mathbb{T}=\prod_i(T_{p_i}-a_{p_i})$ be a product of Hecke operators and denote $\tilde{\mathbb{T}}=\prod_i(T_{p_i}-p^{1-2k}a_{p_i})$. Then
\begin{align*}
   \mathcal{F}_{1-k,N,D,D_0}|_{2-2k} \tilde{\mathbb{T}}&=\mathcal{P}_{k,N,D,D_0}|_{2-2k}\tilde{\mathbb{T}}+(DD_0\prod_ip_i^2)^{\frac{1}{2}-k}(f_{k,N,D,D_0}|_{2k}\mathbb{T})\\
    &-(DD_0\prod_ip_i^2)^{\frac{1}{2}-k}\frac{(2k-2)!}{(4\pi)^{2k-1}}\mathcal{E}_{f_{k,N,D,D_0}|_{2k}\mathbb{T}}.
\end{align*}
In particular we have
    \begin{align*}
        f_{k,N,D,D_0}|_{2k}\mathbb{T}=0 \Leftrightarrow
        \mathcal{F}_{1-k,N,D,D_0}|_{2-2k}\tilde{\mathbb{T}}(z)&=\mathcal{P}_{k,N,D,D_0}|_{2-2k}\tilde{\mathbb{T}}(z).
    \end{align*}
    for all $ z\in \mathbb{H}\setminus E_{N,DD_0}$.
\end{lemma}

\begin{proof}
A direct calculations shows that for any form $F\in S_{2k}$:
\begin{align*}
   \mathcal{E}_F|_{2-2k} (T_p-p^{1-2k}a_p)&=p^{1-2k}\mathcal{E}_{F|_{2k}(T_p-a_p)}\\
    F^*|_{2-2k} (T_p-p^{1-2k}a_p)&=p^{1-2k}(F|_{2k}(T_p-a_p))^*.
\end{align*}
\end{proof}

We end this section by mentioning the behavior of $\mathcal{F}_{1-k,N,D,D_0}$ under the Fricke involution. In Lemma \ref{lem: Flm Fricke involution} we prove that 
\begin{align}\label{eq: Flm Fricke involution}
\mathcal{F}_{1-k,N,D,D_0}\Big\vert_{2-2k}\left(W_N - \left(\frac{D_0}{N}\right) \text{ id} \right) = 0.
\end{align}

\section{The local polynomial}\label{sec: local polynomial}
Before proving the main theorem, we describe some properties of the non-constant part of the local polynomial on the rational numbers $\Q$. This section has two goals: we want to define $\lim_{z \rightarrow x}\mathcal{P}_{k,N,D,D_0}(z)$ and show that this definition is the only sensible one. Moreover we explain how it is possible that in some cases $P_{k,N,D,D_0}(x)$ is constant for all $x \sim_{\Gamma_0(N)}0$.

To simplify the notation for $\Delta:=DD_0$  we set:
\begin{align*}
    \mathcal{Q}_{N,\Delta}(x):=\{ Q \in \mathcal{Q}_{N, \Delta} : a<0<Q(x,1)\}.
\end{align*}
We start by defining the limit $\lim_{z \rightarrow x} \mathcal{P}_{k,N,D,D_0}$ in the following way:
\begin{lemma}\label{lem: def lim Plz =Plx}
We define $\lim_{z \rightarrow x} \mathcal{P}_{k,N,D,D_0}(z):=\lim_{y \rightarrow 0^+}\mathcal{P}_{k,N,D,D_0}(x+iy)$. Then we have that:
    \begin{align*}
        \lim_{z \rightarrow x} \mathcal{P}_{k,N,D,D_0}(z)=P_{k,N,D,D_0}(x).
    \end{align*}
\end{lemma}
\begin{proof}
For $x,y \in \R$ with $y \geq 0$, we set
\begin{align*}
    \mathcal{Q}_{N,DD_0}(x+iy):= \{ Q \in \mathcal{Q}_{N,DD_0}: a<0<a(x^2+y^2)+bx+c\}
\end{align*}

Let $Q \in \mathcal{Q}_{N,DD_0}$ with $a<0$. Let $x, y_1,y_2 \in \R$ with $0 \leq y_1 <y_2$, then we have:
\begin{align*}
    a(x^2+y_1^2)+bx+c > a(x^2+y_2^2)+bx+c.
\end{align*}
Therefore if $Q_{x+iy_2}>0$, then $Q_{x+iy_1}>0$. This shows that $\mathcal{Q}_{N,DD_0}(x+iy_2) \subset \mathcal{Q}_{N,DD_0}(x+iy_1)$. For $x \in \Q$, we know that $\mathcal{Q}_{N,DD_0}(x)$ is finite. Therefore, there exists $y_0>0$ such that for all $0 \leq y < y_0$: 
\begin{align*}
    \mathcal{Q}_{N,DD_0}(x)=\mathcal{Q}_{N,DD_0}(x+iy)
\end{align*}
and the claim follows. 
\end{proof}
We will justify in a moment, why we do not allow to take the limit along other directions. We first make the following observation:

Let $p_{k,N,D,D_0}$ be the non-constant part of $P_{k,N,D,D_0}(x)$, i.e. for $x \in \Q$ we set:
\begin{align}\label{eq: non constant part local polynomial}
         p_{k,N,D,D_0}(x):=\sum_{\substack{Q \in \mathcal{Q}_{N,DD_0}\\a<0<Q(x,1)}}\chi_{D_0}(Q)Q(x,1)^{k-1}.
     \end{align}
Then $P_{k,N,D,D_0}(x)=c_{k,N,D,D_0,\infty}+c_{k,DD_0}p_{k,N,D,D_0}(x)$. One immediate corollary from Theorem \ref{theo: main prod L-functions} is that:
\begin{corollary}
    Assume that $\dim(S_{2k}(N))=1$, and $F \in S_{2k}(N)$ is nonzero. Let $D,D_0$ be two discriminants as in Theorem \ref{theo: main prod L-functions} and such that $L(F \otimes \chi_D,k)=0$ then $p_{k,N,D,D_0}(x)$ is constant for all $ x \sim_{\Gamma_0(N)} 0$.
\end{corollary}
One might naively guess, that the reason for this is because for all $x_0 \in \Q$, the sum $\sum_{\substack{Q \in \mathcal{Q}_{N,DD_0}\\a<0<Q(x_0,1)}}\chi_{D_0}(Q)Q(X,1)^{k-1}$ is a constant polynomial in $X$ independent of $x_0$. However, one can check computationally that this is not the case.
\begin{example}
When $N=9$, $D=172$ and $D_0=13$ then we can check that for different values of $x_0$ the local polynomial $p_{k,N,D,D_0}(x_0)$ is constant, while the polynomial in $X$ given by $\sum_{\substack{Q \in \mathcal{Q}_{N,DD_0}\\a<0<Q(x_0,1)}}\chi_{D_0}(Q)Q(X,1)^{k-1}$ is not. Some numerical values are in the table below:
\begin{table}[H]
    \centering
    \begin{tabular}{c|cc}
    $x_0$ & $\sum_{\substack{Q \in Q_{9,172\cdot13}\\a<0<Q(x_0,1)}}\chi_{13}(Q)Q(X,1)^{k-1}$ & $p_{2,9,172,13}(x_0)$\\
    \hline
    0 & $-3024X^2 + 336$ & 336 \\
    1/2 & $-12096X^2 + 12096X - 2688$ & 336 \\
    1/3 & $-12096X^2 + 8064X - 1008$ & 336 \\
    1/5 & $-75600X^2 + 30240X - 2688$ & 336 \\
    1/7 & $-148176X^2 + 42336X - 2688$ & 336 \\
    1/9 & $-27216X^2 + 6048X$ & 336 \\
    \end{tabular}
    \caption{ The polynomial $\sum_{\substack{Q \in Q_{9,172\cdot 13}\\a<0<Q(x_0,1)}}\chi_{13}(Q)Q(X,1)$ and $p_{2,9,172,13}(x_0)$ }
    \label{tab: comparison p_{k,N,D,D_0}, N=9 D=172,D_0=13}
\end{table}
\end{example}
Why can we still expect that $p_{k,N,D,D_0}(x_0)$ is constant in these cases? The answer to this question as well as the reason for the definition of limit $\lim_{z \rightarrow x}\mathcal{P}_{k,N,D,D_0}$, lies in the sets $\mathcal{Q}_{N,DD_0}(x)$. We therefore study them.

It suffices to consider $\mathcal{Q}_{N,\Delta}(x)$ for $0\leq x \leq1/2$, since we have the following two bijections:
\begin{align}\label{eq: bij: binary quadratic forms}
    \mathcal{Q}_{N,\Delta}(x)&\rightarrow \mathcal{Q}_{N,\Delta}(-x), \quad\text{and } &\mathcal{Q}_{N,\Delta}(x+1)&\rightarrow \mathcal{Q}_{N,\Delta}(x)\\
    [a,b,c]&\mapsto [a,-b,c] \quad &[a,b,c]&\mapsto [a,b+2a,a+b+c].
\end{align}
Moreover, we have the following result:
\begin{lemma}\label{lem: Q_N,Delta}
    Assume that $\sqrt{\Delta}>N$. Then the following are equivalent:
    \begin{itemize}
        \item[(i)] $\mathcal{Q}_{N,\Delta}\neq \emptyset$.
        \item[(ii)] For all $x \in \Q$ we have that $\mathcal{Q}_{N,\Delta}(x) \neq \emptyset$.
    \end{itemize}
\end{lemma}
\begin{proof}
     Clearly $(ii)$ implies $(i)$. Thus we only need to show that $(i)$ implies $(ii)$. Assume that there exists a binary quadratic form $Q=[a,b,c]$ which satisfies $b^2-4ac=\Delta$, $a<0$ and $N|a$. Without loss of generality, we may assume that $a=-N$. If not, then $a=-a_0N$ and we can replace the form with $[-N,b,a_0c]$. For any $r \in \Z$, the binary quadratic form $Q_r:=[-N,b+2rN,c+br-Nr^2]$ is in $\mathcal{Q}_{N,\Delta}$. We have that $Q_r(x,1)>0$ for all $ x \in I_r:=(\frac{-b}{-2N}-r+\frac{\sqrt{\Delta}}{-2N},\frac{-b}{-2N}-r-\frac{\sqrt{\Delta}}{-2N})$. Since $\sqrt{\Delta}>N$, the set of intervals $\{I_r\}_{r\in \Z}$ covers $\R$ and therefore for any $x \in \Q$, there is at least one element in $\mathcal{Q}_{N,\Delta}(x)$.
\end{proof}

This is not true any longer when considering complex numbers on the upper half plane, i.e. there are $z \in \mathbb{H}$ such that $\{Q \in \mathcal{Q}_{N,\Delta}: a<0<Q_z\}=\emptyset$, even if $\mathcal{Q}_{N,\Delta}\neq \emptyset$. Recall that $Q_z:=\frac{1}{\im(z)}(a|z|^2+b\re(z)+c)$ and that $ S_Q:=\{ z \in \mathbb{H}: Q_z=0\}$.
Let $Q\in \mathcal{Q}_{N,\Delta}$ with $a<0$.  Since $a<0$, then $Q_z>0$ for all $z$ in the connected component whose boundary is the real line and $S_Q$ and $Q_z<0$ for all $z$ outside of this connected component. But as $N|a$ and $b^2-4ac=\Delta$, we know that $S_Q$ has a radius of at most $\frac{\sqrt{\Delta}}{2N}$. Thus for all $z$ with $\im(z)>\frac{\sqrt{\Delta}}{2N}$, we have that $Q_z<0$. This leads to the following remark:
\begin{remark}\label{rem: Q_z=0 if im >>0}
\begin{itemize}
    \item[(i)] Even if $\mathcal{Q}_{N,\Delta}\neq \emptyset$, then for all $z \in \mathbb{H}$ with $\im(z)>\frac{\sqrt{\Delta}}{2N}$ we have $\{Q \in \mathcal{Q}_{N,\Delta}: a<0<Q_z\}=\emptyset$.
    \item[(ii)] In particular for $z \in \mathbb{H}$ with  $\im(z)>\frac{\sqrt{\Delta}}{2N}$ we have
    \begin{align*}
    \sum\limits_{\substack{Q \in \mathcal{Q}_{N,DD_0},\\a<0<Q_z}}\chi_{D_0}(Q)Q(z,1)^{k-1}=0    
    \end{align*}
    thus $\mathcal{P}_{k,N,D,D_0}(z)=c_{k,N,D,D_0,\infty}$ is constant.
\end{itemize}
\end{remark}
\subsection{The case $N=1$}
In this part, we assume that $N=1$. Our goal is to show that if we fix a rational number $x_1$, then there is at most one more number $x_2$ such that $Q_{1,\Delta}(x_1) = Q_{1,\Delta}(x_2)$. This clarifies, why in some cases $p_{k,1,D,D_0}(x_0)$ is constant and why the definition of the limit as in Lemma \ref{lem: def lim Plz =Plx} is the only one that we can expect to converge.

In \cite{Zagier-1999} Zagier studied the following function:
\begin{align*}
    p_{1, \Delta}(x):=\sum_{Q \in Q_{1, \Delta}(x)}Q(x,1).
\end{align*}
He showed in particular that this function is constant:
\begin{theorem}[Theorem 1 in \cite{Zagier-1999}]
    For every fundamental discriminant $\Delta$ the function $p_{1, \Delta}$ has a constant value $c_{\Delta}$.
\end{theorem}
This immediately implies the following corollary:
\begin{corollary}
    For any $x_1 \in \Q$, there exists at most one other $x_2 \in \Q$ such that 
    \begin{align*}
        Q_{1,\Delta}(x_1) = Q_{1,\Delta}(x_2).
    \end{align*}
\end{corollary}
\begin{proof}
    If $x_1, x_2 \in \Q$ such that $ Q_{1,\Delta}(x_1) = Q_{1,\Delta}(x_2)$, then we have that
    \begin{align*}
        \sum_{Q_{1, \Delta (x_1)}}Q(x_1,1)= \sum_{Q_{1, \Delta (x_1)}}Q(x_2,1)=c_{\Delta}
    \end{align*}
    Thus the polynomial of second degree $\sum_{Q_{1, \Delta (x_1)}}Q(x,1)-c_{\Delta}$ has two roots. Since it can have at most two roots, there is at most one such $x_2$.
\end{proof}
\subsection{General case}
We now turn to the case when $N\geq 1$. One might assume that one can prove the general case following the same methods from \cite{Zagier-1999}. However, surprisingly, when $N>1$ one realizes that studying the sets $\mathcal{Q}_{N,\Delta}(x)$ becomes more involved. Therefore, here we only show that for any integer $n_0$ and any neighborhood $\mathcal{U}$ of $n_0$, there exist rational numbers $x \in \mathcal{U}$ such that $\mathcal{Q}_{N,\Delta}(n_0)\neq \mathcal{Q}_{N,\Delta}(x)$.

We begin by defining a generalization of the local polynomial defined by Zagier:
\begin{align}\label{eq: def p_N,Delta}
        p_{N,\Delta}(x):=\sum_{Q \in \mathcal{Q}_{N,\Delta}(x)}Q(x,1).
\end{align}
Note that the difference between  $p_{N,\Delta}$ and $p_{k,N,D,D_0}$ is that in the latter one we multiply the quadratic forms with the generalized genus character. A straightforward computation shows that $p_{N, \Delta}(x+1)=p_{N, \Delta}(x)$.
\begin{remark}
We saw earlier that if $N=1$, then $p_{N,\Delta}(x)$ is constant. If $N>1$ this is no longer true. Indeed for $N=9$, $\Delta=172\cdot 13$, we find the following values:
\begin{table}[H]
    \centering
    \begin{tabular}{c|c}
    $x$ & $p_{N,\Delta}$\\
    \hline
    0 & 696 \\
        1/2 & 680 \\
1/3 & 2056/3 \\
1/5 & 17272/25 \\
1/7 & 34040/49 \\
1/9 & 696 \\
    \end{tabular}
    \caption{Some explicit values of the local polynomial $p_{9,172\cdot 13}(x)$.}
    \label{tab: p_9,172*13 }
\end{table}
\end{remark}
We also define the polynomial described by the binary quadratic forms corresponding to zero:
\begin{align*}  
        p_{N,\Delta,0}(x):=\sum_{\substack{Q \in \mathcal{Q}_{N,\Delta}\\a<0<c }}Q(x,1)= \sum_{\substack{Q \in \mathcal{Q}_{N,\Delta}(0)}}Q(x,1).
\end{align*}
Note that if $[aN,b,c] \in \mathcal{Q}_{N,\Delta}(0)$, then so is $[aN,-b,c]$. Similarly, if $[-a_1a_2N,b,c] \in \mathcal{Q}_{N,\Delta}(0)$, then so is $[-a_1N,b,ca_1]$ (where we assume that $a_1,a_2 >0$). Thus there exists some $A,C \in \Z$ with $A=CN$ such that:
\begin{align*}
    p_{N, \Delta,0}(x)=Ax^2+C.
\end{align*}
The functions $p_{N,\Delta}$ and $p_{N,\Delta,0}$ satisfy the following relation:
\begin{lemma}\label{lem: x^2p(1/(Nx))-p(x)=-p_0(x)}
We have that
\begin{align*}
    Nx^2p_{N, \Delta}(1/(Nx))-p_{N,\Delta}(x)=-p_{N, \Delta,0}(x).
\end{align*}
\end{lemma}
\begin{remark}
    This was proven in \cite{Zagier-1999} for $N=1$.
\end{remark}
\begin{proof}
We compute:
\begin{align*}
    Nx^2p_{N, \Delta}(1/(Nx))-p_{N,\Delta}(x)&=\sum_{\substack{Q \in \mathcal{Q}_{N,\Delta}\\a<0<Q(1/(Nx),1) }}Nx^2Q(1/(Nx),1)-\sum_{\substack{Q \in \mathcal{Q}_{N,\Delta}\\a<0<Q(x,1) }}Q(x,1)\\
    &=\sum_{\substack{Q \in \mathcal{Q}_{N,\Delta}\\c<0<Q(x,1) }}Q(x,1)-\sum_{\substack{Q \in \mathcal{Q}_{N,\Delta}\\a<0<Q(x,1) }}Q(x,1).
\end{align*}
Here we used the bijection:
\begin{align*}
    \{ Q \in \mathcal{Q}_{N,\Delta} :Q(1/(Nx),1)>0 \text{ and } a<0\}&\longleftrightarrow \{ Q \in \mathcal{Q}_{N,\Delta} :Q(x,1)>0 \text{ and } c<0\}\\
    [a,b,c]&\mapsto[cN,b,a/N].
\end{align*}
Thus 
\begin{align*}
    Nx^2p_{N, \Delta}(1/(Nx))-p_{N,\Delta}(x)&=\sum_{\substack{Q \in \mathcal{Q}_{N,\Delta}\\c<0 }}\text{max}(0,Q(x,1))-\sum_{\substack{Q \in \mathcal{Q}_{N,\Delta}\\a<0}}\text{max}(0,Q(x,1))\\
    &=\sum_{\substack{Q \in \mathcal{Q}_{N,\Delta}\\c<0<a }}\text{max}(0,Q(x,1))-\sum_{\substack{Q \in \mathcal{Q}_{N,\Delta}\\a<0<c}}\text{max}(0,Q(x,1)).
\end{align*}
Where we rewrote the summand and noticed that if $a,c<0$ then a binary quadratic form will show up in both sums and therefore the corresponding terms will cancel out. Finally, we have: 
\begin{align*}
        Nx^2p_{N, \Delta}(1/(Nx))-p_{N,\Delta}(x)&=\sum_{\substack{Q \in \mathcal{Q}_{N,\Delta}\\c<0<a }}\text{max}(0,Q(x,1))+\sum_{\substack{Q \in \mathcal{Q}_{N,\Delta}\\c<0<a}}\text{min}(0,Q(x,1))\\
    &=\sum_{\substack{Q \in \mathcal{Q}_{N,\Delta}\\c<0<a }}Q(x,1)
    =-\sum_{\substack{Q \in \mathcal{Q}_{N,\Delta}\\a<0<c }}Q(x,1)=-p_0(x)
\end{align*}
This proves the claim.
\end{proof}
We are now ready to prove the following:
\begin{corollary}
Let $n_0\in \N$ be an integer. 
\begin{itemize}
\item[(i)] If $\mathcal{Q}_{N,\Delta}(0)\neq \emptyset$, then for all $n \in \Z$ we have that
\begin{align*}
    \mathcal{Q}_{N,\Delta}(n_0+1/(nN))\neq \mathcal{Q}_{N,\Delta}(n_0).
    \end{align*}
\item[(ii)] If $\sqrt{\Delta}>N$ and $\mathcal{Q}_{N, \Delta}\neq \emptyset$, then for any $x \in \Q \setminus \{n_0 \}$ we have that
\begin{align*}
    \mathcal{Q}_{N,\Delta}(x) \neq \mathcal{Q}_{N,\Delta}(n_0).
\end{align*}
\end{itemize}
\end{corollary}

\begin{proof}
Since $p_{N,\Delta}$ is translation invariant (i.e. $p_{N,\Delta}(x)=p_{N,\Delta}(x+1)$), it suffices to prove the claim for $n=0$.

For $(i)$ assume that for some $n \in \Z$ we have $\mathcal{Q}_{N,\Delta}(1/(nN))=\mathcal{Q}_{N,\Delta}(0)$. This implies that
\begin{align}\label{eq: p(1/(nN)}
    p_{N, \Delta}(1/(nN))=p_{N, \Delta,0}(1/(nN))=\frac{A}{(nN)^2}+C
\end{align}
Since $p_{N, \Delta,0}(x)=Ax^2+C$ for some $A,C \in \Z$. Since $\mathcal{Q}_{N, \Delta}(0) \neq \emptyset$, $A,C \neq 0$. From Lemma \ref{lem: x^2p(1/(Nx))-p(x)=-p_0(x)} with $x=n$, we have that
    \begin{align}\label{eq: Nn^2p(1/(Nn))&=-p_0(n)+p(n)}
        Nn^2p_{N,\Delta}(1/(Nn))-p_{N,\Delta}(n)&=-p_{N, \Delta,0}(n)\\
        \Longleftrightarrow  Nn^2p_{N,\Delta}(1/(Nn))&=-p_{N,\Delta,0}(n)+p_{N,\Delta}(n)
    \end{align}
    As $p_{N,\Delta}$ is translation invariant under $1$ (i.e. $p_{N,\Delta}(x+1)=p_{N,\Delta}(x)$) it follows that 
    \begin{align}\label{eq: p(n)=c}
        p_{N,\Delta}(n)=p_{N,\Delta}(0)=p_{N, \Delta,0}(0)=C.
    \end{align}
     Therefore combining equations \eqref{eq: p(1/(nN)}, \eqref{eq: Nn^2p(1/(Nn))&=-p_0(n)+p(n)} and \eqref{eq: p(n)=c} we get:
    \begin{align*}
        Nn^2\left(\frac{A}{(nN)^2}+C\right)=-An^2-C+C=-An^2
    \end{align*}
    Rearranging yields:
    \begin{align*}
        (A+CN)n^2+A/N=0
    \end{align*}
    Noting that $A=CN$, we get that $n$ satisfies:
    \begin{align*}
        n^2=\frac{-1}{2N}
    \end{align*}
    But clearly, no integer number $n$ satisfies this equality. This proves the claim.

    We now prove $(ii)$. Assume that $\mathcal{Q}_{N,\Delta}(x)=\mathcal{Q}_{N,\Delta}(0)$, then $p_{N,\Delta}(x)=p_{N,\Delta,0}(x)$
    \begin{align*}
        Nx^2p_{N, \Delta}(1/(Nx))=-p_{N, \Delta,0}(x)+p_{N, \Delta}(x)=0.
    \end{align*}
    But by Lemma \ref{lem: Q_N,Delta}, since $\mathcal{Q}_{N,\Delta}\neq \emptyset$, we have that $p_{N, \Delta}(x)$ is the sum over strictly positive numbers over a non-empty set, therefore $p_{N, \Delta}(1/(Nx))$ cannot  be zero. 
\end{proof}

From this it follows that for any integer $n_0$ for any neighborhood $ \mathcal{U}$ of $n_0$ there exist several $x \in \mathcal{U}$ such that $\mathcal{Q}_{N,\Delta}(x) \neq \mathcal{Q}_{N,\Delta}(n_0)$. This justifies in particular our definition of the limit $\lim_{z \rightarrow x} \mathcal{P}_{k,N,D,D_0}(z)$.

\section{Proof of the main Theorem}\label{sec: proof main theorem}
In this section we prove the Theorem \ref{theo: main prod L-functions}.

\begin{proof}
    By equation \eqref{eq: <F,f>=0 <=> LL=0} (see also Theorem  \ref{theo: composite N product l-function inner product}) we have that
    \begin{align*}
        L(F \otimes \chi_D,k)L(F \otimes \chi_{D_0},k)=0
        \Leftrightarrow \langle F,f_{k,N,D,D_0}\rangle=0.
    \end{align*}
    \textbf{Case 1: $\dim(S_{2k}(N))=1$}. If we assume that that $\dim(S_{2k}(N))=1$. Then $f_{k,N,D,D_0}$ is a multiple of $F$ and therefore:
    \begin{align*} \langle F,f_{k,N,D,D_0}\rangle=0 \Leftrightarrow 
        f_{k,N,D,D_0}=0.
    \end{align*}
    Note that since $\dim(S_{2k}(N))=1$, we may ignore the product of Hecke operators from the statement of Theorem \ref{theo: main prod L-functions}.

    We first show that if $f_{k,N,D,D_0}=0$, then $P_{k,N,D,D_0}|_{2-2k} \gamma(x)=P_{k,N,D,D_0}(x)$ for all $x \in \Q, \gamma \in \Gamma_0(N)$ and then how this implies that $P_{k,N,D,D_0}(x)=0$ for all $x \sim_{\Gamma_0(N)}0$.

    Recall equation \eqref{eq: split F}:
    \begin{align*}
        \mathcal{F}&_{1-k,N,D,D_0}(z)=\\
        &P_{k,N,D,D_0}(z)+(DD_0)^{\frac{1}{2}-k}f^*_{k,N,D,D_0}(z)-(DD_0)^{\frac{1}{2}-k}\frac{(2k-2)!}{(4\pi)^{2k-1}}
\mathcal{E}_{f_{k,N,D,D_0}}(z).
    \end{align*}
Since $f_{k,N,D,D_0}=0$, we have $f^*_{k,N,D,D_0}=0$ and $\mathcal{E}_{f_{k,N,D,D_0}}=0$. Hence 
        \begin{align*}
        \mathcal{F}_{1-k,N,D,D_0}(z)=\mathcal{P}_{k,N,D,D_0}(z).
    \end{align*}
    Since the left side is invariant under the $|_{2-2k} \gamma$ for all  $\gamma \in \Gamma_0(N)$, so is the right hand side. As $\lim_{z \rightarrow x}\mathcal{P}_{k,N,D,D_0}(z)=P_{k,N,D,D_0}(x)$, so is the local polynomial defined on the rationals, whenever the M\"{o}bius transform is well-defined. 

    We now show that $P_{k,N,D,D_0}(x)=0$ for all $x \sim_{\Gamma_0(N)}0$. We distinguish different cases.
    \begin{itemize}
        \item We first prove the case when $x=0$. By equation \eqref{eq: Flm Fricke involution} (see also Lemma \ref{lem: Flm Fricke involution}) we have that
    \begin{align*}
        \mathcal{F}_{1-k,N,D,D_0}|_{2-2k} \left( W_N -\left( \frac{D_0}{N}\right)  \right)(z) =0.
    \end{align*}
    Since $\mathcal{F}_{1-k,N,D,D_0}(z)=\mathcal{P}_{k,N,D,D_0}(z)$, we have that 
    \begin{align*}
        \mathcal{P}_{k,N,D,D_0}|_{2-2k} \left( W_N -\left( \frac{D_0}{N}\right)  \right) (z)=0.
    \end{align*}
    In Remark \ref{rem: Q_z=0 if im >>0} we noted that for $z \in \mathbb{H}$ with $\im(z)>\frac{\sqrt{DD_0}}{2N}$ we have that $ \mathcal{P}_{k,N,D,D_0}(z)=c_{k,N,D,D_0,\infty}$ is constant. Thus:
\begin{align*}
    \lim_{z \rightarrow 0}(\sqrt{N}z)^{2k-2} \mathcal{P}_{k,N,D,D_0}\left(\frac{-1}{Nz}\right)=\lim_{z \rightarrow 0}(\sqrt{N}z)^{2k-2}c_{k,N,D,D_0,\infty}=0.
\end{align*}
    On the other hand using the Fricke involution we find that:
\begin{align*}
        \lim_{z \rightarrow 0}(\sqrt{N}z)^{2k-2} \mathcal{P}_{k,N,D,D_0}\left(\frac{-1}{Nz}\right)&=\lim_{z \rightarrow 0}\left( \frac{D_0}{N}\right) \mathcal{P}_{k,N,D,D_0}(z) \\
        &= \left( \frac{D_0}{N}\right)  P_{k,N,D,D_0}(0).
    \end{align*}
   Thus $P_{k,N,D,D_0}(0)=0$.
    \item Let $x \in \Q$ with $x \sim_{\Gamma_0(N)}0$. Then there exists a matrix $\gamma= \begin{pmatrix}
        \gamma_{11} & \gamma_{12} \\ \gamma_{21}& \gamma_{22}
    \end{pmatrix}\in \Gamma_0(N)$ such that $x=\gamma\cdot 0=\frac{\gamma_{12}}{\gamma_{22}}$.
    Since $P_{k,N,D,D_0}$ is invariant under $|_{2-2k} \gamma$ for $\gamma \in \Gamma_0(N)$, we have:
    \begin{align*}\gamma_{22}^{2k-2}P_{k,N,D,D_0}\left(\frac{\gamma_{12}}{\gamma_{22}}\right)=
        \gamma_{22}^{2k-2}P_{k,N,D,D_0}(\gamma\cdot 0)=P_{k,N,D,D_0}(0)=0
    \end{align*}
    and thus $\mathcal{P}_{k,N,D,D_0}(x)=0$.
        \end{itemize}

Now we assume that $f_{k,N,D,D_0}\neq 0$. First we show that there exists a matrix $\gamma \in \Gamma_0(N)$ such that $P_{k,N,D,D_0}$ is not invariant under $|_{2-2k}\gamma$.

Since $S_2(N)$ is of dimension $1$,  $f_{k,N,D,D_0}\neq 0$ implies that $f_{k,N,D,D_0}=\alpha F$ for some non-zero $\alpha \in \C$. Since $\mathcal{F}_{1-k,N,D,D_0}$ is invariant under $|_{2-2k}\gamma$ for all $\gamma \in \Gamma_0(N)$, after rearranging we get that:
\begin{align*}
        \mathcal{P}&_{k,N,D,D_0}(z)|_{2-2k}(\gamma^{-1}-1)\\
        =&(DD_0)^{\frac{1}{2}-k}\frac{(2k-2)!}{(4\pi)^{2k-1}}
\mathcal{E}_{\alpha F}|_{2-2k}(\gamma^{-1}-1)(z) -(DD_0)^{\frac{1}{2}-k}(\alpha F)^*|_{2-2k}(\gamma ^{-1}-1)(z)\\
=& (DD_0)^{\frac{1}{2}-k}\alpha \left( \frac{(2k-2)!}{(4\pi)^{2k-1}}
\mathcal{E}_{F}|_{2-2k}(\gamma ^{-1}-1)(z) -F^*|_{2-2k}(\gamma ^{-1}-1)(z)\right)
    \end{align*}
We will show that there exists a matrix such that the right hand side is nonzero. By Lemma \ref{lem: error mod eichler integrals}, this is equal to:
\begin{align*}
(DD_0)^{\frac{1}{2}-k}\alpha  (2i)^{1-2k} \left(  \int_{i \infty}^{-(\gamma (-i\infty))} \right. &\overline{F(-\overline{w})}(w+z)^{2k-2}dw \\
& \left. -\int_{i \infty}^{\gamma (i \infty)} F( w) (z-w)^{2k-2} d w\right)
\end{align*}
We set 
\begin{align*}
    R_{\gamma} (z):=\int_{i \infty}^{-(\gamma (-i\infty))}\overline{F(-\overline{w})}(w+z)^{2k-2}dw -\int_{i \infty}^{\gamma (i \infty)} F( w) (z-w)^{2k-2} d w
\end{align*}
Since $f_{k,N,D,D_0}$ has real coefficients (see also Remark \ref{rem: real cf of f_k,N,D,D_0}), we may assume that so has $F$ and therefore $F(z)=\overline{F(-\overline{z})}$. Hence, we get:
\begin{align}\label{def: polynom P_M}
    R_{\gamma} (z):=\int_{i \infty}^{-(\gamma (-i\infty))}F(w)(w+z)^{2k-2}dw -\int_{i \infty}^{\gamma (i \infty)} F( w) (z-w)^{2k-2} d w.
\end{align}
Note that $R_{\gamma} $ is a polynomial in $z$ of degree at most $2k-2$ whose coefficients are multiples of the integrals  of the form $\int_{i \infty}^{\pm(\gamma (-i\infty))}F(\omega)w^{n}dw$. Set
\begin{align*}
    \tilde{R}_{\gamma} (z):&=\left(z^2N\right)^{k-1}R_{\gamma}\left(\frac{-1}{Nz}\right)\\
    &= 
    z^{2k-2}N^{k-1}\left(
\int_{i \infty}^{-(\gamma (-i\infty))}F(w)\left(w+\frac{-1}{Nz}\right)^{2k-2}dw\right.\\
& \left.-  \int_{i \infty}^{\gamma (i \infty)} F( w) \left(\frac{-1}{Nz}
 -w\right)^{2k-2} d w\right)
\end{align*}
Then clearly $\tilde{R}_{\gamma} $ is the zero polynomial if and only if $R_{\gamma} $ is the zero polynomial. From Lemma \ref{lem: transform to Antoniadis function}, it follows that there exists a matrix $\gamma  \in \Gamma_0(N)$ such that $\tilde{R}_{\gamma}(z)$ is not the zero polynomial. 
This implies in particular that the polynomial has at most finitely many zeros on the real line. Thus $P_{k,N,D,D_0}(x)$ is not invariant under $|_{2-2k}\gamma$ for all $\gamma \in \Gamma_0(N)$.

We now prove that this implies that there exists $x \sim_{\Gamma_0(N)}0$ such that $P_{k,N,D,D_0}(x) \neq 0$. Let $\gamma=\begin{pmatrix}
    \gamma_{11} & \gamma_{12}\\ \gamma_{21} & \gamma_{22}
\end{pmatrix} \in \Gamma_0(N)$ be such that $R_{\gamma}$ is not the zero polynomial. Let $x \neq \frac{-\gamma_{22}}{\gamma_{21}}$ with $x \sim_{\Gamma_0(N)}0$ and such that  $R_{\gamma}(x) \neq 0$. If $P_{k,N,D,D_0}(x)\neq 0$, then we are done. Otherwise from 
\begin{align*}
(\gamma_{21}x+\gamma_{22})^{2k-2}P_{k,N,D,D_0}(\gamma \cdot x)-P_{k,N,D,D_0}(x)= R_{\gamma}(x)
\end{align*}
if follows that 
\begin{align*}
    P_{k,N,D,D_0}(\gamma \cdot x)= (\gamma_{21}x+\gamma_{22})^{2-2k}R_{\gamma}(x).
\end{align*}
By assumption $\gamma_{21}x+\gamma_{22}\neq 0$. By transitivity of the action of $\Gamma_0(N)$, as $x \sim_{\Gamma_0(N)}0$, so is $ \gamma \cdot x$. This finishes the proof of this direction and hence the proof of the statement in the case when $\dim(S_{2k}(N))=1$.

\textbf{Case 2: $\dim(S_{2k}(N))>1$} Assume that $\dim(S_{2k}(N))>1$. As explained in section \ref{sec: 3-Hecke Operators}, we have that
\begin{align*}
    \langle F,f_{k,N,D,D_0}\rangle=0 \Leftrightarrow \langle F,f_{k,N,D,D_0}|_{2k}\mathbb{T}\rangle=0.
\end{align*}
Moreover by Lemma \ref{lem: lifting Hecke operator} we have that $f_{k,N,D,D_0}|_{2k}\mathbb{T}=0$ if and only
\begin{align*}
    \mathcal{F}_{1-k,N,D,D_0}|_{2-2k}\tilde{\mathbb{T}}=\mathcal{P}_{k,N,D,D_0}|_{2-2k}\tilde{\mathbb{T}}.
\end{align*}
 We therefore can argue as in the case when $\dim(S_{2k}(N))=1$. 
\end{proof}

\begin{remark}\label{rmk: Easy computations}
    It might seem inconvenient, that from the (proof of this) theorem it only follows that, if $L(F\otimes \chi_D,k)L(F\otimes \chi_{D_0},k)=0$ there exists some matrix for which the modularity of $P_{k,N,D,D_0}$ fails. However, we have that $\tilde{R}_{\gamma _1\gamma _2}=\tilde{R}_{\gamma _1}+\gamma _1 \cdot \tilde{R}_{\gamma _2}$ for all $\gamma _1,\gamma _2 \in \Gamma_0(N)$ (see Proposition \ref{prop: Antonidas}). Therefore, since $\Gamma_0(N)$ is finitely generated, it suffices to check finitely many values.
    
    We can proceed as follows: for every generator $\gamma _i$ of $\Gamma_0(N)$, it suffices to compute the local polynomial for $x_{i,n}:=\gamma _i^n(0)$ for $1 \leq n \leq 2k-2$. If for some $n$, we have $P_{k,N,D,D_0}(x_{i,n})$ is not zero, then we know that modularity fails. However, if all of $P_{k,N,D,D_0}(x_{i,n})$ are zero, then as $\tilde{R}_{\gamma _i}$ is a polynomial of degree at most $2k-2$ with $2k-1$ roots, it is the zero polynomial.
\end{remark}

Note that we have shown that $L(F\otimes \chi_D,k)L(F\otimes \chi_{D_0},k)=0$ implies that $(P_{k,N,D,D_0} |_{2-2k}\tilde{\mathbb{T}})(x)=0$ for all $x \sim_{\Gamma_0(N)}0$, i.e. $x$ of the form $\frac{q_1}{q_2}$ with $(q_1,q_2)=(q_2,N)=1$. We further show that $(P_{k,N,D,D_0} |_{2-2k}\tilde{\mathbb{T}})(x)=0$ for all $x$ of the form $\frac{q_1}{Nq_2}$ with $(q_1,Nq_2)=1$. In general we expect that one can prove that $(P_{k,N,D,D_0} |_{2-2k}\tilde{\mathbb{T}})(x)=0$ for all $x$ using Atkin-Lehner involutions. 
\begin{proposition}
    Under the same assumptions as in Theorem \ref{theo: main prod L-functions}, we have:
    If $L(F\otimes \chi_{D},k)L(F\otimes \chi_{D_0},k)=0$, then $P_{k,N,D,D_0} |_{2-2k} \tilde{\mathbb{T}}(x)=0$ for all $x \in \Q$ of the form $x=\frac{q_1}{q_2N}$ with $(q_1,Nq_2)=1$.
\end{proposition}

\begin{proof}
    We first assume that $\text{dim}(S_2(N))=1$. Let $x=\frac{q_1}{q_2N}$ with $(q_1,Nq_2)=1$. Then there exists $\gamma_{11},\gamma_{12} \in \Z$ such that $\gamma=\begin{pmatrix} \gamma_{11} & \gamma_{12} \\ q_2N & -q_1\end{pmatrix} \in \Gamma_0(N)$. Thus
\begin{align*}
    \lim_{z \rightarrow x} q_2Nz-q_1=0 \text{ and }\lim_{z \rightarrow x}\gamma \cdot z=i \infty.
\end{align*}
    This implies that 
    \begin{align*}
        \lim_{z \rightarrow x}(q_2N z-q_1)^{2k-2} \mathcal{P}_{k,N,D,D_0}\left(\gamma \cdot z\right)=\lim_{z \rightarrow x}(q_2N z-q_1)^{2k-2} c_{k,N,D,D_0, \infty}=0
    \end{align*}
    On the other hand, we just showed if $L(F\otimes \chi_D,k)L(F\otimes \chi_{D_0},k)=0$, then $\mathcal{P}_{k,N,D,D_0}$ is invariant under $|_{2-2k}\gamma$ for $\gamma \in \Gamma_0(N)$. Hence:    \begin{align*}
        \lim_{z \rightarrow x}(q_2N z-q_1)^{2k-2} \mathcal{P}_{k,N,D,D_0}\left(\gamma \cdot z\right)=\lim_{z \rightarrow x} \mathcal{P}_{k,N,D,D_0}(z) = P_{k,N,D,D_0}(x).
    \end{align*}
    This shows that $P_{k,N,D,D_0}(x)=0$ for all $x$ of the form $\frac{q_1}{Nq_2}$ with $(q_1,Nq_2)=1$. 
    Arguing as in the proof of the Theorem \ref{theo: main prod L-functions}, we can replace $P_{k,N,D,D_0}$ with $P_{k,N,D,D_0}|_{2-2k} \tilde{\mathbb{T}}$ when $\dim(S_{2k}(N))>1$.
\end{proof}

\section{Computational Examples}\label{sec: comp examples}
In this section we compute some explicit examples: we consider newforms in the space $S_4(9)$ and $S_{4}(25)$.
To compute the local polynomials, we used the code provided by \cite{males-mono-rollen-wagner} replacing the computation of the generalized genus character with the approach described in proposition \ref{prop: explicit genus char}.
First we outline the approach we used for the computations:
\subsection{Algorithm}
Let $F \in S_{2k}(N)$. In the following, we assume that $D$ and $D_0$ are two fundamental discriminants satisfying the conditions of the Theorem \ref{theo: main prod L-functions}. Fix $D_0$ such that $L(F\otimes \chi_{D_0},k)\neq 0$. Our goal is to run over discriminants $D$, satisfying the conditions of the main theorem, and determine if $L(F\otimes \chi_D,k)=0$.

Recall that $\tilde{\mathbb{T}}=\prod_i\left( T_p-a_{p_i}p_i^{1-2k}\right)$ is a product of Hecke operators. If $L(F\otimes \chi_D,k)=0$ then we expect that
\begin{align*}
    P_{k,N,D,D_0}|_{2-2k}\tilde{\mathbb{T}}(0)
    &= \prod_i (p_i^{1-2k}(1+a_{p_i})+1)c_{k,N,D,D_0,\infty} \\
    &+c_{k,DD_0}\left(\sum_{\substack{Q \in \mathcal{Q}_{N,DD_0},\\a<0<Q(x,1)}}\chi_{D_0}(Q)Q(x,1)^{k-1}\right)|_{2k-2}\tilde{\mathbb{T}}(0)
    =0.
\end{align*}
We set
\begin{align*}
    C_0:&=
    -\prod_{i}(p_i^{1-2k}(1+a_{p_i})+1)c_{k,N,D,D_0,\infty} c_{k,DD_0}^{-1}\\
    &=\left(\sum_{\substack{Q \in \mathcal{Q}_{N,DD_0}\\a<0<Q(x,1)}}\chi_{D_0}Q(x,1)^{k-1}\right)|_{2k-2}\tilde{\mathbb{T}}(0).
\end{align*}

Let $\gamma _1,...,\gamma _l$ be generators of $\Gamma_0(N)$. Set $x_i:=\gamma _i \cdot 0$. (Note that one needs to make sure that this is well-defined, i.e. if $\gamma _i=\begin{pmatrix}
    \gamma_{11} & \gamma_{12} \\\gamma_{21} & \gamma_{22}
\end{pmatrix}$, then $\gamma_{22} \neq 0$. If this is not the case, one may replace $0$ with another integer $n$ such that $\gamma_{21}n+\gamma_{22} \neq 0$ for all $l$.)

For all generators, i.e. for $1 \leq i \leq l$, we compute
\begin{align*}
    \left( \sum_{\substack{Q \in \mathcal{Q}_{N,DD_0}\\a<0<Q(x,1)}}\chi_{D_0}(Q)Q(x,1)^{k-1}\right)|_{2-2k}\tilde{\mathbb{T}}(x_i)-C_0
\end{align*}
We distinguish two cases:
\begin{itemize}
    \item If this is non-zero for one of $x_i$, then $P_{k,N,D,D_0}|_{2-2k}\tilde{\mathbb{T}}$ is not modular and therefore $L(F\otimes \chi_D,k)\neq 0$.
    \item If this is non-zero for all generators, then set $x_{i,j}:=\gamma _i^j \cdot 0$ for $1\leq i \leq l$ and $ 1\leq j \leq 2k-1$. Then
\begin{align*}
    \left( \sum_{\substack{Q \in \mathcal{Q}_{N,DD_0}\\a<0<Q(x_{i,j},1)}}\chi_{D_0}(Q)Q(x,1)^{k-1}\right)|_{2-2k}\tilde{\mathbb{T}}(x_{i,j})-C_0
\end{align*}
is equal to zero for all choices of $i,j$ if and only if $L(F \otimes \chi_D,k)=0.$ (See also remark \ref{rmk: Easy computations}.)
\end{itemize}

\subsection{The space $S_{4}(9)$.}\label{subsec: 3- ex S_4(9)}
One can show that the space $S_{4}^{new}(9)=S_{4}(9)$ is spanned by:
\begin{align*}F=q-8q^{4}+20q^7-70q^{13}+64q^{16}+56q^{19}&-
125q^{25}-160q^{28}\\&+308q^{31}+110q^{37}+O(q^{40})
\end{align*}
Since moreover $F|_2W_{3^2}=F$, the function $F$ satisfies the condition of the Theorem \ref{theo: main prod L-functions}. By one-dimensionality of the $S_4(9)$, we can ignore the Hecke operators. In Table \ref{tbl: sml L(F chiD,k) N=9, cf half-integral weight forms} we computed some central values of the L-function of twists of $F$:
\begin{table}[H]
    \centering
    \begin{tabular}{c|c}
  $D$& $L(F \otimes \chi_D,2)$ \\
  \hline
5  &  $1.22352\dots$ \\
13  &  $0.58368\dots$ \\
28  &  $2.95446\dots$ \\
53  &  $0.03545\dots$  \\
88  &  $0.53026\dots$  \\
152  &  $1.86870\dots$  \\
161  &  $1.31245\dots$ \\
172  &  $0$  
\end{tabular}
    \caption{$L(F \otimes \chi_D,2)$ for different fundamental discriminants $D$}
    \label{tbl: sml L(F chiD,k) N=9, cf half-integral weight forms}
\end{table}
Therefore we pick $D_0=13$ if $\left( \frac{D}{3}\right)=1$ and $D_0=5$ if $\left( \frac{D}{3}\right)=-1$. We set
\begin{align*}
P(x):= \sum_{\substack{Q \in Q_{9,DD_0}\\a<0<Q(x,1)}}\chi_{D_0}(Q)Q(x,1)-C_0,
\end{align*}
with
\begin{align*}
    C_0:=  \sum_{\substack{Q \in Q_{9,DD_0}\\a<0<c}}\chi_{D_0}(Q)Q(0,1).
\end{align*}
The group $\Gamma_0(9)$ is generated by
\begin{align*}
    \gamma_1:=\begin{pmatrix}
        1 & 1\\ 0 & 1
    \end{pmatrix}, \gamma_2:=\begin{pmatrix}
        4 & -1\\ 9 & -2
    \end{pmatrix}, \gamma_3:=\begin{pmatrix}
        7 & -4\\ 9 & -5
    \end{pmatrix}, \text{ and }\gamma_4:=\begin{pmatrix}
        -1 & 0\\ 0 & -1
    \end{pmatrix}.
\end{align*}
Hence in a first step, we compute $P(x)$ for $x=1,1/2,4/5$ and $0$.
The results of these computations can be found in Table \ref{tbl: sml P_{4,9,D,D_0}}:
\begin{table}[H]
    \centering
    \begin{tabular}{c|c|cccc}
  $D$& $\left( \frac{D}{3}\right)$ & $P(1)$ & $P(1/2)$ & $P(4/5)$ & $P(0)$\\
  \hline
28 & 1 & 0 & 12 & 96/25 & 0 \\
53 & -1 & 0 & -3/2 & -12/25 & 0 \\
88 & 1 & 0 & -12 & -96/25 & 0 \\
152 & -1 & 0 & 24 & 192/25 & 0 \\
161 & -1 & 0 & 21 & 168/25 & 0 \\
172 & 1 & 0 & 0 & 0 & 0 
\end{tabular}
    \caption{Local polynomial $P_{4,9,D,D_0}$}
   \label{tbl: sml P_{4,9,D,D_0}}
\end{table}

From the computed values it follows that we only need to check the case $D=172$ and $D_0=13$ more carefully. Indeed, we find that in this case $ P(\gamma_i^2\cdot0)=0$ and $ P(\gamma_i^3\cdot0)=0$ for all $1 \leq i \leq 4$. 
\subsection{The space $S_4(25)$}\label{subsec: Ex. S_4(25)}
We now study the function
\begin{align*}
    F_{25,1}:=q + q^2 + 7q^3 - 7  q^4 + 7  q^6 + 6  q^7 - 15  q^8 + 22  q^9 + O(q^{11}) \in S_4(25),
\end{align*}
which is a newfrom in $S_4(25)$. The space $S_4(25)$ is $5$ dimensional. Next to $F_{25,1}$, there are two other newforms which we denote by $F_{25,1}$ and $F_{25,2}$. Moreover, there are $2$ oldforms arising from level $5$, which we denote by $F_{5}$ and $F_{5,tw}(z)=F_{5}(5z)$. 

 As $F|W_5=F$, thus $F$ satisfies the conditions of the Theorem \ref{theo: main prod L-functions}. In Table \ref{tbl: small L(F_D,k) F in S4,25}, we computed some central values of the $L$-function of twists of $F_{25,1}$.
\begin{table}[H]
    \centering
    \begin{tabular}{c|c}
    $D$ & $L(F_{25,1} \otimes \chi_D,2)$ \\
    \hline
    8& $1.72936\dots$ \\
     21& $1.62649\dots$  \\
 44& $0.13407\dots$  \\
 53& $0$ \\
 56& $0.37350\dots$  \\ 
 69& $0$ \\
 73& $1.568476\dots$\\ 
 77& $0.23165\dots$ \\
    \end{tabular}
    \caption{$L(F_{25,1} \otimes \chi_D,2)$}
    \label{tbl: small L(F_D,k) F in S4,25}
\end{table}
We will fix $D_0=21$ if $\left( \frac{D}{5}\right)=1$ and $D_0=8$ if $\left( \frac{D}{5}\right)=-1$.

We now explain how we construct the product of Hecke operators: The eigenvalues of the $5$ cusp forms spanning $S_4(5)$ under the Hecke operator for the primes $2,3$ and $7$ can be found in Table \ref{tbl: Eigenvalues small Tp S4(25)}. 
\begin{table}[H]
    \centering
    \begin{tabular}{c|ccc}
        $p$ & 2& 3& 7 \\
        \hline        
        $F_{25,1}$& $1$& $7$ & $6$\\
        $F_{25,2}$& $4$& $-2$ & $-6$\\
        $F_{25,3}$& $-1$& $-7$ & $-6$\\
        $F_5,F_{5,tw}$& $-4$& $2$ & $6$\\
    \end{tabular}
    \caption{Eigenvalues under Hecke operators $T_p$ of eigenforms in the space $S_4(25)$.}
    \label{tbl: Eigenvalues small Tp S4(25)}
\end{table}

It follows that the operator
\begin{align*}
    \mathbb{T}=(T_7+6)(T_2+4): S_4(25) \rightarrow \text{span}_{\C}\{F_{25,1}\}
\end{align*}
is well-defined and non-zero.
We set:
\begin{align*}
\Tilde{\mathbb{T}}:=(T_7+7^{-3}\cdot 6)(T_2+ 2^{-3}\cdot 4).
\end{align*}

 Set
\begin{align*}
P(x):=\left( \sum_{\substack{Q \in Q_{25,DD_0}\\a<0<Q(x,1)}}\chi_{D_0}(Q)Q(x,1)^{k-1}\right)|_{2-2k}\tilde{\mathbb{T}}( x_i)-C_0,
\end{align*}
where
$C_0:= \left( \sum_{\substack{Q \in Q_{25,DD_0}\\a<0<Q(x,1)}}\chi_{D_0}(Q)Q(x,1)^{k-1}\right)|_{2-2k}\tilde{\mathbb{T}}(0)$. The group $\Gamma_0(N)$ is generated by:
\begin{align*}
    \gamma_1:=\begin{pmatrix}
        1 & 1 \\ 0 & 1
    \end{pmatrix}, \gamma_2:=\begin{pmatrix}
        6 & -1 \\ 25 & -4
    \end{pmatrix}, \gamma_3:=\begin{pmatrix}
        7 & -2 \\ 25 & -7
    \end{pmatrix}, \gamma_4:=\begin{pmatrix}
        11 & -4 \\ 25 & -9
    \end{pmatrix},\\
    \gamma_5:=\begin{pmatrix}
        16 & -9\\ 25 & -14
    \end{pmatrix},
    \gamma_6:=\begin{pmatrix}
        18 & -13 \\ 25 & -18
    \end{pmatrix},\text{ and }
    \gamma_7:=\begin{pmatrix}
        21 & -16 \\ 25 & -19
    \end{pmatrix}
\end{align*}
In Table \ref{tbl sml: Local polynomial P_{4,25,D,D_0}} we compute $P(x)$ at $x\in \{ \gamma_i\cdot 0 : 1 \leq i \leq 7\}=\{ 1,1/4,2/7,4/9,9/14,13/18,16/19\}$.
\begin{table}[H]
    \centering
    \scriptsize
    \begin{tabular}{c|c|ccccccc}
  $D$& $\left( \frac{D}{3}\right)$& $P(1)$& $P(1/4)$ &$P(2/7)$ &$P(4/9)$& $P(9/14)$& $P(13/18)$ &$P(16/19)$\\
  \hline
44 & 1 & 0 & 25/686 & 580/16807 & 200/3969 & 2475/67228 & 565/15876 & 400/17689 \\
53 & -1 & 0 & 0 & 0 & 0 & 0 & 0 & 0 \\
56 & 1 & 0 & -25/343 & -1160/16807 & -400/3969 & -2475/33614 & -565/7938 & -800/17689 \\
69 & 1 & 0 & 0 & 0 & 0 & 0 & 0 & 0 \\
73 & -1 & 0 & 125/1372 & 1450/16807 & 500/3969 & 12375/134456 & 2825/31752 & 1000/17689 \\
77 & -1 & 0 & -25/686 & -580/16807 & -200/3969 & -2475/67228 & -565/15876 & -400/17689 \\
    \end{tabular}
    \caption{Local polynomial $P_{4,25,D,D_0}|_{2-2k}\tilde{\mathbb{T}}$}
    \label{tbl sml: Local polynomial P_{4,25,D,D_0}}
\end{table}
From this table it follows that it suffices to compute $P(\gamma_i^2\cdot 0)$ for $ 1 \leq i \leq 7$ when $D_0=8,D=53$ as well as $D_0=21,D=69$.  We get that $P(\gamma_i^2\cdot 0)=0$, which confirms the claim of the Theorem \ref{theo: main prod L-functions}.

\appendix
\section{Locally Harmonic Maass Forms}\label{app: locally harmonic Maass Form}

This appendix is dedicated to carefully proving the properties of the locally harmonic Maass form from  equation \eqref{def: Flm}. We recall its definition:
\begin{align*}
    \Flm(z):=\frac{(DD_0)^{\frac{1}{2}-k}}{\binom{2k-2}{k-1}2\pi}\sum_{Q \in \mathcal{Q}_{N,DD_0}}&\chi_{D_0}(Q)\text{sgn}(Q_z)\\
    &Q(z,1)^{k-1}\beta\left( \frac{DD_0y^2}{|Q(z,1)|^2};k-\frac{1}{2},\frac{1}{2}\right).
\end{align*}
Here
\begin{itemize}
    \item $\mathcal{Q}_{N,DD_0}=\{ Q=[a,b,c]: N|a \text{ and }b^2-4ac=DD_0\}$
    \item The generalized genus character is given by (see definition \ref{def: generalized genus character}):
    \begin{align*}
        \chi_{D_0}(Q)=\begin{cases}
            0 & \text{if } (a,b,c,D)>1\\
            \left( \frac{D_0}{r}\right) &\text{if } (a,b,c,D_0)=1, \quad Q \text{ represents }r, \quad (r,D_0)=1.
        \end{cases}
    \end{align*}
    \item for $z \in \mathbb{H}$, we set: $Q_z:=\frac{1}{\im(z)}(a|z|^2+b\re(z)+c)$
    \item for $s,w \in \C$ with $\re(s),\re(w)>0$, we set
    \begin{align*}
        \beta(v;s,w):=\int_0^v u^{s-1}(1-u)^{w-1}du.
    \end{align*}
\end{itemize}

This work primarily consists in applying the straightforward modifications in the proofs from \cite{B-K-K} — typically substituting $SL_2(\Z)$ with $\Gamma_0(N)$ and incorporating the generalized genus character. Therefore often we will sketch a proof and refer for further details to \cite{B-K-K}.
We recall the definition of a locally harmonic Maass form: 
\begin{definition}\label{def: lhmF -app1}
A function $\mathcal{F}:\mathbb{H} \rightarrow \C$ is called a locally harmonic Maass form of weight $2-2k$ for the congruence subgroup $\Gamma$ and with exceptional set $E$ if:
\begin{itemize}
    \item[(i)] For every $\gamma \in \Gamma$ for every $ z\in \mathbb{H}$: $\mathcal{F}|_{2-2k}\gamma(z)=\mathcal{F}(z)$.
    \item[(ii)] For every $z \in \mathbb{H} \setminus E$, there is a neighborhood $\mathcal{U}$ of $z$ in which $\mathcal{F}$ is real analytic and $\Delta_{2-2k}(\mathcal{F})=0$. 
    
    \item[(iii)] For every $z \in E$, we have:
    \begin{align*}
        \mathcal{F}(z)=\frac{1}{2}\lim_{w \rightarrow 0^+} (\mathcal{F}(z +iw)+\mathcal{F}(z -iw)), \quad w \in \R.
    \end{align*}
    \item[(iv)] The function $\mathcal{F}$ exhibits at most polynomial growth towards $ i \infty$.
\end{itemize}
\end{definition}

This appendix is structured as follows: we first explain how to decompose $\Flm$ and $f_{k,N,D,D_0}$ into a sum of functions corresponding to equivalent classes of binary quadratic forms, then we show that $\Flm$ and $f_{k,N,D,D_0}$ can be written in terms of Poincare series. Afterwards we show the convergence of $\Flm$ and study its value on the exceptional set. We then study the action of the operators $\xi_{2-2k}$ and $D^{1-2k}$ and use this to prove the expansion of $\Flm$ in terms of (non-)holomorphic Eichler integrals of $f_{k,N,D,D_0}$ and a local polynomial. In the last section we study the Hecke operator and the Atkin-Lehner involution on $\Flm$. 

\subsection{Decomposition along equivalence classes of binary forms}
We decompose the functions $f_{k,N,D,D_0}$ and $\Flm$ along equivalence classes of binary quadratic forms. We recall that we defined (see equation \eqref{eq: def sec3-f_k,N,D,D_0}):
\begin{align*}
        f_{k,N,D,D_0}(z):=\frac{(DD_0)^{k-\frac{1}{2}}}{\binom{2k-2}{k-1}\pi}\sum_{Q \in \mathcal{Q}_{N,DD_0}}\chi_{D_0}(Q) Q(z,1)^{-k} \in S_{2k}(N).
\end{align*}

A subset $\mathcal{A} \subset \mathcal{Q}_{N,DD_0}$ is an equivalence class of binary quadratic forms if there exists $Q_0 \in \mathcal{Q}_{N,DD_0}$ such that:
\begin{align}\label{eq: A equiv class}
    \mathcal{A}=[Q_0]:=\{ Q \in \mathcal{Q}_{N,DD_0}: \exists \gamma \in \Gamma_0(N) \text{ such that }Q=Q_0\circ \gamma\}
\end{align}
We set
\begin{align}\label{eq: def f_k,N,D,D_0,A}
    f_{k,N,D,D_0,\mathcal{A}}:=\frac{(-1)^k(DD_0)^{k-\frac{1}{2}}\chi_{D_0}([A])}{\binom{2k-2}{k-1}\pi}\sum_{Q\in \mathcal{A}}Q(z)^{-k}.
\end{align}
Moreover we define
\begin{align}\label{eq: def F_1-k,N,D,D_0,A}
    \mathcal{F}_{1-k,N,D,D_0,\mathcal{A}}:=\frac{(DD_0)^{\frac{1}{2}-k}(-1)^k\chi_{D_0}([\mathcal{A}])}{\binom{2k-2}{k-1}\pi}\sum_{Q \in \mathcal{A}}\text{sgn}(Q_z)Q(z,1)^{k-1}\psi\left( \frac{DD_0y^2}{|Q(z,1)|^2}\right).
\end{align}
Here in order to simplify the notation we denote as in \cite{B-K-K}:
\begin{align*}
    \psi(v):=\frac{1}{2}\beta\left( v; k-\frac{1}{2},\frac{1}{2}\right)
\end{align*}

Note that, since the generalized genus character is invariant under the action of $SL_2(\Z)$, the value $\chi_{D_0}([\mathcal{A}])$ is indeed well-defined. 

The functions $\Flm$ and $f_{k,N,D,D_0}$ can be written as a sum over finitely many $\mathcal{F}_{1-k,N,D,D_0,\mathcal{A}}$, resp. $f_{k,N,D,D_0,\mathcal{A}}$. It is known (e.g. Theorem $1$ in §8 in \cite{Zagier-zetafunktionen-und-quadratische-korper}) that $SL_2(\Z)\setminus Q_{DD_0}$ is finite. Since $[SL_2(\Z):\Gamma_0(N)]<\infty$ and $\mathcal{Q}_{N,DD_0}\subset Q_{DD_0}$, so is $ \Gamma_0(N) \setminus \mathcal{Q}_{N,DD_0}$.

Instead of studying $\mathcal{F}_{1-k,N,D,D_0}$ we will often study $\mathcal{F}_{1-k,N,D,D_0, \mathcal{A}}$.
\subsection{Poincare series}
\subsubsection{Equivalence classes of binary forms and equivalence classes of matrices}
The goal of this section is to rewrite $f_{k,N,D,D_0,\mathcal{A}}$ and $\mathcal{F}_{1-k,N,D,D_0,\mathcal{A}}$ as Poincare series.
To do so we first describe the equivalence classes in terms of hyperbolic points. Here we say that a pair of points $\eta_1, \eta_2 \in \R \cup \{ \infty \}$ is hyperbolic if there exists a matrix $ \gamma \in SL_2({\Z})$ which fixes $\eta_1$ and $\eta_2$ and such that $|\text{tr}(\gamma)|>2$.

We start with making the following general observation. To any matrix
\begin{align*}
    \gamma=\begin{pmatrix}
    \gamma_{11} & \gamma_{12} \\ \gamma_{21} & \gamma_{22}\end{pmatrix} \in SL_2(\Z)
\end{align*}
 we can associate a binary quadratic from in the following way:
\begin{align*}
    Q_{\gamma}=\gamma_{21}z^2+(\gamma_{22}-\gamma_{11})z-\gamma_{12}.
\end{align*}
One can compute that $\text{det}(Q_{\gamma})=\text{tr}(\gamma)^2-4$.

We will use this to show that there is a bijection between equivalence classes of binary quadratic forms and equivalence classes of matrices in $\Gamma_0(N)$.

Let $Q_0=[a,b,c] \in \mathcal{Q}_{N,DD_0}$. We denote by $\eta_1,\eta_2$ its roots. Moreover let $\Gamma_{Q_0}=\Gamma_{\eta}$ be the group of matrices in $SL_2(\Z)$ fixing $Q_0$. Recall from equation \eqref{eq: def stab Q} that $\Gamma_{\eta}/ \{ \pm I \}$ is generated by
\begin{align*}
    \gamma_n:=\begin{pmatrix}
        \frac{t + bu}{2} & cu \\ -au& \frac{t -bu}{2}
    \end{pmatrix}.
\end{align*}
Here $t,u$ are the smallest positive integers such that $t^2-DD_0u^2=4$. Note that $t$ and $u$ only depend on $DD_0$, but not on $Q_0$. As $N|a$ we have that $\gamma_{\eta} \in \Gamma_0(N)$. 
The binary quadratic form associated to this matrix is given by:
\begin{align*}
   Q_{\gamma_{\eta}}(z)=-u(az^2+bz+c)=-uQ_0(z) \in Q_{N,DD_0u^2}
\end{align*}
In particular $\gamma_{\eta}$ fixes $\eta_1$ and $\eta_2$ and since $\text{tr}(\gamma_{\eta})^2=u^2DD_0+4$, it follows that $\eta_1,\eta_2$ are hyperbolic elements.

To prove the bijection between binary quadratic forms equivalent to $Q_0$ and $\Gamma_{Q_0}\setminus\Gamma_0(N)=\Gamma_{\eta}\setminus\Gamma_0(N)$, we need the following fact. There exists a matrix $\sigma_{\eta}=\sigma_{Q_0} \in SL_2(\R)$ such that $\sigma_{\eta}(0)=\eta_1$ and $\sigma_{\eta}(\infty)=\eta_2$ and:
\begin{align}\label{eq: sigma def prop}
    \sigma_{\eta}^{-1}\gamma_{\eta}\sigma_{\eta}=\pm \begin{pmatrix}
        \xi & 0 \\ 0 & \xi^{-1}
    \end{pmatrix} :=\gamma_{\xi}.
\end{align}
By replacing $\xi$ with its inverse if necessary, we may assume that $\xi>1$. (Note that if $|\xi|=1$, then $u=0$, contradicting that $u$ is positive). 
We are now ready to show the following:
\begin{lemma}
    There is a bijection:
\begin{align*}
    \{ Q \in \mathcal{Q}_{N,DD_0} : \exists \gamma \in \Gamma_0(N) \text{ such that } Q_0\circ \gamma = Q\} &\leftrightarrow \left\{ \gamma \in \Gamma_{\eta} \setminus \Gamma_0(N) \right\},\\
    Q\circ \gamma &\mapsto [\gamma]\\
    \frac{1}{-u}Q_{\gamma^{-1}\gamma_\eta\gamma} &\mapsfrom \gamma.
\end{align*}
\end{lemma}

\begin{proof}
    Let $Q=Q_0\circ \gamma$ for some matrix $\gamma \in \Gamma_0(N)$. Then we have that
\begin{align*}
    Q=Q_0\circ \gamma = \frac{1}{-u}(Q_{\gamma_{\eta}} \circ \gamma)=\frac{1}{-u} Q_{ \gamma^{-1} \gamma_{\eta} \gamma}.
\end{align*}
The last step follows from a straight forward computation. It remains to show that
\begin{align*}
    \{ \gamma \in \Gamma_0(N) | \gamma^{-1}\gamma_{\eta}\gamma=\gamma_{\eta}\}=\Gamma_{\eta}.
\end{align*}
One sees that $\Gamma_{\eta}$ is a subset of the left hand side and one can check that the latter one is indeed a subgroup. Now let $\gamma \in \Gamma_0(N)$ be such that $\gamma^{-1}\gamma_{\eta}\gamma=\gamma_{\eta}$. We will show that $\sigma_{\eta}^{-1}\gamma\sigma_{\eta}$ is a diagonal matrix. It then follows from the definition of $\sigma_{\eta}$ that $\gamma \in \Gamma_{\eta}$. Using the properties of $\gamma$ and $\sigma_{\eta}$, we find that:
\begin{align*}
    \sigma_{\eta}^{-1}\gamma\sigma_{\eta}=\sigma_{\eta}^{-1}\gamma_{\eta}^{-1}\gamma\gamma_{\eta}\sigma_{\eta}
    = \gamma_{\xi}^{-1}\sigma_{\eta}^{-1}\gamma\sigma_{\eta}\gamma_{\xi}
\end{align*}
    Thus:
\begin{align*}\gamma_{\xi}\sigma_{\eta}^{-1}\gamma\sigma_{\eta}=\sigma_{\eta}^{-1}\gamma\sigma_{\eta}\gamma_{\xi}.
\end{align*}
Since $\sigma_{\eta}^{-1}\gamma\sigma_{\eta}$ commutes with a diagonal matrix and $\xi^{-1}\neq\xi$, this means that $\sigma_{\eta}^{-1}\gamma\sigma_{\eta}$ is a diagonal matrix itself. Since its determinant is $1$, it needs to be of the same shape as $\gamma_{\xi}$.
\end{proof}
\begin{remark}
Note that here the normalization by $-1/u$ is necessary, to get a binary quadratic form with determinant $DD_0$.
However, when rewriting $\mathcal{F}_{1-k,N,D,D_0,\mathcal{A}}$ and $f_{k,N,D,D_0,\mathcal{A}}$ in terms of Poincare series we sum over terms of the form
\begin{align*}
    \frac{uQ(z,1)}{(\text{det}(uQ))^{1/2}}=\frac{Q(z,1)}{(\text{det}(Q))^{1/2}}.
\end{align*}
\end{remark}
We also need the following lemma:
\begin{lemma}[Lemma 3.1 in \cite{B-K-K}]\label{lem: B-K-K}
Let $\eta_1, \eta_2$ be a hyperbolic pair. Let $Q_{\eta}=\frac{1}{-u}Q_{\gamma_{\eta}}\in \mathcal{Q}_{N,DD_0}$ be the associated binary quadratic form. Let $Q:=Q_{\eta}\circ \gamma$ for some $\gamma \in \Gamma_0(N)$ Then there exists a constant $r \in \mathbb{R}^+$ such that:
\begin{align*}
    \sigma_{\eta}^{-1}\gamma=\begin{pmatrix}
        \sqrt{r} & 0 \\ 0 & \frac{1}{\sqrt{r}}
    \end{pmatrix}\sigma_Q^{-1}
\end{align*}
    In particular:
    \begin{align*}
        \text{arg}(\sigma^{-1}_{\eta}\gamma z)=\text{arg}(\sigma_{Q}^{-1}z) \quad \text{and }\quad \text{sgn}(\re(\sigma_{\eta}^{-1}\gamma \tau))=\text{sgn}(\re(\sigma_{\eta}^{-1}\tau)).
    \end{align*}
Moreover we have:
    \begin{align*}
       ( z|_{-2}\sigma_{\eta}^{-1}\gamma )(z)
       =(z|_{-2}\sigma_{Q}^{-1}) (z)
       =\frac{-Q(z,1)}{\sqrt{DD_0}}
    \end{align*}
    Here we denote by $\sigma_{Q}$ the matrix associated to the roots of $Q$ as in equation \eqref{eq: sigma def prop}.
\end{lemma}
\subsubsection{Poincare series for $f_{k,N,D,D_0,\mathcal{A}}$}
We show that we can write $f_{k,N,D,D_0,\mathcal{A}}$ as a Poincare series. Let $\mathcal{A}=[Q_0]$ be an equivalence of binary quadratic forms. Denote with $\eta:=\{\eta_1,\eta_2\}$ the roots of $Q_0$.
We set
\begin{align*}
    h_k(z):=z^{-k}.
\end{align*}
 We define the following hyperbolic Poincare series:
\begin{align*}
    m_{k,\eta,N,D_0}:= \chi_{D_0}([Q_0])\sum_{\gamma \in \Gamma_{\eta} \setminus \Gamma_0(N)} (h_k |_{2k}\sigma_{\eta}^{-1}\gamma )(z).
\end{align*}
Note that apart from multiplying with the generalized genus character, this series only differs from the one defined in \cite{B-K-K} by the set we are summing over, replacing $\Gamma_0(N)$ with  $SL_2(\Z)$.

\begin{proposition} 
Let $k>0$ and $D,D_0$ be two fundamental discriminants with $(-1)^kD,(-1)^kD_0>0$ and such that $DD_0$ is not a perfect square. Then we have:
    \begin{align*}
        m_{k, \eta, N,D_0}=(DD_0)^{\frac{1-k }{2}}\pi \binom{2k-2}{k-1}f_{k,N,D,D_0,\mathcal{A}}
    \end{align*}
\end{proposition}

\begin{proof}
    Using the bijection between the set of binary quadratic froms equivalent to $Q_0$ and $\Gamma_{\eta}\setminus \Gamma_0(N)$ and Lemma \ref{lem: B-K-K} we get:
\begin{align*}
    \sum_{Q \in [Q_0]} Q(z,1)^{-k}
    &=\sum_{\gamma \in \Gamma_{\eta}\setminus \Gamma_0(N)}\left((Q_{\eta}\circ\gamma)(z)\right)^{-k}\\
    &=\left(-\sqrt{DD_0}\right)^{-k}\sum_{\gamma \in \Gamma_{\eta}\setminus \Gamma_0(N)} (h_k |_{2k} \sigma_{\eta}^{-1}\gamma) (z),
\end{align*}
which yields the desired result. 
\end{proof}

\subsubsection{Poincare series for $\mathcal{F}_{k,N,D,D_0,\mathcal{A}}$}
We also want to write $\mathcal{F}_{1-k,N,D,D_0,\mathcal{A}}$ in terms of Poincare series.
First we define
\begin{align*}
    \varphi\left( \text{arctan}\left|\frac{\sqrt{DD_0}y}{a|z|^2+b\re z+c}\right| \right):=\frac{1}{2}\beta\left( \frac{DD_0y^2}{|az^2+bz+c|^2};k-\frac{1}{2},\frac{1}{2}\right)
\end{align*}
Moreover we set:
\begin{align*}
    \hat{\varphi}(z):=z^{k-1}\text{sgn}(x)\varphi\left(\text{arctan}\left|\frac{y}{x}\right|\right),
\end{align*}
where if $a|z|^2+b\re z+c=0$, we assume the arctangent to be equal to $\frac{\pi}{2}$.
We define:
\begin{align*}
    \mathcal{M}_{1-k, \eta,N,D_0}(z):=\chi_{D_0}(Q_{\eta})\sum_{\gamma \in \Gamma_{\eta}\setminus \Gamma_0(N)} (\hat{\varphi}|_{2-2k} \sigma_{\eta}^{-1}\gamma) (z)
\end{align*}
We have the following result:
\begin{proposition}
Let $k>1$ and let $D,D_0$ be two fundamental discriminants with $(-1)^kD,(-1)^kD_0>0$ and such that $DD_0$ is not a perfect square.
Let $\mathcal{A}=[Q_0]$ where $Q_0$ has roots $\eta_1,\eta_2$. Then we have:
\begin{align*}
        \mathcal{M}_{1-k,\eta,N,D_0}=(DD_0)^{\frac{k}{2}}\pi \binom{2k-2}{k-1}\mathcal{F}_{1-k,N,D,D_0,\mathcal{A}}.
    \end{align*}
\end{proposition}
The proof is essentially the same as the proof of lemma 3.2 in \cite{B-K-K}.
\begin{proof-sketch}
The statement is clear if $\chi_{D_0}(Q_0)=0$.
Hence we may assume that $\chi_{D_0}(Q_0)\neq 0$. Let $Q \in \mathcal{Q}_{N,DD_0}$ with root $\eta_Q$. We denote by $\sigma_Q=\sigma_{\eta_{Q}}$.

Unraveling the definitions of $\mathcal{M}_{1-k,\eta, N,D_0}$ we rewrite $\chi_{D_0}(Q_0) \mathcal{M}_{1-k,\eta,N,D_0}$:
\begin{align*}
    \chi&_{D_0}(Q_0) \mathcal{M}_{1-k,\eta,N}(z)\\
    &=\sum_{Q \in [Q_0] }(-1)^{k-1}\left(\frac{Q(z,1)}{(DD_0)^{1/2}}\right)^{k-1}\text{sgn}(\re( \sigma^{-1}_{Q}(z,1)))\varphi\left( \text{arctan}\left(\frac{|\im (\sigma^{-1}_{Q}(z))|}{|\re(\sigma^{-1}_{Q}(z))|}\right)\right)
\end{align*}
Here we used Lemma \ref{lem: B-K-K}.
Moreover note that if $Q=[a,b,c]$ as above then we can rewrite:
\begin{align*}
    \re (\sigma_Q^{-1} z)=-\frac{a|z|^2+bx+c}{|a||-z+\eta_{Q}|^2}\quad \text{ and }\quad \im (\sigma_Q^{-1}z)=\frac{y \sqrt{DD_0}}{|a||-z+\eta_{Q}|^2}.
\end{align*}
Then the claim follows.
\end{proof-sketch}
\subsection{Convergence and values at exceptional points}
In this section we study the convergence of $\mathcal{F}_{1-k,N,D,D_0,\mathcal{A}}$ and the values at points in the exceptional set. 
\begin{proposition}\label{prop: convergence flm}
    Let $k>1$ and $D,D_0$ be two fundamental discriminants with $(-1)^kD,(-1)^kD_0>0$ and such that $DD_0$ is not a perfect square. The function $\mathcal{F}_{1-k,N,D,D_0,\mathcal{A}}$ converges compactly on $\mathbb{H}$, i.e. on every compact subset $\mathcal{U}\subset \mathbb{H}$ $\mathcal{F}_{1-k,N,D,D_0,\mathcal{A}}$ converges uniformly.
\end{proposition}

\begin{proof}
    Recall that $\mathcal{A}=[Q_0]$ are binary quadratic forms $Q \in \mathcal{Q}_{N,DD_0}$ which are equivalent to $Q_0$ with respect to the action of $\Gamma_0(N)$. We denote $\mathcal{A}'$ all matrices in $Q_{DD_0}$ which are equivalent to $Q_0$ with respect to $SL_2(\Z)$. Clearly, we have that $\mathcal{A} \subset \mathcal{A}'$.
    From Proposition 4.1 in \cite{B-K-K} it follows that for $k>1$ the function:
    \begin{align*}
        \mathcal{F}_{1-k,DD_0,\mathcal{A}'}:=\frac{(-1)^k(DD_0)^{\frac{1}{2}-k}}{\binom{2k-2}{k-1}\pi}\sum_{Q \in \mathcal{A}'}\text{sgn}(Q_z)Q(z)^{k-1}\psi\left(\frac{DD_0y^2}{|Q(z,1)|^2}\right)
    \end{align*}
    converges compactly. They show that for each $z \in \mathbb{H}$ there exists a compact neighborhood $\mathcal{U}\subset \mathbb{H}$ in which $\mathcal{F}_{1-k,D,A'}$ converges absolutely.
     Since
\begin{align*}
        |\mathcal{F}_{1-k,N,D,D_0,\mathcal{A}}(z)|&\leq \frac{|(-1)^k(DD_0)^{\frac{1}{2}-k}\chi_{D_0}([A])|}{\binom{2k-2}{k-1}\pi} \times \\
        &\sum_{Q \in \mathcal{A}} \left| \text{sgn}(Q_z)Q(z)^{k-1}\psi\left(\frac{DD_0y^2}{|Q(z,1)|^2}\right)\right|\\
        &\leq \frac{(DD_0)^{\frac{1}{2}-k}}{\binom{2k-2}{k-1}\pi}\sum_{Q \in \mathcal{A}'} \left| \text{sgn}(Q_z)Q(z)^{k-1}\psi\left(\frac{DD_0y^2}{|Q(z,1)|^2}\right)\right|
    \end{align*}
    the result follows.
\end{proof}
Recall that we defined the exceptional set as
\begin{align*}
    E_{N,DD_0}:=\{z \in \mathbb{H}: \exists Q \in \mathcal{Q}_{N,DD_0} \text{with } Q_z=0\}.
\end{align*}
\begin{proposition}\label{prop: values at E_N,DD_0}
Let $k>1$ and $D,D_0$ two fundamental discriminants with $(-1)^kD,(-1)^kD_0>0$ and such that $DD_0$ is not a perfect square. Let $z_0 \in E_{N,DD_0}$. Then
    \begin{align*}
       \mathcal{F}_{1-k,N,D,D_0,\mathcal{A}}(z_0)= \frac{1}{2}\lim_{w \rightarrow 0+}\left( \mathcal{F}_{1-k,N,D,D_0,\mathcal{A}}(z_0+iw) + \mathcal{F}_{1-k,N,D,D_0,\mathcal{A}}(z_0-iw)\right).
    \end{align*}
\end{proposition}
Since the argument is the same as in the proof of Proposition 5.2 in \cite{B-K-K}, we only give a sketch of the proof here. 
\begin{proof-sketch}
By Lemma 5.1 in \cite{B-K-K} the set of binary quadratic form $Q \in Q_{DD_0}$ with $Q_{z_0}=0$ is finite. Hence so is
\begin{align*}
    \mathcal{B}_{z_0,N,DD_0}:=\{Q \in \mathcal{Q}_{N,DD_0}: Q_{z_0}=0\}.
\end{align*}
 Thus there exists a neighborhood $\mathcal{U}$ of $z_0$, such that for all $z \in \mathcal{U}$ for all $Q \in \mathcal{Q}_{N,DD_0}\setminus \mathcal{B}_{z_0,N,DD_0}$ we have that $\text{sgn}(Q_z)=\text{sgn}(Q_{z_0})$. 
Since $Q(z,1)$ and $\psi\left(\frac{DD_0y^2}{|Q(z,1)|^2}\right)$ are continuous for all such $Q$ we have that:
\begin{align*}
    \frac{1}{2}\lim_{w \rightarrow 0^+}\text{sgn}(Q_{z_0+iw})&Q(z_0+iw,1)^{k-1}\psi\left(\frac{DD_0(y_0+w)^2}{|Q(z_0+iw,1)|^2}\right)\\
    &+  \text{sgn}(Q_{z_0-iw})Q(z_0-iw,1)^{k-1}\psi\left(\frac{DD_0(y_0+w)^2}{|Q(z_0-iw,1)|^2}\right)\\
    &= \text{sgn}(Q_{z_0})Q(z_0,1)^{k-1}\psi\left(\frac{DD_0y_0^2}{|Q(z_0,1)|^2}\right).
\end{align*}
If $Q \in \mathcal{B}_{z_0,N,DD_0}$ one can show that $\text{sgn}(Q_{z_0+iw})=-\text{sgn}(Q_{z_0-iw})$. Hence for such $Q$
\begin{align*}
    \frac{1}{2}\lim_{w \rightarrow 0^+}\text{sgn}(Q_{z_0+iw})&Q(z_0+iw,1)^{k-1}\psi\left(\frac{DD_0(y_0+w)^2}{|Q(z_0+iw,1)|^2}\right)\\
    &+  \text{sgn}(Q_{z_0-iw})Q(z_0-iw,1)^{k-1}\psi\left(\frac{DD_0(y_0+w)^2}{|Q(z_0-iw,1)|^2}\right)=0.
\end{align*}
Comparing this to the definition of $\mathcal{F}_{1-k,N,D,D_0, \mathcal{A}}(z)$ (equation \eqref{eq: def F_1-k,N,D,D_0,A}) yields the claim. 
\end{proof-sketch}

\subsection{Action of $\xi_{2-2k}$ and $D^{2k-1}$ and the expansion of $\mathcal{F}_{1-k,N,D,D_0}$}
In this section, we begin by examining the action of $\xi_{2-2k}$ and $D^{2k-1}$ on $\mathcal{F}_{1-k,N,D,D_0, \mathcal{A}}$, which helps us to study its expansion.
We are then finally ready to show that $\mathcal{F}_{1-k,N,D,D_0,\mathcal{A}}$ is indeed a locally harmonic Maass form. 

\begin{theorem}\label{theo: Flm under xi/D}
    Let $k>1$ and let $D,D_0$ be fundamental discriminants such that $(-1)^kD, (-1)^kD_0>0$ and $DD_0$ is not a perfect square. Let $z  \in \mathbb{H}\setminus E_{N,DD_0}$Then
    \begin{align*}
        \xi_{2-2k}(\mathcal{F}_{1-k,N,D,D_0,\mathcal{A}})(z)&=(DD_0)^{\frac{1}{2}-k}f_{k,N,D,D_0, \mathcal{A}},\\
        D^{2k-1}(\mathcal{F}_{1-k,N,D,D_0,\mathcal{A}})(z)&=-(DD_0)^{\frac{1}{2}-k}\frac{(2k-2)!}{(4\pi)^{2k-1}}f_{k,N,D,D_0,\mathcal{A}}.
    \end{align*}
In particular this implies that
\begin{align*}
    \Delta_{2-2k}(\mathcal{F}_{1-k,N,D,D_0,\mathcal{A}})(z)=0.
\end{align*}
\end{theorem}
The statement follows almost directly from the proof of Proposition 6.1 in \cite{B-K-K}.
\begin{proof}
For each $z \in \mathbb{H}\setminus E_{N,DD_0}$ there exists a neighborhood such that $\mathcal{F}_{1-k,N,D,D_0,\mathcal{A}}$ is continuous and real differentiable (see also Lemma 5.1 in \cite{B-K-K}). Moreover we can write $\mathcal{F}_{1-k,N,D,D_0,\mathcal{A}}$ as a Poincare series. Since $\xi_{2k-2}$ and $D^{2k-1}$ commute with the the action of $\Gamma_0(N)$, it suffices to study their action on $\hat{\varphi}$. Note that by Lemma \ref{lem: B-K-K} we may assume that the real part of $z$ is nonzero. 

From the proof of Proposition 6.1 in \cite{B-K-K} it follows that for $z \in \mathbb{H}\setminus E_{N,DD_0}$ with $x \neq 0$:
    \begin{align*}
        \xi_{2-2k}(\hat{\varphi})(z)=z^{-k} \text{ and }
        D^{2k-1} ( \hat{\varphi})(z)=-\frac{(2k-2)!}{(4 \pi)^{2k-1}}z^{-k}.
    \end{align*}
Now the claim follows by comparing the Poincare series of $\mathcal{F}_{1-k,N,D,D_0,\mathcal{A}}$ and $f_{k,N,D,D_0,\mathcal{A}}$. Note that $ \Delta_{2k-2}=-\xi_k\xi_{2-2k}$.
\end{proof}
To study the expansion of $\mathcal{F}_{1-k,N,D,D_0,\mathcal{A}}$ we need the following definitions.

Let $\mathcal{A}$ be an equivalence class in $\mathcal{Q}_{N,DD_0}$. For $a>0, b \in \Z$ we set:
\begin{align*}
    r_{a,b}(\mathcal{A}):=\begin{cases}
        1 + (-1)^k &\text{if }[a,b, \frac{b^2-DD_0}{4a}] \in\mathcal{A} \text{ and }[-a,-b, - \frac{b^2-DD_0}{4a}] \in \mathcal{A},\\
        1  &\text{if }[a,b, \frac{b^2-DD_0}{4a}] \in \mathcal{A} \text{ and }[-a,-b, - \frac{b^2-DD_0}{4a}] \notin \mathcal{A},\\
        (-1)^k &\text{if }[a,b, \frac{b^2-DD_0}{4a}] \notin \mathcal{A} \text{ and }[-a,-b, - \frac{b^2-DD_0}{4a}] \in \mathcal{A},\\
        0 & \text{otherwise}.
    \end{cases}
\end{align*}

Let $Q \in \mathcal{Q}_{N,DD_0}$. For $\epsilon =\pm 1$ we set
\begin{align*}
    \mathcal{C}^{\epsilon}_Q:=\left\{ z \in \mathbb{H}:  \text{sgn}\left( \left| z+\frac{b}{2a}\right| - \frac{(DD_0)^{1/2}}{2|a|}\right)=\epsilon\right\}.
\end{align*}
A short computation reveals that $z \in \mathcal{C}^{\epsilon}_Q$ is equivalent to $\text{sgn}(Q_z)=\epsilon \text{sgn}(a)$. Let $\mathcal{C}\subset \mathbb{H}\setminus E_{N,DD_0}$ be a connected component. We define:
\begin{align*}
    \mathcal{B}_{\mathcal{C},\mathcal{A}}:=\{ Q \in \mathcal{A}: z \in \mathcal{C}^-_Q, \text{ for all }z \in \mathcal{C}\}.
\end{align*}
In other words the set $\mathcal{B}_{\mathcal{C},\mathcal{A}}$ consists of those binary quadratic forms such that $\text{sgn}(Q_z)=-\text{sgn}(a)$ for all $z \in \mathcal{C}$.
Finally we set:
\begin{align*}
         c_{\infty}(\mathcal{A})=\frac{-\chi_{D_0}([\mathcal{A}])}{2^{2k-2}(2k-1)\binom{2k-2}{k-1}}\sum_{\substack{a \in \N, \\N |a}} a^{-k}\sum_{\substack{b \mod 2a,\\b^2\equiv DD_0 \mod 4a}}r_{a,b}(\mathcal{A}).
     \end{align*}

We have the following theorem:
\begin{theorem}\label{theo: splitting Flm}
Let $k>1$ and let $D,D_0$ be two fundamental discriminants such that $(-1)^kD, (-1)^kD_0>0$ and such that $DD_0$ is not a perfect square.
Let $\mathcal{C} \subset \mathbb{H} \setminus E_{N,DD_0}$ be a connected component not intersecting the exceptional set. For every $z \in \mathcal{C}$, the function $\mathcal{F}_{1-k,N,D,D_0,\mathcal{A}}$ has the following expansion:
    \begin{align*}
        \mathcal{F}_{1-k,N,D,D_0,\mathcal{A}}(z)=&(DD_0)^{\frac{1}{2}-k}f^*_{k,N,D,D_0,\mathcal{A}}(z)\\
        &-(DD_0)^{\frac{1}{2}-k}\frac{(2k-2)!}{(4 \pi)^{2k-1}}\mathcal{E}_{f_{k,N,D,D_0,\mathcal{A}}}(z)
        +P_{\mathcal{C}, \mathcal{A}}(z).
    \end{align*}
    Here $P_{\mathcal{C}, \mathcal{A}}(z) \in \C[z]$ is the polynomial given by:
    \begin{align*}
        P_{\mathcal{C}, \mathcal{A}}(z)=c_{\infty}(\mathcal{A})-(-1)^k2^{2-2k}(DD_0)^{\frac{1}{2}-k}\chi_{D_0}([A])\sum_{Q \in \mathcal{B}_{\mathcal{C,\mathcal{A}}}}\text{sgn}(a)Q(z,1)^{k-1}.
    \end{align*}
\end{theorem}

This proof follows closely the proof of Theorem 7.1 in \cite{B-K-K}. We therefore only sketch it.
\begin{proof-sketch}
     Using Theorem \ref{theo: Flm under xi/D} and the definition of the (non)holomorphic Eichler integral, one can show that for every $z \in \mathbb{H} \setminus E_{N,DD_0}$, we have that:
     \begin{align*}
         P_{C,A}(z):=\mathcal{F}_{1-k,N,D,D_0,\mathcal{A}}&-(DD_0)^{\frac{1}{2}-k}f^*_{k,N,D,D_0,\mathcal{A}}(z)\\
         &+(DD_0)^{\frac{1}{2}-k}\frac{(2k-2)!}{(4\pi)^{2k-1}}\mathcal{E}_{f_{k,N,D,D_0,\mathcal{A}}}(z)
     \end{align*}
     is a polynomial of degree at most $2-2k$.
     To prove the explicit formula for the local polynomial we proceed by induction on the cardinality of the set $\mathcal{B}_{\mathcal{C},A}$: \\
    If $z=x+iy \in \mathcal{C}$ with $y> \frac{\sqrt{DD_0}}{2|N|}$ then $\mathcal{B}_{\mathcal{C},A}$ is the empty set. We thus want to show that the local polynomial is equal to the constant at infinity $c_{\infty}(\mathcal{A})$.
    
    Given a binary quadratic form $[a,b,c]=[a,b,\frac{b^2-DD_0}{4a}]$, we can write b as a residue modulo $2a$. This yields that for $y> \frac{\sqrt{DD_0}}{2|N|}$
     \begin{align*}
         \mathcal{F}_{1-k,N,D,D_0,\mathcal{A}}(z)= &\frac{(-1)^k(DD_0)^{\frac{1}{2}-k}\chi_{D_0}([\mathcal{A}])}{\binom{2k-2}{k-1}\pi}\\
         &\times \sum_{\substack{a \in \N, \\N|a}} \sum_{\substack{b \mod 2a\\b^2\equiv DD_0 \mod 4a\\Q=[a,b \frac{b^2-DD_0}{4a}]}}r_{a,b}(\mathcal{A})
          \sum_{n \in Z}Q(z+n,1)^{k-1}
         \psi\left(\frac{DD_0y^2}{|Q(z,1)|^2}\right)
     \end{align*}
     Using Poisson summation and Lemma 7.3 in \cite{B-K-K}, we find that the constant term corresponds to 
     \begin{align*}
         c_{\infty}(\mathcal{A})=\frac{-\chi_{D_0}([\mathcal{A}])}{2^{2k-2}(2k-1)\binom{2k-2}{k-1}}\sum_{\substack{a \in \N, \\N |a}} a^{-k}\sum_{\substack{b \mod 2a,\\b^2\equiv DD_0 \mod 4a}}r_{a,b}(\mathcal{A})
     \end{align*}

     For the induction step one shows that given a component $\mathcal{C}$ with $\# \mathcal{B}_{\mathcal{C}, \mathcal{A}}=n$, then there exists a neighboring component $\mathcal{C}_1$, only separated from $\mathcal{C}$ by one geodesic $S_{Q_0}$, and  such that $\#\mathcal{B}_{\mathcal{C}_1, \mathcal{A}}<n$. Hence the statement is true for $\mathcal{C}_1$. The difference between the two local polynomials $P_{\mathcal{C}, \mathcal{A}}$ and $P_{\mathcal{C}_1, \mathcal{A}}$ is determined by the binary quadratic forms associated to the geodesic $S_{Q_0}$. Two cases are possible: 1) Only the binary quadratic form $Q_0$ corresponds to $S_{Q_0}$; 2) If $ -Q_0 \in \mathcal{A}$, then both $Q_0$ and $-Q_0$ correspond to $S_{Q_0}$. The claim follows by computing explicitly 
     \begin{align*}
         \lim_{w \rightarrow 0^+}\mathcal{F}_{1-k,N,D,D_0,\mathcal{A}}(z-iw)-\mathcal{F}_{1-k,N,D,D_0,\mathcal{A}}(z+iw)
     \end{align*}
for $z$ on $S_{Q_0}$.
\end{proof-sketch}
 Apriori we do not know if both $Q$ and $-Q$ are in the same equivalence class, which makes it difficult to compute $P_{\mathcal{C},\mathcal{A}}$ explicitly. However taking the sum over all equivalence classes, we find an unexpectedly elegant expression. Set
\begin{align}\label{eq:app def c_infty}
    c_{k,N,D,D_0,\infty}: =\frac{-2^{3-2k}}{(2k-1)\binom{2k-2}{k-1}}\sum_{\substack{a \in \N, \\N |a}} a^{-k}\sum_{\substack{b \mod 2a,\\b^2\equiv DD_0 \mod 4a}}\chi_{D_0}\left(\left[a,b\frac{b^2-DD_0}{4a}\right]\right)
\end{align}
We have that
\begin{align*}
    \mathcal{F}_{1-k,N,D,D_0}=\sum_{\mathcal{A}}\mathcal{F}_{1-k,N,D,D_0, \mathcal{A}}
\end{align*}
where we sum over all equivalence classes $\mathcal{A}$ of binary quadratic forms. We have the following:
\begin{theorem}\label{theo: decomposition F}
    Let $k>1$ and let $D,D_0$ be two fundamental discriminant such that $(-1)^kD,(-1)^kD_0>0$ and such that $DD_0$ is not a perfect square.
    \begin{align*}
    \mathcal{F}_{1-k,N,D,D_0}(z)
        =&(DD_0)^{\frac{1}{2}-k}f^*_{k,N,D,D_0}(z)\\
        &-(DD_0)^{\frac{1}{2}-k}\frac{(2k-2)!}{(4 \pi)^{2k-1}}\mathcal{E}_{f_{k,N,D,D_0}}(z)
        +\mathcal{P}_{k,N,D,D_0}(z).
    \end{align*}
    Where
    \begin{align*}
        \mathcal{P}_{k,N,D,D_0}(z)=
        c_{k,N,D,D_0,\infty}+ (-1)^k 2^{3-2k}(DD_0)^{\frac{1}{2}-k}\sum_{\substack{Q \in \mathcal{Q}_{N,DD_0},\\a<0<Q_z}} \chi_{D_0}(Q)Q(z,1)^{k-1}.
    \end{align*}
\end{theorem}

\begin{proof}
We have that 
    \begin{align*}
    \mathcal{F}_{1-k,N,D,D_0}(z)=&\sum_{\mathcal{A}}
        \mathcal{F}_{1-k,N,D,D_0,\mathcal{A}}(z)\\
        =&(DD_0)^{\frac{1}{2}-k}\sum_{\mathcal{A}}f^*_{k,N,D,D_0,\mathcal{A}}(z)\\
        &-(DD_0)^{\frac{1}{2}-k}\frac{(2k-2)!}{(4 \pi)^{2k-1}}\sum_{\mathcal{A}}\mathcal{E}_{f_{k,N,D,D_0,\mathcal{A}}}(z)
        +\sum_{\mathcal{A}}P_{\mathcal{C}, \mathcal{A}}(z).
    \end{align*}
    Since $\sum_{\mathcal{A}}f_{k,N,D,D_0,\mathcal{A}}=f_{k,N,D,D_0}$ it follows that 
    \begin{align*}
        \sum_{\mathcal{A}}f^*_{k,N,D,D_0,\mathcal{A}}(z)=f^*_{k,N,D,D_0}(z), \text{ and }\sum_{\mathcal{A}}\mathcal{E}_{f_{k,N,D,D_0,\mathcal{A}}}(z)=\mathcal{E}_{f_{k,N,D,D_0}}(z).
    \end{align*}

  Before studying the local polynomial, note that
\begin{align*}
    \chi_{D_0}([a,b,c])=\left( \frac{D_0}{-1}\right) \chi_{D_0}([-a,-b,-c])=\begin{cases}
        \chi_{D_0}([a,b,c]) & \text{if }D_0>0, \\
        -\chi_{D_0}([a,b,c]) & \text{if }D_0<0.
    \end{cases}
\end{align*}
Recall that we assume $(-1)^kD_0>0$. Thus $D_0>0$ implies that $k$ is even and $D_0<0$ implies that $k$ is odd.

We first study the constant at infinity:
    \begin{align*}
         \sum_{\mathcal{A}}c_{\infty}(A)=\frac{-1}{2^{2k-2}(2k-1)\binom{2k-2}{k-1}}\sum_{\mathcal{A}}\chi_{D_0}([\mathcal{A}])\sum_{\substack{a \in \N, \\N |a}} a^{-k}\sum_{\substack{b \mod 2a,\\b^2\equiv DD_0 \mod 4a}}r_{a,b}(A)
     \end{align*}
To write this sum without sets of equivalence classes $\mathcal{A}$ we focus on pairs of binary quadratic forms $Q$ and $-Q$. We distinguish depending on the parity of $k$.\\
If $k$ is even then from the definition of $r_{a,b}(\mathcal{A})$ one sees that independently whether $Q$ and $-Q$ are in the same equivalence class or not, they both contribute $(-1)^k\chi_{D_0}([\mathcal{A}])=\chi_{D_0}(Q)$ to the sum. Hence we may rewrite:
\begin{align}\label{eq: cst at infty sum}
    \sum_{\mathcal{A}}\chi_{D_0}([A])\sum_{\substack{a \in \N, \\N |a}} a^{-k}&\sum_{\substack{b \mod 2a,\\b^2\equiv DD_0 \mod 4a}}r_{a,b}(A)\\
    &=2\sum_{\substack{a \in \N, \\N |a}} a^{-k}\sum_{\substack{b \mod 2a,\\b^2\equiv DD_0 \mod 4a}}\chi_{D_0}\left(\left[a,b,\frac{b^2-DD_0}{4a}\right]\right).
\end{align}
If $k$ is odd, then since the generalized genus character is invariant under $SL_2(\Z)$, if $Q$ and $-Q$ are in the same equivalence class, then $\chi_{D_0}(Q)=0$. Otherwise their added contribution is:
\begin{align*}
    \chi_{D_0}(Q)+(-1)^k\chi_{D_0}(-Q))=2\chi_{D_0}(Q).
\end{align*}
Hence also in this case we recover the equation \eqref{eq: cst at infty sum}.

We use a similar trick to rewrite the nonconstant part of the local polynomial: 
\begin{align*}
    \sum_{\mathcal{A}}\chi_{D_0}{([A])}\sum_{Q \in \mathcal{B}_{\mathcal{C},\mathcal{A}}} \text{sgn}(a)Q(z,1)^{k-1}= \sum_{\mathcal{A}}\sum_{Q \in \mathcal{B}_{\mathcal{C},\mathcal{A}}} \text{sgn}(a)\chi_{D_0}{(Q)}Q(z,1)^{k-1}
\end{align*}
Recall that $Q \in \mathcal{B}_{\mathcal{C},\mathcal{A}}$ if for all $ z \in \mathcal{C}$ we have that:
\begin{align*}
    \text{sgn}(Q_z)=-\text{sgn}(a).
\end{align*}
Thus if $Q \in \mathcal{B}_{\mathcal{C},\mathcal{A}}$ for some equivalence class $\mathcal{A}$, then there exists an equivalence class $\mathcal{A}'$ (which might be $\mathcal{A}$) with $ -Q \in \mathcal{B}_{\mathcal{C},\mathcal{A}'}$. Hence we can study again the contribution of pairs $Q$ and $-Q$. Let $Q=[a,b,c]$ with $a<0$, then we get:
\begin{align*}
    \left(\text{sgn}(a)\chi_{D_0}{(Q)}Q(z,1)^{k-1}\right)&+\left(\text{sgn}(-a)\chi_{D_0}(-Q)(-Q)(z,1)^{k-1}\right)\\
    &= \chi_{D_0}(Q)Q(z,1)^{k-1}\left((-1)+(-1)^{k-1}\left(\frac{D_0}{-1}\right)\right)\\
    &= -2 \chi_{D_0}(Q)Q(z,1)^{k-1}
\end{align*}
Here we used that by the same argument as above $(-1)^{k-1}\left(\frac{D_0}{-1}\right)=-1$ if $\chi_{D_0}(Q)\neq 0$. This proves the claim.
\end{proof}
We are now ready to prove that $\mathcal{F}_{1-k,N,D,D_0, \mathcal{A}}$ and $\mathcal{F}_{1-k,N,D,D_0}$ are locally harmonic Maass forms:

\begin{theorem}\label{theo: Flm is locally harmonic}
Let $k>1$ and let $D,D_0$ be two fundamental discriminants such that $(-1)^kD, (-1)^kD_0>0$ and $DD_0$ is not a perfect square. Moreover let $\mathcal{A}$ be any equivalence class of binary quadratic forms. The functions $\mathcal{F}_{1-k,N,D,D_0,\mathcal{A}}$ and $\Flm$ are locally harmonic Maass form with exceptional set $E_{N, DD_0}$.
\end{theorem}
The proof is the same as the proof of Theorem 7.4 in \cite{B-K-K}.
\begin{proof}
We check that $\mathcal{F}_{1-k,N,D,D_0, \mathcal{A}}$ satisfies definition \ref{def: lhmF -app1}.
Writing $\mathcal{F}_{1-k,N,D,D_0,\mathcal{A}}$ in terms of Poincare series and using the absolute convergence proven in Proposition \ref{prop: convergence flm} yields that $\mathcal{F}_{1-k,N,D,D_0,\mathcal{A}}$ is modular of weight $2-2k$. Properties $(ii)$ and $(iii)$ from definition \ref{def: lhmF -app1} is proven by Proposition \ref{prop: values at E_N,DD_0} and Theorem \ref{theo: Flm under xi/D}. Finally the bounded growth follows from the expansion of $\mathcal{F}_{1-k,N,D,D_0, \mathcal{A}}$ proven in Theorem \ref{theo: splitting Flm} and the explicit description of the local polynomial. This proves the claim for $\mathcal{F}_{1-k,N,D,D_0, \mathcal{A}}$. Since $\mathcal{F}_{1-k,N,D,D_0}=\sum_{\mathcal{A}}\mathcal{F}_{1-k,N,D,D_0, \mathcal{A}}$, the statement holds for $\mathcal{F}_{1-k,N,D,D_0}$.
\end{proof}

\subsection{Hecke operators}
In this section we study the action of the Hecke operators on $f_{k,N,D,D_0}$ and $\Flm$. They are crucial in both the statement and the proof of Theorem \ref{theo: main prod L-functions} for higher-dimensional spaces of cusp forms. We show that when $(D_0,p)=1$, we can explicitly describe the action of the Hecke operators. Moreover, we study the behavior of $\Flm$ under the Atkin-Lehner involution. 

We start by recalling that for any translation invariant function $f$ we can define the Hecke operators $T_p$ of weight $k$ :
\begin{align*}
    f|_{k}T_p:=p^{k-1}f(pz)+p^{-1}\sum_{r \bmod p}f\left( \frac{z+r}{p}\right).
\end{align*}
We start by studying the action of $T_p$ on $f_{k,N,D,D_0}$.
\begin{theorem}
Let $k>1$ and let $D,D_0$ be two fundamental discriminants such that $(-1)^kD, (-1)^kD_0>0$ and $DD_0$ is not a perfect square. 
    Let $p$ be a prime such that $(p,N)=(p,D_0)=1$. Then
    \begin{align*}
        f_{k,N,D,D_0}\mid_{2k} T_p&=
        f_{k,N,Dp^2,D_0}+\left( \frac{D}{p}\right)p^{k-1}f_{k,N,D,D_0}+p^{2k-1}f_{k,N,\frac{D}{p^2},D_0}
    \end{align*}
    Here $f_{k,N,\frac{D}{p^2},D_0}=0$ if $p^2 \nmid D$.
\end{theorem}
In \cite{bengoechea-irreducibilitygaloisgrouphecke}, Bengoechea proves the same statement for $N=1$. If $N>1$, one can follow exactly the same argumentation. We therefore only sketch the proof.
\begin{proof-sketch}
Using the definition, we find that
\begin{align*}
    f_{k,N,D,D_0}\mid_{2k} T_p&=\frac{(DD_0p)^{k-\frac{1}{2}}}{\binom{2k-2}{k-1}\pi}\sum_{[a,b,c]\in \mathcal{Q}_{N,DD_0}} \left( \frac{\chi_{D_0}([a,b,c])}{([a,bp,cp^2](z,1))^k }\right.\\
    & \left. +\sum_{r \bmod p} \frac{\chi_{D_0}([a,b,c])}{([a,bp+2ar,cp^2+bpr+ar^2](z,1))^k}\right)
\end{align*}
    Since $(p,D_0)=1$, one can show that $\chi_{D_0}([a,b,c])=\chi_{D_0}([ap^2,bp,c])$ and that $\chi_{D_0}([a,b,c])=\chi_{D_0}([a,bp,cp^2])=\chi_{D_0}([a,bp+2ar,cp^2+bpr+ar^2])$.
    Then we can rewrite the sum:
    \begin{align*}
        f_{k,N,D,D_0}\mid_{2k} T_p=\frac{(DD_0p)^{k-\frac{1}{2}}}{\binom{2k-2}{k-1}\pi}\sum_{[a,b,c]\in Q_{N,DD_0p^2}} n(a,b,c)([a,b,c](z,1))^{-k}.
    \end{align*}
    Here
    \begin{align*}
        n(a,b,c):&=\chi_{D_0}([a,b,c])\left( \epsilon\left(\frac{a}{p^2}, \frac{b}{p},c\right)+ \sum_{r \bmod p }\epsilon\left(a, \frac{b-2aj}{p},\frac{c-bj+aj^2}{p^2}\right)\right)\\
    \end{align*}
Where we assume that $ \epsilon(a,b,c)=1$ if $a,b,c$ are integers and $0$ otherwise.

    Let $[a,b,c] \in Q_{DD_0p^2}$ with $p^m|a,b,c$ for some $ m \in \{1,2\}$. Then since $(N,p)=1$, we have that $N|a \Leftrightarrow N |(a/p^m)$. Now the result follows by rewriting $n(a,b,c)$ as in \cite{Zagier-Eisenstein-Riemann-Zeta} (p. 291/293) and noting that if $p^m|a,b,c$ for $m \in \{1,2\}$ then
    \begin{align*}
        \chi_{D_0}([a,b,c])=\left(\frac{D_0}{p^m}\right)\chi_{D_0}([a/p,b/p,c/p]).
    \end{align*}
\end{proof-sketch}
We now study the action of the Hecke operator of $\Flm$. We prove the following:
\begin{theorem}\label{theo: Flm Hecke}
Let $k>1$ and $D,D_0$ be two fundamental discriminants such that $(-1)^kD,(-1)^kD_0>0$ and $DD_0$ is not a perfect square. 
Let $p$ be a prime such that $(p,N)=(p,D_0)=1$. Then
\begin{align*}
    \mathcal{F}_{1-k,N,D,D_0}|_{2-2k} T_p=
    \mathcal{F}_{1-k,N,Dp^2,D_0}(z)&+p^{-k}\left( \frac{D}{p}\right)\mathcal{F}_{1-k,N,D,D_0}\\
    &+p^{1-2k}\mathcal{F}_{1-k,N,\frac{D}{p^2},D_0}.
\end{align*}
Here if $p^2 \nmid D$, then $\mathcal{F}_{1-k,N,\frac{D}{p^2},D_0}=0$.
\end{theorem}
We show this statement in the same manner as in the previous proof. Note that this argumentation is different than the approach in \cite{B-K-K}.
\begin{proof}
It follows from the definition that:
\begin{align*}
    \mathcal{F}&_{1-k,N,D,D_0}|_{2-2k}T_p(z)\\
    &=p^{1-2k}\mathcal{F}_{1-k,N,D,D_0}(pz)+p^{-1}\sum_{r \bmod p}\mathcal{F}_{1-k,N,D,D_0}\left( \frac{z+r}{p}\right)\\
    &= \frac{(DD_0)^{\frac{1}{2}-k}}{\binom{2k-2}{k-1}\pi}\left( \sum_{Q \in \mathcal{Q}_{N,DD_0}}p^{1-2k}\chi_{D_0}(Q)\text{sgn}(Q_{pz})Q(pz,1)^{k-1}\psi\left( \frac{DD_0p^2y^2}{|Q(pz,1)|^2}\right)\right.\\
    &\left.+ p^{-1}\sum_{r \bmod p }\sum_{Q \in \mathcal{Q}_{N,DD_0}}\chi_{D_0}(Q)\text{sgn}(Q_{(z+r)/p})Q((z+r/p),1)^{k-1}\psi\left( \frac{DD_0(y/p)^2}{|Q((z+r)/p,1)|^2}\right)\right)\\
\end{align*}
Note that
\begin{align*}
    [a,b,c](pz,1)=[ap^2,bp,c](z,1).
\end{align*}
Moreover we have that
\begin{align*}
    p^2\left([a,b,c]\left(\frac{z+r}{p},1\right)\right)&=a(z+r)^2+bp(z+r)+cp^2\\
    &=[a,bp+2ar,cp^2+bpr+ar^2](z,1).
\end{align*}
Since we assume that $(D_0,p)=1$ it follows that:
\begin{align*}
    \chi_{D_0}([a,b,c])=\chi_{D_0}([ap^2,bp,c])=\chi_{D_0}([a,bp+2ar,cp^2+bpr+ar^2]).
\end{align*}
Therefore, we can rewrite $\mathcal{F}_{1-k,N,D,D_0}|_{2-2k}T_p(z)$ as 
\begin{align*}
    \mathcal{F}&_{1-k,N,D,D_0}|_{2-2k}T_p(z)\\
    &= \frac{(DD_0)^{\frac{1}{2}-k}}{\binom{2k-2}{k-1}\pi}  \sum_{[a,b,c] \in \mathcal{Q}_{N,DD_0}}\left( p^{1-2k}\chi_{D_0}([ap^2,bp,c])\text{sgn}([ap^2,bp,c]_{z}) \right.\\
    &\times\left([ap^2,bp,c](z,1)\right)^{k-1}\psi\left( \frac{DD_0p^2y^2}{|[ap^2,bp,c](z,1)|^2}\right)\\
    &+ \sum_{r \mod p }p^{1-2k} \chi_{D_0}([a,bp+2ar,cp^2+bpr+ar^2])\text{sgn}([a,bp+2ar,cp^2+bpr+ar^2]_z)\\
    &\times \left([a,bp+2ar,cp^2+bpr+ar^2](z,1)\right)^{k-1}\psi\left( \frac{DD_0(y/p)^2}{p^{-4}|[a,bp+2ar,cp^2+bpr+ar^2](z,1)|^2}\right)\\
    &= \frac{(DD_0p^2)^{\frac{1}{2}-k}}{\binom{2k-2}{k-1}\pi} \sum_{Q=[a,b,c] \in \mathcal{Q}_{N,DD_0p^2}}\left( \text{sgn}(Q_z)Q(z,1)^{k-1}  \psi\left( \frac{DD_0p^2y}{|Q(z,1)|^2}\right) \times n(a,b,c)\right).
\end{align*}
Here we set again:
\begin{align*}
        n(a,b,c):&=\chi_{D_0}([a,b,c])\left( \epsilon\left(\frac{a}{p^2}, \frac{b}{p},c\right)+ \sum_{r \mod p }\epsilon\left(a, \frac{b-2aj}{p},\frac{c-bj+aj^2}{p^2}\right)\right)\\
        &=\chi_{D_0}([a,b,c])\left( 1 +\left(\frac{DD_0}{p}\right)\epsilon\left(\frac{a}{p}, \frac{b}{p},\frac{c}{p}\right) +p  \epsilon\left(\frac{a}{p^2}, \frac{b}{p^2},\frac{c}{p^2}\right)\right) 
    \end{align*}
The second equation follows from \cite{Zagier-Eisenstein-Riemann-Zeta} (p. 291/292). We have that:
\begin{align*}
    \mathcal{F}&_{1-k,N,Dp^2,D_0}\\&=\frac{(DD_0p^2)^{\frac{1}{2}-k}}{\binom{2k-2}{k-1}\pi} 
   \sum_{Q=[a,b,c] \in \mathcal{Q}_{N,DD_0p^2}} \chi_{D_0}([a,b,c])\text{sgn}(Q_z)Q(z,1)^{k-1} \psi\left( \frac{DD_0p^2y}{|Q(z,1)|^2}\right). 
\end{align*}
 Let $m \in \{ 1,2\}$. If  $p^m|a,b,c$ then for any $x,y \in \Z$ we have: $ax^2+bxy+cy^2=p^m\left(ax^2/p^m+bxy/p^m+cy^2/p^m\right)$ it follows that
\begin{align*}
    \chi_{D_0}([a,b,c])=\left(\frac{D_0}{p^m}\right) \chi_{D_0}[a/p^m,b/p^m,c/p^m].
\end{align*}
Thus 
\begin{align*}
    \frac{(DD_0p^2)^{\frac{1}{2}-k}}{\binom{2k-2}{k-1}\pi} \sum_{Q=[a,b,c] \in \mathcal{Q}_{N,DD_0p^2}} &\left( \frac{DD_0}{p}\right)\epsilon \left(\frac{a}{p}, \frac{b}{p},\frac{c}{p}
    \right) \chi_{D_0}([a,b,c])\\
    &\times \text{sgn}(Q_z)Q(z,1)^{k-1} \psi\left( \frac{DD_0p^2y}{|Q(z,1)|^2}\right)\\
    &= \frac{(DD_0p^2)^{\frac{1}{2}-k}}{\binom{2k-2}{k-1}\pi} \sum_{Q=[a,b,c] \in \mathcal{Q}_{N,DD_0}}\left( \frac{D}{p}\right)\chi_{D_0}([a,b,c])\\& \times \text{sgn}(Q_z)p^{k-1}Q(z,1)^{k-1} \psi\left( \frac{DD_0y}{|Q(z,1)|^2}\right)\\
    &=p^{-k}\left( \frac{D}{p}\right)\mathcal{F}_{1-k,N,D,D_0}.
\end{align*}
Finally for the last term note that:
\begin{align*}
    \frac{(DD_0p^2)^{\frac{1}{2}-k}}{\binom{2k-2}{k-1}\pi} &\sum_{Q=[a,b,c] \in \mathcal{Q}_{N,DD_0p^2}}p\epsilon \left(\frac{a}{p^2}, \frac{b}{p^2},\frac{c}{p^2}
    \right) \chi_{D_0}([a,b,c])\\
    & \times \text{sgn}(Q_z)Q(z,1)^{k-1} \psi\left( \frac{DD_0p^2y}{|Q(z,1)|^2}\right)\\
    &= \frac{(DD_0p^2)^{\frac{1}{2}-k}p}{\binom{2k-2}{k-1}\pi} \sum_{Q=[a,b,c] \in \mathcal{Q}_{N,DD_0/p^2}}\chi_{D_0}([a,b,c])\\
    & \times \text{sgn}(Q_z)p^{2k-2}Q(z,1)^{k-1} \psi\left( \frac{(DD_0/p^2)y}{|Q(z,1)|^2}\right)\\
    &=p^{1-2k}\mathcal{F}_{1-k,N,D/p^2,D_0}
\end{align*}
This proves the claim.
\end{proof}
We finish this section by determining the behavior of $\Flm$ under the Fricke involution. 
\begin{lemma}\label{lem: Flm Fricke involution}
We have:
\begin{align*}
\Flm\Big\vert_{2-2k}\left(W_N - \left(\frac{D_0}{N}\right) \text{ id} \right) = 0.
\end{align*}
\end{lemma}

\begin{remark}
    If $ \left( \frac{D_0}{N}\right)=1$, then we recover  Lemma 3.3. \cite{males-mono-rollen-wagner}.
\end{remark}
The proof follows the one of Lemma 3.3 in \cite{males-mono-rollen-wagner}:
\begin{proof}
Using that:
\begin{align*}
   [a,b,c]|W_N (z)&= \frac{1}{N}([a,b,c]\circ W_N)(z,1)\\
    &=\frac{1}{N}(a-bNz+cN^2z^2)\\
    &=Nz^2[a,b,c](-1/(Nz),1)
\end{align*}
we get:
\begin{align*}
\frac{DD_0\im\left(W_N(z)\right)^2}{\left\vert Q\left(W_N(z),1\right) \right\vert^2}
&=\frac{DD_0\im\left(-\frac{1}{Nz}\right)^2}{\left\vert Q\left(-\frac{1}{Nz},1\right) \right\vert^2}\\
&= \frac{DD_0(N|z|^2)^{-2}\im(z)^2}{(N|z|^2)^{-2}\left\vert \left(Q |W_N\right)(z,1) \right\vert^2}
= \frac{DD_0\im\left(z\right)^2}{\left\vert \left(Q | W_N\right)(z,1) \right\vert^2}.
\end{align*}
Moreover,
\begin{align*}
[a,b,c]_{W_Nz}
&= \frac{1}{\im\left(-\frac{1}{Nz}\right)}\left(a\left\vert -\frac{1}{Nz}\right\vert^2+b\re\left(-\frac{1}{Nz}\right)+c\right) \\
&= \frac{N|z|^2}{\im(z)}\left(\frac{a}{N^2|z|^2}-b\frac{\re(z)}{N|z|^2}+c\right) \\
&= \left[cN,-b,\frac{a}{N},\right]_{z}
= \left([a,b,c] | W_N\right)_{z}.
\end{align*}
In particular, we have that $ sgn( [a,b,c]_{W_Nz})=sgn(\left([a,b,c] |W_N \right)_{z}).$

We compute $\mathcal{F}_{1-k,N,D,D_0}(z)|_{2k-2}W_N$:
\begin{align*}
    &\frac{\binom{2k-2}{k-1}2\pi}{(DD_0)^{\frac{1}{2}-k}}\Flm|_{2-2k}W_N (z)\\
    &= (Nz)^{2k-2}\sum_{Q \in \mathcal{Q}_{N,DD_0}} \chi_{D_0}(Q)sgn(Q_{W_N(z)})Q(W_N(z),1)^{k-1}\beta\left(\frac{DD_0\im((W_N(z)))^2}{|Q(W_N(z),1)|^2};k-\frac{1}{2}, \frac{1}{2}\right)\\
    &=(Nz)^{2k-2}\sum_{Q \in \mathcal{Q}_{N,DD_0}} \chi_{D_0}(Q) sgn((Q|W_N)_{z}) (Q|W_N) (z,1)^{k-1} \left( Nz \right)^{2-2k}\\&\beta\left(\frac{DD_0\im(z)^2}{|(Q| W_N)(z,1)|^2};k-\frac{1}{2}, \frac{1}{2}\right)\\
    & = \sum_{Q \in \mathcal{Q}_{N,DD_0}} \left(\frac{D_0}{N}\right)\chi_{D_0}(Q) sgn(Q_{z}) Q (z,1)^{k-1} \beta\left(\frac{DD_0\im(z)^2}{|Q(z,1)|^2};k-\frac{1}{2}, \frac{1}{2}\right)\\
     &= \left( \frac{D_0}{N}\right) \frac{\binom{2k-2}{k-1}2\pi}{(DD_0)^{\frac{1}{2}-k}} \Flm|_{2-2k} \text{ id} (z)
\end{align*}
Here we replaced $Q$ with $Q|W_N$ and used Lemma \ref{lem: chi(W_N Q)}.
\end{proof}

This finishes our study of the function $\mathcal{F}_{1-k,N,D,D_0}$.

\bibliographystyle{plain}
\bibliography{References}

\end{document}